\documentclass[12pt]{article}
\usepackage{amssymb,amsmath,setspace,graphicx,multicol,wrapfig,amsfonts,amsthm,titling,url,array,parskip,enumitem,bbm,color,tikz,tikz-cd,caption,blkarray}
\usepackage{hyperref}
\usepackage{cases,empheq}
\usepackage{mathtools}
\usetikzlibrary{calc}
\usepackage{placeins,comment}
\usepackage[capitalize,nameinlink,noabbrev,nosort]{cleveref}
\usepackage{aurical,pbsi}
\usepackage[T1]{fontenc}
\usepackage[hmargin=1.25cm,vmargin=2.5cm]{geometry}
\usepackage{array, makecell,calculator,young}
\usepackage{subcaption}
\usepackage{float}
\usepackage[title]{appendix}
\usetikzlibrary{decorations.markings, arrows.meta}

\usepackage[makeroom]{cancel}
\usepackage[backend=bibtex,natbib=true,giveninits,doi=false,url=false,isbn=false,eprint=false,maxbibnames=99]{biblatex}
\usepackage{nicematrix,multicol}
\renewbibmacro{in:}{}
\usepackage[dvipsnames]{xcolor}

\graphicspath{{Figures/}}

\theoremstyle{plain}
\newtheorem{theorem}{Theorem}[section]
\newtheorem{prop}[theorem]{Proposition}
\newtheorem{lemma}[theorem]{Lemma}
\newtheorem{corollary}[theorem]{Corollary}

\theoremstyle{definition}
\newtheorem{definition}[theorem]{Definition}
\newtheorem{defn}[theorem]{Definition}
\newtheorem{example}[theorem]{Example}
\newtheorem{question}[theorem]{Question}
\newtheorem{remark}[theorem]{Remark}
\crefname{defn}{Definition}{Definitions}
\crefname{theorem}{Theorem}{Theorems}
\crefname{lemma}{Lemma}{Lemmas}

\newcommand{\tcr}[1]{\textcolor{red}{#1}}

\newcommand*\cir[0]{%
  \begin{tikzpicture}
    \draw (0,0) circle (4pt) ;
  \end{tikzpicture}}

\newcommand*\squ[0]{%
  \begin{tikzpicture}
    \draw (0,0) rectangle (7.5pt,7.5pt);
  \end{tikzpicture}}

\newcommand*\squa[1]{%
  \begin{tikzpicture}
    \draw (0,0) rectangle (#1,#1);
  \end{tikzpicture}}

\newcommand*\tsqua[1]{%
  \begin{tikzpicture}
    \draw[thick] (0,0) rectangle (#1,#1);
  \end{tikzpicture}}

\newcommand{\Cdot}{\mathbin{\makebox[8pt]{$\cdot$}}}

\newcommand{\sqcir}{\begin{pmatrix}
               \squ\\\cir
                \end{pmatrix}}
\newcommand{\circir}{\begin{pmatrix}
            \cir\\\cir
            \end{pmatrix}}
\newcommand{\cirsq}{\begin{pmatrix}
            \cir\\\squ
            \end{pmatrix}}
\newcommand{\sqsq}{\begin{pmatrix}
            \squ\\\squ
            \end{pmatrix}}
\newcommand{\emp}{\begin{pmatrix}
            \Cdot\\\Cdot
            \end{pmatrix}}

\DeclareMathOperator{\V}{\nu}
\DeclareMathOperator{\CC}{\mathtt{CC}}

\newcommand{\MLQC}{\MLQ_C}
\newcommand{\ab}{\mathbf{a}}
\newcommand{\bb}{\mathbf{b}}
\newcommand{\bc}{\mathbf{c}}
\DeclareMathOperator{\lh}{\ell-\texttt{ht}}

\newcommand{\Qb}{\mathbf{Q}}

\DeclareMathOperator{\wt}{wt}

\DeclareMathOperator{\pth}{par}
\DeclareMathOperator{\MLQ}{MLQ}

\DeclareMathOperator{\KN}{KN}
\DeclareMathOperator{\ZZ}{\mathbb{Z}}
\DeclareMathOperator{\NN}{\mathbb{N}}
\DeclareMathOperator{\prob}{Pr}

\DeclareMathOperator{\StatesC}{States_{C}}
\DeclareMathOperator{\StatesA}{States_{A}}
\DeclareMathOperator{\RowC}{Rows_{C}}
\DeclareMathOperator{\SplitC}{Split_{C}}

\DeclareMathOperator{\spl}{split}
\DeclareMathOperator{\row}{\mathtt{row}}
\DeclareMathOperator{\proj}{proj}

\DeclareMathOperator{\projC}{proj_{C}}
\DeclareMathOperator{\tasepC}{\mathbf{TASEP}_C}
\DeclareMathOperator{\tasepA}{\mathbf{TASEP}_A}
\DeclareMathOperator{\mlqC}{\mathbf{MLQ}_C}
\DeclareMathOperator{\mlqA}{\mathbf{MLQ}_A}
\DeclareMathOperator{\mlqAFM}{\mathbf{MLQ}_A^{FM}}

\DeclareMathOperator{\Ac}{\mathcal{A}}

\DeclareMathOperator{\cPair}{Pair_{C}}

\DeclareMathOperator{\States}{States}
\DeclareMathOperator{\refl}{refl}

\DeclareMathOperator{\FM}{FM}

\newlength\cellsize 
\savebox2{%
\begin{picture}(16,16)
\put(0,0){\line(1,0){16}}
\put(0,0){\line(0,1){16}}
\put(16,0){\line(0,1){16}}
\put(0,16){\line(1,0){16}}
\end{picture}}

\newcommand\cellify[1]{\def\thearg{#1}\def\nothing{}%
\ifx\thearg\nothing
\vrule width0pt height\cellsize depth0pt\else
\hbox to 0pt{\usebox2\hss}\fi%
\vbox to 16\unitlength{
\vss
\hbox to 16\unitlength{\hss$#1$\hss}
\vss}}

\newcommand\tableau[1]{\vtop{\let\\=\cr
\setlength\baselineskip{-16000pt}
\setlength\lineskiplimit{16000pt}
\setlength\lineskip{0pt}
\halign{&\cellify{##}\cr#1\crcr}}}

\savebox3{%
\begin{picture}(12,12)
\put(0,0){\line(1,0){12}}
\put(0,0){\line(0,1){12}}
\put(12,0){\line(0,1){12}}
\put(0,12){\line(1,0){12}}
\end{picture}}

\newcommand\expath[1]{%
\hbox to 0pt{\usebox3\hss}%
\vbox to 12\unitlength{
\vss
\hbox to 12\unitlength{\hss$#1$\hss}
\vss}}

\begin{document}

\title{Type $C$ multiline queues and the open-boundary TASEP}

\author{Olya Mandelshtam and Jer\'onimo Valencia-Porras}

\maketitle

\abstract{The totally asymmetric simple exclusion process (TASEP) with open boundaries is a finite Markov chain describing particles hopping between adjacent sites on a one-dimensional lattice with left and right boundary transitions governed by parameters $\alpha$ and $\beta$. The multispecies TASEP is a higher-rank generalization in which particles have different species.  Multiline queues were introduced by Ferrari and Martin (2007) to compute the stationary distribution of the multispecies TASEP on a circle. It has remained an open problem to find a combinatorial formula for the stationary distribution of the multispecies open-boundary TASEP. Using Kirillov--Reshetikhin crystals of type $C$, we construct type $C$ multiline queues and a corresponding Ferrari--Martin pairing algorithm that projects them to TASEP configurations. This yields a combinatorial formula for the stationary distribution of the multispecies open-boundary TASEP for the $\alpha=\beta=1$ specialization. }

\setcounter{tocdepth}{1}

\section{Introduction}

The \emph{rank $L$ open asymmetric simple exclusion process} (ASEP) is a model on a one-dimensional lattice with $n$ sites. Each site contains either a \emph{particle} labeled $j\in\{L,L-1,\ldots,1\}$, a \emph{barred particle} labeled $\bar j\in \{\bar L,\ldots,\bar 1\}$, or a vacancy labeled $0$. An unbarred particle with label $j$ and a barred particle with label $\bar j$ both have \emph{type} (or species) $j$. The priority order on particle labels is 
\[\bar L<\cdots<\bar 1<0<1<\cdots<L,
\]
and $x>y$ if $x$ has priority over $y$ in this order. In the bulk, adjacent sites swap with rate $1$ if the label on the left has priority, and rate $t$ otherwise. At the boundaries, particles of type $m$ may be replaced by barred particles $\bar m$, or vice versa, governed by the boundary parameters $\alpha,\beta,\gamma,\delta$. Since particles preserve type as the ASEP evolves in time, the multiset of types is fixed. We record this multiset as a partition $\lambda$, where the multiplicity of a part $m$ in $\lambda$ gives the number of particles or barred particles of type $m$.

This model is remarkable due to its connection to Koornwinder polynomials $K_\lambda(X;q,t,a,b,c,d)$: after a certain transformation of the parameters $a,b,c,d$, the \emph{partition function} of the ASEP with particle types given by $\lambda$ is the specialization of $K_\lambda$ at $X=q=1$. Moreover, the stationary probabilities correspond to an analogous specialization of certain nonsymmetric Koornwinder polynomials, called \emph{open ASEP polynomials} \cite{CGdGW2016}.

When $L=1$, the model is equivalent to the well-studied \emph{two-species ASEP} after mapping the 1, 0, and $\bar 1$ to the first class (heavy) particle, the second class (light) particle, and the vacancy, respectively. The classical single species ASEP dates back to \cite{MGP68,spitzer-1970}. Our rank-1 case corresponds to the two-species model studied in \cite{DJLS93,Uchiyama08}, with the interpretation of the second class particle due to Liggett's coupling construction \cite{Liggett76}. This model and many variants and generalizations of it have been studied through various perspectives, including the matrix ansatz \cite{DEHP93,BE07,Uchiyama08}, integrability \cite{Cantini24,KMS25}, and combinatorics \cite{DuchiSchaeffer05,CW07,CW11,Mandelshtam2016,MV15,CMW-2017}. 

The partition function and stationary distribution of the two-species ASEP can be computed as generating functions over \emph{rhombic staircase tableaux} \cite{CMW-2017,CMW23}, which have a rich history of evolution from earlier tableau models. In the single species case, these recover the classical staircase tableaux \cite{CW11} whose partition function is a specialization of moments of the Askey--Wilson polynomials. At $\gamma=\delta=0$, these tableaux specialize to \emph{rhombic alternative tableaux} \cite{MV15,MV16}, which are the two-species generalization of \emph{alternative tableaux}, a reinterpretation of the permutation tableaux of \cite{CW07}, which are in bijection with permutations. 
When $t=0$, the ASEP becomes the \emph{totally asymmetric simple exclusion process} (TASEP). The TASEP specializations of alternative tableaux and rhombic alternative tableaux are enumerated by Catalan numbers \cite{Viennot07Catalan,DuchiSchaeffer05} and triangular Catalan numbers \cite{Mandelshtam2016}, respectively; they are called \emph{Catalan tableaux} in the literature. Even earlier, Duchi and Schaeffer described probabilities of the single species TASEP using certain particle configurations called \emph{complete configurations} \cite{DuchiSchaeffer05}; we return to these in \cref{sec:AT}.

It has remained an open problem to extend the combinatorics of the open ASEP beyond rank-one. In this paper, we make the first step towards this goal by solving this problem in the specialization $\alpha=\beta=1$ at  $t=\gamma=\delta=0$, through constructing \emph{type $C$ multiline queues}. This TASEP is a generalization of a model studied in \cite{AALP}, referred to as the \emph{type $\check{C}$ TASEP} therein.

We explain the motivation for our construction. It is known that the Koornwinder polynomial factorizes at $q=1$ \cite{CGdGW2016}, and that the partition function of the rank $L$ ASEP with particle types given by $\lambda$, denoted $\mathcal{Z}_{\lambda,n}$, similarly factorizes into a product of partition functions of rank-1 ASEPs:
\begin{equation}
    \mathcal{Z}_{\lambda,n}=\prod_{j=1}^{L} Z(\lambda_j',n),
\end{equation}
where $L\coloneq \lambda_1$ is the rank, $\lambda'$ is the conjugate partition of $\lambda$, and $Z(k,n)$ denotes the partition function of the rank-one ASEP of size $n$ with $n-k$ second class particles. It is then reasonable to expect that one could describe the stationary distribution of the rank $L$ ASEP using a Cartesian product of the rank-1 objects. To do this, one would need to find a suitable projection map from tuples of rank-1 objects to states of the rank $L$ ASEP, together with a weight function on these tuples such that restricting along fibers of the projection map gives the stationary distribution. However, it is unclear how to produce either of these maps using the currently existing rank-one objects, even in the simplest cases (see also \cite[Section 4]{CMW23}). Instead, we turn to the type $A$ setting of the ASEP on a circle, for which the parallel story is well understood.

The partition function of the multispecies ASEP on a circle at $X=q=1$ is a specialization of the (type $A$) Macdonald polynomials \cite{CGW-2015}, which similarly factorize at $q=1$. A combinatorial formula for these polynomials was given in \cite{CMW18} by enhancing the weights of the \emph{generalized multiline queues} introduced by Martin in \cite{martin-2020} to compute probabilities of the multispecies ASEP. These multiline queues have precisely the form described above: a multiline queue for a rank $L$ ASEP on a circle can be viewed as a tuple of $L$ single rows, where each row may be viewed as a rank-1 configuration. Martin's construction gives both the projection map and the corresponding weight function to yield the desired probabilities. In the $t=0$ case, this construction becomes especially simple, reducing to the classical multiline queues of Ferrari and Martin \cite{FM07}. This formula was reinterpreted in terms of Kirillov--Reshetikhin (KR) crystals of type $A$ by Kuniba, Maruyama, and Okado in \cite{KMO15,KMOtetrahedron} with a crystal-theoretic reformulation of the projection map to states of the multispecies TASEP via the corner transfer matrix. Such techniques can be naturally extended to KR crystals of type $C$, leading us to introduce type $C$ multiline queues.

Type C multiline queues are particle configurations resembling multiline queues, but with a pair of rows at each level, and with particles and barred particles that pair in opposite directions. This construction arose as a natural diagrammatic interpretation of the procedure of Lenart and Schilling in \cite{LenartSchilling11} for computing the charge statistic on tensor products of type $C$ columns. Their procedure is based on \emph{virtualization} of type $C$ KR crystals \cite{Kashiwara96,OSS-VirtualCrystals}; for KN columns, this virtualization is realized by the splitting map \cite{Baker00,Lecouvey2002TypeC}. 

Several other objects in the literature resemble rank-one pieces of our construction \cite{DuchiSchaeffer05,Man20}. However, it is not clear how to stack those objects into a multiline object that projects to a multispecies process. The key point in our construction is that it is fundamentally based on type $C$ crystal structure: our objects reinterpret the virtualization map, allowing us to build on the type $A$ construction. Remarkably, the resulting objects admit a natural analog of the Ferrari--Martin multiline queue algorithm.

We define a type $C$ analog of the Ferrari--Martin projection map on the type $C$ multiline queues, and thus obtain a combinatorial formula for the stationary probabilities of the multispecies open TASEP at $\alpha=\beta=1$: for each TASEP state $\tau$, its stationary probability is proportional to the number of type $C$ multiline queues projecting to
$\tau$. We prove our result by constructing a Markov chain on type $C$ multiline queues that projects to the TASEP. Unlike the Ferrari--Martin Markov chain on type $A$ multiline queues, ours has a crystal-theoretic formulation. This crystal-theoretic construction, in both types $A$ and $C$, is new and gives a new perspective on how the TASEP naturally arises from KR crystals.

This article is organized as follows. \cref{sec:Preliminaries} contains definitions for TASEP Markov chains, crystals of types $A$ and $C$, and Ferrari--Martin multiline queues. In \cref{sec:typeC-mlqs}, we reinterpret type $C$ KR crystals as particle configurations in order to define \emph{type $C$ multiline queues}. Finally, in \cref{sec:crystal-MCs} we construct Markov chains on multiline queues of types $A$ and $C$ based on the corresponding crystal graphs, and prove that these project to the corresponding TASEP chains. We conclude with a discussion of the rank-one case and compare it with existing combinatorial objects for the same case in \cref{sec:AT}.

\section*{Acknowledgements}

We thank Travis Scrimshaw for many helpful conversations throughout this project and for valuable comments that helped improve the manuscript. We also thank Arvind Ayyer for discussions related to this project, and Svante Linusson for pointing us to a relevant paper. The authors were partially supported by the Natural Sciences and Engineering Research Council of Canada (NSERC) through Discovery Grants RGPIN-2026-04714 and RGPIN-2021-02568. 

\section{Preliminaries}
\label{sec:Preliminaries}

\subsection{TASEP Markov chains}

   A \emph{finite discrete Markov chain} consists of a finite set of \emph{states} $S$ and a $|S|\times |S|$ \emph{transition matrix} $P$ that encodes the dynamics: given two states $\sigma,\tau \in S$, the entry $P_{\sigma,\tau} = P(\sigma\to\tau)$ is the transition probability from state $\sigma$ to state $\tau$. The \emph{state diagram} of the Markov chain is a weighted directed graph where each directed edge corresponds to a transition with positive probability, recorded as the weight of the edge. In these diagrams, loops are usually ignored. A Markov chain is \emph{irreducible} if for any two states in the state diagram, there exists a path from one state to the other.

    The \emph{stationary distribution} of a Markov chain on $S$ with transition matrix $P$ is the left eigenvector $\vec{\pi}$ of $P$ with eigenvalue $1$ such that $\sum_{\sigma \in S} \pi(\sigma) = 1$ where $\pi(\sigma) = \vec{\pi}_\sigma$. If the Markov chain is irreducible, such an eigenvector is unique. Equivalently, the stationary distribution is proportional to the solution to the \emph{global balance equations:}
    \begin{align}\label{eq:global-balance}
        \pi(\tau)\sum_{\sigma \neq \tau} P(\tau\to\sigma) &= \sum_{\sigma\neq\tau} \pi(\sigma)P(\sigma\to\tau)\quad \text{ for all } \tau\in S. 
    \end{align}

A standard technique for computing the stationary distribution of a Markov chain is to realize it as the projection of a larger Markov chain with a simpler stationary distribution. We recall the notion of \emph{lumping} (or projection) of Markov chains. 

\begin{definition}[Lumping]
    Let $X$ be a Markov chain with states $S_X$ and transition matrix $P$, and let $f\colon S_X \to S_Y$ be a surjective map. Suppose $Q$ is a $|S_Y| \times |S_Y|$ matrix such that for all $x_0\in S_X$ and all $y\in S_Y$, writing $y_0=f(x_0)$, we have
    \begin{equation}\label{eq:def-lumping}
        \sum_{\substack{x\in S_X \\ f(x) = y}} P(x_0\to x) = Q(y_0\to y).
    \end{equation}
    In this case $Q$ is the transition matrix of a Markov chain $Y$ with states $S_Y$. We say that $Y$ is a \emph{lumping} of $X$ via the projection map $f$, and that $X$ is a \emph{lift} of $Y$.
\end{definition}

Once the stationary distribution of the lifted chain is known, the stationary distribution of the lumped Markov chain can be obtained by summing over the fibers of the projection map.

\begin{prop}
    Let $X$ and $Y$ be Markov chains with states $S_X$ and $S_Y$ and with stationary distributions $\pi_X$ and $\pi_Y$ respectively. If $Y$ is a lumping of $X$ via the projection map $f:S_X\to S_Y$, then 
    \begin{equation}\label{eq:SD-lumping}
        \pi_Y(y) = \sum_{\substack{x\in S_X \\ f(x) = y}}\pi_X(x).
    \end{equation}
    In particular, if $\pi_X$ is the uniform distribution, i.e., $\pi_X(x) = 1\,/\,|S_X|$ for all $x\in S_X$, then $$\pi_Y(y) = \dfrac{1}{|S_X|}\left| \{x\in S_X\colon f(x) = y\} \right|.$$
\end{prop}

\begin{remark}\label{rem:equiv_conds_lumping}
    Suppose $P$ and $Q$ are the transition matrices of $X$ and $Y$ respectively. Moreover, assume there is a function $f\colon S_X\to S_Y$ satisfying the following properties:
    \begin{itemize}
        \item If $P(x_1\to x_2) > 0$ and $f(x_1) \neq f(x_2)$, then $Q(f(x_1)\to f(x_2)) = P(x_1\to x_2)$.
        \item If $Q(y_1\to y_2) > 0$ then for all $x_1\in S_x$ such that $f(x_1) = y_1$ there is a unique $x_2\in S_X$ with $f(x_2)=y_2$ such that $Q(y_1\to y_2) = P(x_1\to x_2)$.
    \end{itemize}
    Then $Y$ is a lumping of $X$ via the projection map $f$. 
\end{remark}

The TASEP is often formulated as a continuous time Markov chain with rates assigned to local transitions and dynamics governed by exponential clocks. In this paper, we work instead with the associated discrete-time chain. This does not change the stationary distribution: when all allowed transitions have rate $1$, the chain obtained by giving
each possible local move equal probability, and adding self-loops so that
the total probability from each state is $1$, has the same stationary
distribution.

\subsubsection{The multispecies TASEP on a circle}\label{sec:TASEP}

Let $\lambda = (\lambda_1\geq \cdots\geq \lambda_k > 0)$ be a partition, and define the \emph{length} $\ell(\lambda)\coloneq k$ to be the number of nonzero parts. For a tuple $\ab=(a_1,\ldots,a_n)$, let $m_i(\ab)=|\{j:a_j=i\}|$ be the multiplicity of $i$ in $\ab$. Define the conjugate partition $\lambda'=(\lambda'_1,\lambda'_2,\ldots)$ by $\lambda'_i=m_i(\lambda)+m_{i+1}(\lambda)+\cdots$.

\begin{definition}\label{def:TypeA-ParticleSystem}
    The \emph{rank $L$ TASEP on a circle} on $n$ sites with species distribution $\lambda$ with $\lambda_1=L$ is a Markov chain on the states 
    \[\StatesA(\lambda,n)=\Big\{\ab\in\mathbb{N}^n\colon m_i(\ab)=m_i(\lambda) \text{ for }1\leq i\leq L,\text{ and }m_0(\ab)=n-\ell(\lambda)\Big\}.\]
    The transition probabilities $P\colon\StatesA(\lambda,n)\times\StatesA(\lambda,n)\to[0,1]$ are:
    \begin{enumerate}[itemsep=3mm,leftmargin=3em]
        \item[A1.] If $\sigma=(\ldots,a_i,a_{i+1},\ldots)$ and  $\tau=(\ldots,a_{i+1},a_{i},\ldots)$ for $a_i<a_{i+1}$, then $P(\sigma\to\tau) = \dfrac{1}{n}$, and
        \item[A2.] $P(\sigma\to\sigma) = 1-\sum\limits_{\tau\neq \sigma}P(\sigma\to\tau)$,
    \end{enumerate}
    where all indices are on a circle with $n$ sites, that is, $a_{n+1} = a_1$. All other transition probabilities are 0.
    \vspace{1em}
    We denote this Markov chain by $\tasepA(\lambda,n).$
\end{definition}


\begin{example}\label{ex:tasepA}
    The rank 2 TASEP on $n=3$ sites with species $\lambda = (2,1)$ has states \begin{equation*}\StatesA((2,1),3) = \left\{ (0,1,2),(0,2,1),(1,0,2),(1,2,0),(2,0,1),(2,1,0)\right\}
    \label{eq:states}
    \end{equation*}
    with state diagram shown in \cref{fig:ex_tasepA} (left). The stationary distribution for the two distinct states up to cyclic symmetry is given by $\vec{\pi}((1,2,0))=\frac{1}{9}$ and $\vec{\pi}((2,1,0))=\frac{2}{9}$.

    \begin{figure}
        \centering
        \resizebox{!}{2in}{\begin{tikzpicture}[
            vertex/.style={minimum size=20pt, inner sep=0pt},
            edge/.style={->, thick, >=stealth},
            scale = 0.8
        ]
            \node[vertex] (v1) at (0:3.5cm) {$(0,1,2)$};
            \node[vertex] (v2) at (60:3.5cm) {$(0,2,1)$};
            \node[vertex] (v3) at (120:3.5cm) {$(1,0,2)$};
            \node[vertex] (v4) at (180:3.5cm) {$(1,2,0)$};
            \node[vertex] (v5) at (240:3.5cm) {$(2,0,1)$};
            \node[vertex] (v6) at (-60:3.5cm) {$(2,1,0)$};
        
            \draw[edge] (v1) -- (v2);
            \draw[edge] (v1) -- (v3);
            \draw[edge] (v2) -- (v5);
            \draw[edge] (v3) -- (v4);
            \draw[edge] (v4) -- (v2);
            \draw[edge] (v4) -- (v6);
            \draw[edge] (v5) -- (v3);
            \draw[edge] (v5) -- (v6);
            \draw[edge] (v6) -- (v1);
        
        \end{tikzpicture}}
       
        \caption{We show the state diagram for $\tasepA((2,1),3)$, with self-loops omitted, and with each arrow representing a transition with probability $\frac{1}{3}$ (equivalently, with uniform rate).} 
        \label{fig:ex_tasepA}
    \end{figure}
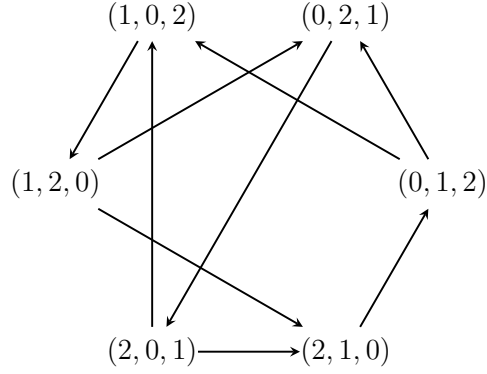
    
\end{example}

In the single species TASEP with $\lambda=(1^k)$, which we call the \emph{rank-one case}, the stationary distribution of $\tasepA(\lambda,n)$ is uniform. In particular, the state diagram is a \emph{balanced graph}, with each vertex having an equal number of incoming and outgoing arrows (all with the same probability). We leave the computation as an exercise for the reader.

\subsubsection{The open-boundary multispecies TASEP}

For a word $\mathbf{a}=(a_1,\ldots,a_n)\in\mathbb{Z}^n$, define $|\mathbf{a}|=(|a_1|,\ldots,|a_n|)$. For $j>0$, we write $\bar j=-j$ and $|\bar j|=|j|=j$.
 
\begin{definition}\label{def:TypeC-ParticleSystem}
    The \emph{rank $L$ open-boundary multispecies TASEP} on $n$ sites with species distribution $\lambda$, where $\lambda_1=L$, is a Markov chain on the state space
    \[\StatesC(\lambda,n)=\Big\{\mathbf{a}\in\mathbb{Z}^n\, \colon\, m_i(|\mathbf{a}|)=m_i(\lambda) \text{ for }1\leq i\leq L,\text{ and }m_0(\ab)=n-\ell(\lambda)\Big\}.\]   
    The transition matrix $P$ is defined by $P(\sigma\to\tau) = \dfrac{1}{n+1}$ whenever $\sigma$ and $\tau$ satisfy\\   
    \begin{enumerate}[itemsep=3mm,leftmargin=3em]
        \item[C1.] $\sigma=(\overline{a_1},a_2,\ldots)$ and $\tau = (a_1,a_2,\ldots)$ for $a_1>0$
        \item[C2.] $\sigma = (\ldots a_{n-1},a_n)$ and $\tau = (\ldots,a_{n-1},\overline{a_n})$ for $a_n>0$
        \item[C3.] $\sigma = (\ldots,a_i,a_{i+1},\ldots)$ and $\tau = (\ldots,a_{i+1},a_{i},\ldots)$ for $a_i>a_{i+1}$,
    \end{enumerate}\vspace{1em}
    and $P(\sigma\to\sigma) = 1-\sum\limits_{\tau\neq \sigma}P(\sigma\to\tau)$. In all other cases, $P(\sigma\to\tau)=0$.
    \vspace{1em}
    
    We denote this Markov chain by $\tasepC(\lambda,n).$
\end{definition}

\begin{remark}
In the type $C$ formulation, the boundary transitions of the open boundary TASEP are interpreted as reflecting boundary transitions (C1) and (C2), following the terminology of \cite{CGdGW2016}. 
When $L=1$, the label $0$ corresponds to the ``second class particle'' in the two-species open TASEP studied in the literature \cite{Uchiyama08}.  
    
    The model we study is a generalization of the $C^\vee$ TASEP studied in \cite{AALP}, where only the case $\lambda=(n,n-1,\ldots,1)$ was considered. It is also a special case of the \emph{open-boundary multispecies ASEP} at $\gamma=\delta=t=0$ and $\alpha=\beta=1$, and thus has a close connection to the type $C$ Macdonald polynomials (Koornwinder polynomials) \cite{CGdGW2016}. In greater generality, parameters $\alpha$ and $\beta$ guide the boundary rates: transitions (C1) and (C2) occur with rates $\alpha$ and $\beta$, respectively, while the bulk transitions (C3) have rate $1$. The definition above is the discretized specialization $\alpha=\beta=1$.
\end{remark}

\begin{remark}\label{rem:particle/hole-symmetry}
    The open-boundary multispecies ASEP exhibits a particle-hole symmetry. In terms of the elements of $\StatesC(\lambda,n)$, the symmetry is given by exchanging barred and unbarred particles and flipping the direction of the transition arrows. 
    This can be seen explicitly in \cref{ex:rank1-tasepC}, where the symmetry is depicted by reflecting horizontally over the center of the state diagram. 
\end{remark}

The stationary distribution for $\tasepC$ is significantly more complicated than for $\tasepA$. This can be already seen in the rank-one case, where unlike the type $A$ case, the stationary distribution for $\tasepC(1^k,n)$ is no longer uniform.

\begin{example}\label{ex:rank1-tasepC}
    Let $\lambda = (1,1)$ and $n=3$. The states of $\tasepC((1,1),3)$ are the twelve vectors 
    \begin{align*}
        \StatesC((1,1),3) = \{ &(1,1,0),(\overline{1},1,0),(1,\overline{1},0),(\overline{1},\overline{1},0),
        (1,0,1),(\overline{1},0,1),\\&(1,0,\overline{1}),(\overline{1},0,\overline{1}),
        (0,1,1),(0,\overline{1},1),(0,1,\overline{1}),(0,\overline{1},\overline{1})\}.
    \end{align*}
    \cref{fig:tasepC-singlespecies} shows the state diagram of this Markov chain. The stationary distribution is given by $\vec{\pi}\left((1,\overline{1},0)\right)=\vec{\pi}\left((0,1,\overline{1})\right)=\frac{2}{14}$ and $\vec{\pi}(\tau)=\frac{1}{14}$ for all other states $\tau$.

    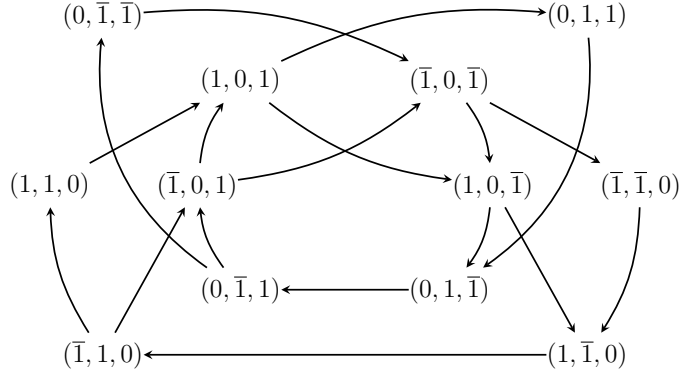
\begin{figure}
        \centering
        \resizebox{3.5in}{!}{\begin{tikzpicture}[
            vertex/.style={minimum size=20pt, inner sep=0pt},
            edge/.style={->, thick, >=stealth},
            scale = 0.6
        ]
            \node[vertex] (v11) at (180:8 cm) {$(1,1,0)$};
            \node[vertex] (v12) at (215:8 cm) {$(\overline{1},1,0)$};
            \node[vertex] (v13) at (325:8 cm) {$(1,\overline{1},0)$};
            \node[vertex] (v14) at (0:8 cm) {$(\overline{1},\overline{1},0)$};

            \node[vertex] (v21) at (135:4 cm) {$(1,0,1)$};
            \node[vertex] (v22) at (180:4 cm) {$(\overline{1},0,1)$};
            \node[vertex] (v23) at (0:4 cm) {$(1,0,\overline{1})$};
            \node[vertex] (v24) at (45:4 cm) {$(\overline{1},0,\overline{1})$};

            \node[vertex] (v31) at (35:8 cm) {$(0,1,1)$};
            \node[vertex] (v32) at (225:4 cm) {$(0,\overline{1},1)$};
            \node[vertex] (v33) at (315:4 cm) {$(0,1,\overline{1})$};
            \node[vertex] (v34) at (145:8 cm) {$(0,\overline{1},\overline{1})$};

            \draw[edge] (v11) -- (v21);
            \draw[edge] (v12) to[bend left=15] (v11);
            \draw[edge] (v12) -- (v22);
            \draw[edge] (v13) to[bend left=0] (v12);
            \draw[edge] (v14) to[bend left=15] (v13);

            \draw[edge] (v21) to[bend left=15] (v31);
            \draw[edge] (v21) to[bend right=15] (v23);
            \draw[edge] (v22) to[bend left=15] (v21);
            \draw[edge] (v22) to[bend right=15] (v24);
            \draw[edge] (v23) -- (v13);
            \draw[edge] (v23) to[bend left=15] (v33);
            \draw[edge] (v24) -- (v14);
            \draw[edge] (v24) to[bend left=15] (v23);

            \draw[edge] (v31) to[bend left=30] (v33);
            \draw[edge] (v32) to[bend left=30] (v34);
            \draw[edge] (v32) to[bend left=15] (v22);
            \draw[edge] (v33) to[bend left=0] (v32);
            \draw[edge] (v34) to[bend left=15] (v24);
            
        \end{tikzpicture}}
        \caption{The state diagram for $\tasepC((1,1),3)$ (with self-loops omitted), where each arrow represents a transition with probability $\frac{1}{4}$ (equivalently, with uniform rate). The reflection about the center of the diagram shows the symmetry of the Markov chain when exchanging $1$ with $\overline{1}$ and reversing the arrows.}
        \label{fig:tasepC-singlespecies}
    \end{figure}
\end{example}

The stationary distribution for the rank one case $\tasepC(1^k,n)$ is well-understood. It can be described through several related combinatorial formulas where the resulting objects are enumerated by triangular Catalan numbers, including formulas involving pattern-avoiding permutations when $k=n$, assembl\'ees for general $k$, and tableau objects for the more general boundary parameters $\alpha,\beta$ \cite{DuchiSchaeffer05,Mandelshtam2016}. The partially asymmetric model with hopping parameter $t$ was studied in \cite{MV16,Corteel2020a}. See \cref{sec:AT} for further discussion on the rank one case.

Our motivation for this work is to obtain a combinatorial formula for the stationary distribution of $\tasepC(\lambda,n)$ for general $\lambda$.

\subsection{KR crystals}

Crystal bases were defined to study certain limits of representations of quantum (affine) algebras $U_q'(\mathfrak{g})$ (i.e. without derivation) as $q\rightarrow0$ where $\mathfrak{g}$ is an Lie (resp.~affine Kac-Moody) algebra. Classical crystals are associated to a Lie algebra coming from a \emph{finite Cartan datum}, each of which has an equivalent description as a \emph{Dynkin diagram}. 
Affine crystals come from \emph{affine Cartan datum} associated to affine Kac-Moody algebras. We will identify the corresponding Lie algebra with its 
\emph{affine Dynkin diagram}. We will focus on the combinatorics of the Dynkin diagrams rather than the algebras themselves to give our definitions, and refer the reader to \cite{hong2002introduction,humphreys1992reflection} for a more detailed exposition. Our definitions follow the exposition of \cite{briggs2021combinatorial,fourier2009kirillov,LenartSchilling11}. The crystals we consider are called \emph{normal} or \emph{regular} in the literature. 

Let $\Phi$ be a finite or affine root system inside an Euclidean space $V$ with inner product denoted by $\langle \cdot\,,\cdot\rangle$. Let $\{\alpha_i\}_{i\in I}$ be the \emph{simple roots} of $\Phi$, whose relationships to each other are encoded by the Dynkin diagram. $\Phi$ comes equipped with a \emph{weight lattice} $\Lambda$, which is a lattice spanning $V$ such that $\Phi \subset \Lambda$, and if $\lambda \in \Lambda$ and $\alpha \in \Phi$ then $\langle\lambda,\alpha^\vee \rangle \in \mathbb{Z}$ where $\alpha^{\vee} = \dfrac{2\alpha}{\langle \alpha,\alpha\rangle}.$ We work with the affine root systems of type $A_{n-1}^{(1)}$ and $C_{n}^{(1)}$ whose Dynkin diagrams are shown in \cref{fig:dynkinA1,fig:dynkinC1}. For a complete classification of affine root systems we refer the reader to \cite[Chapter 1]{Macdonald_2003}. 

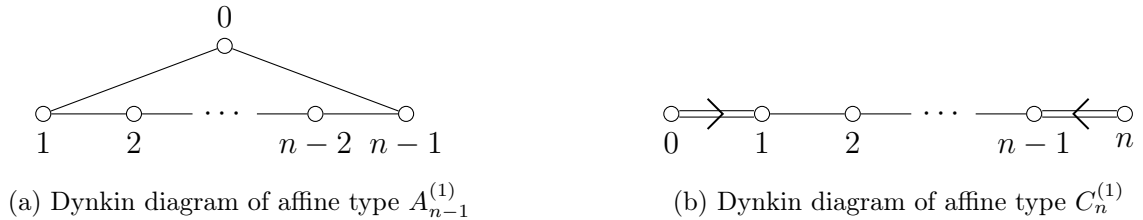
\begin{figure}[h!]
    \centering
    \begin{subfigure}[b]{0.45\textwidth}
        \centering
        \begin{tikzpicture}[scale=1.2]
        \node[draw, circle, inner sep=2pt, label=above:{$0$}] (0) at (2, 0.75) {};
        \node[draw, circle, inner sep=2pt, label=below:{$1$}] (1) at (0, 0) {};
        \node[draw, circle, inner sep=2pt, label=below:{$2$}] (2) at (1, 0) {};
        \node (dots) at (2, 0) {$\cdots$};
        \node[draw, circle, inner sep=2pt, label=below:{$n-2$}] (n-1) at (3, 0) {};
        \node[draw, circle, inner sep=2pt, label=below:{$n-1$}] (n) at (4, 0) {};
    
        \draw (0) -- (1);
        \draw (1) -- (2);
        \draw (2) -- (dots);
        \draw (dots) -- (n-1);
        \draw (n-1) -- (n);
        \draw (n) -- (0);
    \end{tikzpicture}
        \caption{Dynkin diagram of affine type $A_{n-1}^{(1)}$}
        \label{fig:dynkinA1}
    \end{subfigure}
    \begin{subfigure}[b]{0.45\textwidth}
        \centering
        \begin{tikzpicture}[
    mid arrow/.style={
        postaction={decorate,
            decoration={
                markings,
                mark=at position #1 with {\arrow{Straight Barb[line width=0.75pt,length=2mm]}}
            }
        }
    },
    mid arrow/.default=0.5,
    scale=1.2, baseline=(current bounding box.center)]
    \node[draw, circle, inner sep=2pt, label=below:{$0$}] (0) at (-1,0) {};
    \node[draw, circle, inner sep=2pt, label=below:{$1$}] (1) at (0,0) {};
    \node[draw, circle, inner sep=2pt, label=below:{$2$}] (2) at (1,0) {};
    \node (dots) at (2,0) {$\cdots$};
    \node[draw, circle, inner sep=2pt, label=below:{${n-1}$}] (n-1) at (3,0) {};
    \node[draw, circle, inner sep=2pt, label=below:{$n$}] (n) at (4,0) {};

    \draw (1) -- (2);
    \draw (2) -- (dots);
    \draw (dots) -- (n-1);
    
    \draw[double, double distance=2pt, , mid arrow=0.6] (n) -- (n-1);
    \draw[double, double distance=2pt, , mid arrow=0.6] (0) -- (1);
\end{tikzpicture}
        \caption{Dynkin diagram of affine type $C_n^{(1)}$}
        \label{fig:dynkinC1}
    \end{subfigure}
\end{figure}

\begin{defn}
    A \emph{crystal} of type $\mathfrak{g}$ is a nonempty set $B$ together with
    \begin{itemize}
        \item (Kashiwara) \emph{crystal operators} $e_i,f_i \colon B\to B\sqcup\{\mathbf{0}\}$ indexed by $i\in I$, and 
        \item a \emph{weight map} $\wt \colon B \to \Lambda$ 
    \end{itemize}
    such that for any $x,y\in B$, $y = e_i(x)$ if and only if $x = f_i(y)$, and 
    $\wt(y) = \wt(x)+\alpha_i.$ We have statistics $\varepsilon_i,\varphi_i \colon B\to \mathbb{N}\sqcup \{\infty\}$ given by $$\varepsilon_i(x) = \max\{j \colon e_i^j(x) \neq \mathbf{0}\} \quad\text{and}\quad \varphi_i(x) = \max\{j \colon f_i^j(x) \neq \mathbf{0}\}$$ and for all $x\in B$ they satisfy 
    \begin{equation}\label{eq:eps-phi}
        \varphi_i(x)-\varepsilon_i(x) = \langle \wt(x),\alpha_i^{\vee}\rangle.
    \end{equation}
\end{defn}

Given a sequence of crystals $B_1,\ldots,B_n$ of the same type, we construct their \emph{tensor product}, which is also a seminormal crystal endowed with a crystal structure as follows. The underlying set of $B_1\otimes \cdots \otimes B_n$ is simply the Cartesian product $B_1\times \cdots \times B_n$. For $\mathbf{b}=b_1\otimes b_2\otimes\cdots\otimes b_n \in B_1\otimes \cdots \otimes B_n$ the weight function is
\[\wt(\mathbf{b})=\wt(b_1)+\cdots+\wt(b_n).\]
The  action of the crystal operators is determined by the \emph{signature rule}. To compute the crystal operators and statistics associated to the index $i\in I$, we construct the \emph{ith parenthesis word} of $\mathbf{b}$: 

\begin{equation}\label{eq:par}\underbrace{)\,)\, \ldots\, )}_{\displaystyle \varphi_i(b_1)}\;\underbrace{(\,( \,\ldots \,(}_{\displaystyle\varepsilon_i(b_1)}\;\underbrace{)\,)\, \ldots\, )}_{\displaystyle\varphi_i(b_2)}\;\underbrace{(\,( \,\ldots \,(}_{\displaystyle\varepsilon_i(b_2)}\;\ldots\;\underbrace{)\,)\, \ldots\, )}_{\displaystyle\varphi_i(b_n)}\;\underbrace{(\,( \,\ldots \,(}_{\displaystyle\varepsilon_i(b_n)}\end{equation}
We reduce this word according to the \emph{bracketing (or signature) rule}, as follows. 

\begin{definition}[Bracketing rule]\label{def:bracketing-rule}
    Let $w$ be a word in the alphabet $\{\,(\,,\,)\,\}$. The \emph{bracketing rule} iteratively matches pairs of open and closed parentheses that have no unmatched parentheses between them. 
    This procedure determines pairs of parentheses that are \emph{matched} in $w$ and leftover parentheses that are \emph{unmatched}.
\end{definition}

Ignoring all the matched parentheses in \eqref{eq:par} we obtain the \emph{ith reduced parenthesis word} of $\bb$, which has the form 
$$\underbrace{)\;)\;\ldots\;)}_{\displaystyle p}\;\underbrace{(\;(\;\ldots\;(}_{\displaystyle q}$$
for some nonnegative integers $p$ and $q$. The string lengths for $\bb$ are given by $\widetilde{\varepsilon}_i(\bb) = q$ and $\widetilde{\varphi}_i(\bb) = p$. Suppose the leftmost open parenthesis ``$($'' corresponds to $\varepsilon_i(b_k)$ and the rightmost closed parenthesis ``$)$'' corresponds to $\varphi_i(b_\ell)$ in the parenthesis word. Then 
\begin{align*}
    \widetilde{e}_i(\bb) = b_1\otimes\cdots\otimes e_i(b_k)\otimes\cdots\otimes b_n \qquad\text{and}\qquad
    \widetilde{f}_i(\bb) = b_1\otimes\cdots\otimes f_i(b_\ell)\otimes\cdots\otimes b_n.
\end{align*}
If no such indices $k$ or $\ell$ exist then $\widetilde{e}_i$ and $\widetilde{f}_i$ act trivially, i.e., $\widetilde{e}_i(\bb) = \mathbf{0}$ and $\widetilde{f}_i(\bb) = \mathbf{0}$, respectively.

Affine crystals were introduced in \cite{Kang92,kashiwara02} to extend the theory of crystal bases to representations of quantum affine algebras. \emph{Kirillov--Reshetikhin (KR) crystals} are a class of affine crystals corresponding to finite-dimensional Kirrilov--Reshetikhin modules over \(U_q'(\mathfrak{g})\), where $\mathfrak{g}$ is an affine Lie algebra. We restrict our description to KR crystals and direct the reader to \cite{kashiwara02} for the general definitions. Fourier, Okado, and Schilling \cite{fourier2009kirillov} constructed combinatorial models for KR crystals of nonexceptional type. These crystals are denoted by $B^{r,s}$, where $r$ is an index of the Dynkin diagram of $\mathfrak{g}$ that is not the affine node $0$, and $s$ is an arbitrary positive integer. 

We restrict to the case $s=1$ for types $A_{n-1}^{(1)}$ and $C_n^{(1)}$, and to simplify notation, we write $B^k$ to denote the KR crystal $B^{k,1}$. In these cases, $B^{k}$ is isomorphic, as a classical crystal, to the crystal of Kashiwara-Nakashima columns in types $A_{n-1}$ and $C_n$, respectively  \cite{KASHIWARA1994295}. This means that $e_i$ and $f_i$ agree in both the classical and affine cases for $i\neq 0$, and the affine structure comes from the \emph{affine operators} $e_0$ and $f_0$, associated to the affine node of the Dynkin diagram. We describe these crystal structures explicitly in the following sections.

The \emph{combinatorial R matrix} is the counterpart of the R matrix for $U_q(\mathfrak{g})$-modules in terms of crystal bases \cite[Section 1.3]{Kang92}. For KR crystals, the combinatorial $R$ matrix is the unique crystal isomorphism $$R\colon B^{r_1,s_1}\otimes B^{r_2,s_2} \xrightarrow{\quad\simeq \quad} B^{r_2,s_2}\otimes B^{r_1,s_1}.$$
For $s_1=s_2=1$, the combinatorial R matrix has an explicit combinatorial description for type $A_{n-1}^{(1)}$, given by Nakayashiki and Yamada \cite{NY95}, which we review in the following section. 

\begin{definition}
    For a partition $\lambda$ with $\lambda_1 = L$, define the tensor product of KR crystals
    \[B_\lambda\coloneq \bigotimes_{i=1}^L B^{\lambda_i'} = B^{\lambda_1'}\otimes B^{\lambda_2'}\otimes\cdots B^{\lambda_L'}.\] 
\end{definition}

\begin{definition}\label{def:Ri}
    Let $B = B^{r_1}\otimes B^{r_2}\otimes\cdots B^{r_k}$ be a tensor product of
    KR crystals of a fixed type, and let $R$ be the combinatorial $R$ matrix of such type. For $1\leq i \leq k+1$, define $R_i$ to be the combinatorial $R$ matrix acting on the $i$th and $(i+1)$th components of an element $b\in B$. Explicitly, if $b = b_1\otimes b_2 \otimes \cdots \otimes b_k$, then $$R_i(b) = b_1\otimes\cdots\otimes R(b_i\otimes b_{i+1})\otimes\cdots\otimes b_{k}.$$
\end{definition}


We review two particular crystals introduced by Kashiwara and Nakashima \cite{KASHIWARA1994295} in their seminal work describing crystal graphs for representations of classical Lie algebras.  

\subsubsection{KN columns of type $A$}

Fix $n\in \NN$ and let $I=\{0,1,\ldots,n-1\}$. We use the level-zero projection of the affine root system of type $A_{n-1}^{(1)}$ on the lattice $\Lambda=\ZZ^n$. The simple roots are $\alpha_i = \vec{e}_i - \vec{e}_{i+1}$ for $i\in I$, where $\vec{e}_j$ denotes the $j$th the standard basis vector of $\mathbb{R}^n$ and indices are taken cyclically modulo $n$. 

The affine Weyl group $W(A_{n-1}^{(1)})$ is the Coxeter group generated by $\{s_0,s_1,\ldots,s_{n-1}\}$ subject to the relations 
\[s_i^2 = 1,\qquad s_is_{i+1}s_i = s_{i+1}s_is_{i+1}\quad \text{ for all }i\in I\]
(again with indices taken cyclically modulo $n$), and 
$s_is_j = s_js_i$ whenever $i$ and $j$ are not adjacent in the affine Dynkin diagram. The level-zero action of this group on the lattice $\Lambda$ is defined by 
\[
s_i(\mu_1,\ldots,\mu_n) = \begin{cases}
    (\mu_1,\ldots,\mu_{i+1},\mu_i,\ldots,\mu_n) &\text{ if } i\in I\setminus\{0\}, \\
    (\mu_n,\mu_2,\ldots,\mu_{n-1},\mu_1) &\text{ if } i=0.
\end{cases}
\]
Under this action, the generators $s_i$ correspond exactly to the $\tasepA$ transitions from \cref{def:TypeA-ParticleSystem}. 

Note that the usual affine action differs from the level-zero action when defining $s_0$, which normally includes a translation: 
\[s_0(\mu_1,\ldots,\mu_n)=(\mu_n+1,\mu_2,\ldots,\mu_{n-1},\mu_1-1).
\]
The level-zero action is obtained by only taking the linear part.

We now review the construction of a crystal of type $A_{n-1}^{(1)}$. For  $n\geq 1$, denote the ordered alphabet $\mathcal{A}_n^A =  1<2<\cdots<n.$

\begin{definition}\label{def:KN-cols-type-A}
Let $k,n \in\NN$. A filling $T$ of a tableau of shape $1^k$ with entries from $\mathcal{A}^A_n$ is a \emph{Kashiwara-Nakashima (KN) column of type $A$} of size $k$ if the entries strictly increase from top to bottom according to the order of the alphabet. We denote the set of all such columns with $\KN_A(k,n).$  The weight map on $T\in \KN_A(k,n)$ is 
$$\wt(T) = \sum_{j\in T}\vec{e}_j.$$
\end{definition}

For conciseness, we describe the crystal operators directly on tensor products of KN columns. The action of the operators act is defined uniformly in types $A$ and $C$ using a parentheses word associated to the columns. Only the construction of this parentheses words depends on the type.

\begin{definition}[Type $A$ parenthesis word]
     Let $\mathbf{T}=T_1\otimes\cdots\otimes T_m$ be a tensor product of type $A$ columns, and let $1\leq i< n$. The \emph{$i$th parenthesis word} $\pth_i(\mathbf{T})$ is obtained by concatenating the words
    \[\pth_i(T)=\pth_i(T_1)\cdots\pth_i(T_m),\qquad\qquad \pth_i(T_j)=\begin{cases}
        )&\text{if }i\in T_j\text{ and }i+1\not\in T_j\\ (&\text{if }i+1\in T_j\text{ and }i\not\in T_j\\\emptyset&\text{otherwise}
    \end{cases}\,.\]
    Let $\overline{\pth}_i(\mathbf{T})$ be the reduced parenthesis word obtained from applying the bracketing rule to $\pth_i(\mathbf{T})$.
\end{definition}

\begin{definition}[Classical crystal operators from parenthesis words]
\label{def:crystal-operators-parentheses}
   Let $\mathbf{T}=T_1\otimes\dots\otimes T_m$  be a tensor product of columns with a reduced $i$th parenthesis word $w=\overline{\pth}_i(\mathbf{T})$. 
   
   If $w$ has no ``$)$''s , $\widetilde{f}_i(\mathbf{T})=\mathbf{0}$. Otherwise, let $T_k$ be the factor corresponding to the rightmost closed parenthesis of $w$. Then 
   \[\widetilde{f}_i(\mathbf{T})=T_1\otimes\cdots\otimes f_i(T_k)\otimes \cdots\otimes T_m.\]\
   
    If $w$ has no ``$($''s , $\widetilde{e}_i(\mathbf{T})=\mathbf{0}$. Otherwise, let $T_\ell$ be the factor corresponding to the leftmost open parenthesis of $w$. Then    \[\widetilde{e}_i(\mathbf{T})=T_1\otimes\cdots\otimes e_i(T_\ell)\otimes \cdots\otimes T_m.\]\
\end{definition}

In a type $A$ column $T$, $f_i(T)$ changes an $i$ to an $i+1$, while $e_i(T)$ changes an $i+1$ to an $i$, when they act nontrivially. When $\mathbf{T}$ is a tensor product of type $A$ columns, $\widetilde{f}_i$ changes an $i$ to an $i+1$ in the tensor factor corresponding to the chosen closed parenthesis, and $\widetilde{e}_i$ changes an $i+1$ to an $i$ in the tensor factor corresponding to the chosen open parenthesis.

\begin{example}\label{ex:crystal-ops-type-A}
    Let $n=6$ and consider the following tensor product of KN columns of type $A$:

    $$\mathbf{T}\;=\;\raisebox{3pt}{\tableau{2\\4}}\;\otimes\;\raisebox{18pt}{\tableau{1\\3\\5\\6}}\;\otimes\;\raisebox{11pt}{\tableau{2\\3\\5}}\;\otimes\;\raisebox{18pt}{\tableau{1\\2\\4\\5}}\;.$$
   To obtain $f_3(B)$ and $e_5(B)$, we compute the following (using ``$.$'' to separate the factors):
    \begin{alignat*}{2}
        \pth_3(w) &= \,( \,. \,)\, .\, )\, .\, (\,  \quad &&\longrightarrow\qquad \overline{\pth}_3(w) = \,.\, .\, )\, .\, (\, \\
        \pth_5(w) &= \,  .\,(\,)\,  .\,)\,  .\, )\,   \quad &&\longrightarrow\qquad \overline{\pth}_5(w) = \,  .\,  .\,)\,  .\, )\,   
    \end{alignat*}
    In $\overline{\pth}_3(w)$, the rightmost closed parenthesis ``$)$'' corresponds to the third tensor factor of $\mathbf{T}$, so 
    $$\widetilde{f_3}(\mathbf{T}) \;=\;\raisebox{3pt}{\tableau{2\\4}}\;\otimes\;\raisebox{18pt}{\tableau{1\\3\\5\\6}}\;\otimes\;\raisebox{11pt}{\tableau{2\\\mathbf{4}\\5}}\;\otimes\;\raisebox{18pt}{\tableau{1\\2\\4\\5}}$$ On the other hand, $\overline{\pth}_5(w)$ has no open parenthesis ``$($'', and so $\widetilde{e_5}(\mathbf{T})=\mathbf{0}$.
\end{example}

A description for the affine crystal operators for the crystal $B^{k,s}$ of type $\Phi = A_{n-1}^{(1)}$ was given in \cite{Kang92} for $s=1$ and in \cite{Shimozono2002} for general $s$. We restrict our description to the column case $s=1$. For $i=1,2,\ldots,n-1$, the operators $\widetilde{e}_i$ and $\widetilde{f}_i$ act as in the classical case.  The affine operators are obtained by considering the
alphabet cyclically modulo $n$, with $0\equiv n$, and defining the $0$th parentheses word $\pth_0$ accordingly. Thus in type $A$, the operator
$\widetilde f_0$ (resp.~$\widetilde e_0$) replaces an $n$ by a $1$ (resp.~a $1$ by an $n$) in the tensor factor chosen according to \cref{def:crystal-operators-parentheses}.

\begin{example}
    For the tensor product $\mathbf{T}$ from \cref{ex:crystal-ops-type-A}, we compute 
    $$\pth_0(w) = \,. \,( \, )\,.\,  .\,  ) \quad \longrightarrow\quad \overline{\pth}_0(w) = \,. \, .\,  .\,  ),$$
    and so $\widetilde{f_0}$ acts on the last factor of $\mathbf{T}$ by replacing $1$ with $6$ and reordering the column: 
    $$\widetilde{f_0}(\mathbf{T}) \;=\;\raisebox{3pt}{\tableau{2\\4}}\;\otimes\;\raisebox{18pt}{\tableau{1\\3\\5\\6}}\;\otimes\;\raisebox{11pt}{\tableau{2\\4\\5}}\;\otimes\;\raisebox{18pt}{\tableau{2\\4\\5\\\mathbf{6}}}\;.$$
    On the other hand, $\overline{\pth}_0(w)$ has no open parenthesis ``$($'', so $\widetilde{e_0}(\mathbf{T})=\mathbf{0}$. 
    
\end{example}

The combinatorial R matrix for type $A$ was described combinatorially by Nakayashiki and Yamada \cite{NY95} in terms of a \emph{cylindrical pairing rule}. The rule is described by reinterpreting the elements of $B^{k}$ as columns of particle configurations by mapping $T\in \KN_A(k,n)$ to a column of height $n$ having a single particle in row $r$ (numbered top to bottom) for each $r\in T$.

\begin{definition}[Nakayashiki--Yamada (NY) rule]\label{def:NY-rule}
     Let $T_1\otimes T_2 \in B^{r_1}\otimes B^{r_2}$. First suppose $r_1 < r_2$. We perform the following \emph{pairing procedure}: for each particle in $T_1$, from top to bottom, find the bottommost particle in $T_2$ that is weakly above it cyclically (that is, assuming the column is on a cylinder). We say those two particles are paired. The remaining $r_2-r_1$ unpaired particles in $T_2$ are then transferred to the corresponding rows in $T_1$ to obtain the output $R(T_1\otimes T_2) = S_2 \otimes S_1 \in B^{r_2}\otimes B^{r_1}$.

     When $r_1>r_2$, the procedure is similar, except the pairing is performed from the other factor in the opposite direction: for each particle in $T_2$, from bottom to top, pair it the topmost particle in $T_1$ that is weakly below it cyclically. The remaining $r_1-r_2$ unpaired particles in $T_1$ are then transferred to the corresponding rows in $T_2$ to obtain  $R(T_1\otimes T_2) = S_2 \otimes S_1 \in B^{r_2}\otimes B^{r_1}$.
\end{definition}

\begin{example}\label{ex:combR-typeA}

Let $r_1 = 5$, $r_2 = 3$ and $n=7$. Consider the tensor product $T_1\times T_2$ and its corresponding particle configuration below, with pairings represented by lines joining the particles:
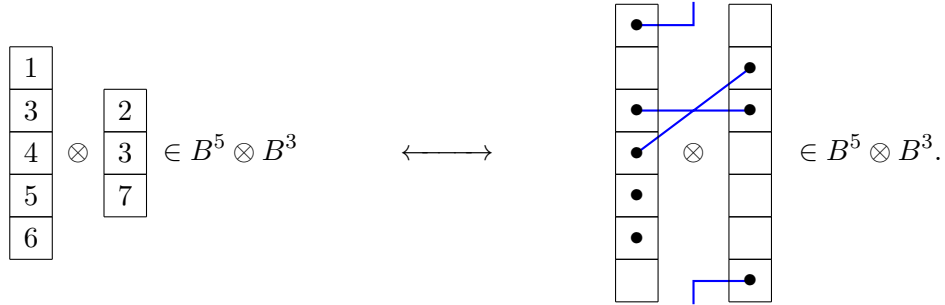
\[\vcenter{\hbox{\tableau{1\\3\\4\\5\\6}}}\;\otimes\;\vcenter{\hbox{\tableau{2\\3\\7}}}\; \in B^{5}\otimes B^{3} \hspace{1.25cm} \xleftrightarrow{\hspace{1cm}} \hspace{1.25cm} \vcenter{\hbox{
 \begin{tikzpicture}[scale=0.5]
         \draw[blue, thick] (0,0)--(3,2.25) (0,1.125)--(3,1.125) (0,3.385)--(1.5,3.385)--(1.5,4) (1.5,-4)--(1.5,-3.385)--(3,-3.385);
        \node at (0,0) {\tableau{\bullet \\ \  \\ \bullet \\ \bullet \\ \bullet \\ \bullet\\ \ }};
        \node at (1.5,0) {$\otimes$};
        \node at (3,0) {\tableau{\  \\ \bullet \\ \bullet  \\ \  \\ \  \\\ \\ \bullet}};

    \end{tikzpicture}
}} \in B^{5}\otimes B^{3}.
\]

Transferring the unpaired particles from $T_1$ to $T_2$, we obtain $R(T_1\otimes T_2)=S_1\otimes S_2$:
$$R\left( \;\,\vcenter{\hbox{\tableau{1\\3\\4\\5\\6}}}\;\otimes\;\vcenter{\hbox{\tableau{2\\3\\7}}}\; \right) = \;\vcenter{\hbox{
 \begin{tikzpicture}[scale=0.5]
        \draw[blue, thick] (0,0)--(3,2.25) (0,1.125)--(3,1.125) (0,3.385)--(1.5,3.385)--(1.5,4) (1.5,-4)--(1.5,-3.385)--(3,-3.385);   
        
        \node at (0,0) {\tableau{\bullet \\ \  \\ \bullet \\ \bullet  \\ \  \\ \ \\ \ }};
        \node at (1.5,0) {$\otimes$};
        \node at (3,0) {\tableau{\  \\ \bullet \\ \bullet  \\ \  \\ \bullet  \\\bullet \\ \bullet}};

    \end{tikzpicture}
}} \; = \;\vcenter{\hbox{\tableau{1\\3\\4}}}\;\otimes\;\vcenter{\hbox{\tableau{2\\3\\5\\6\\7}}}\; \in B^{3}\otimes B^{5}. $$
To compute $R(S_1\otimes S_2)$, the pairing procedure yields the same set of pairings above, so the unpaired particles in $S_2$ return to their original positions in the first factor, recovering $R(S_1\otimes S_2)=T_1\otimes T_2$.
\end{example}

\subsubsection{KN columns of type $C$}

Fix $n\in \NN$ and let $I=\{0,1,\ldots,n\}$. We use the level-zero projection of the affine root system of type $C_{n}^{(1)}$ on the lattice $\Lambda=\ZZ^n$. The simple roots are $\alpha_i = \vec{e}_i - \vec{e}_{i+1}$ when  $i\in I\setminus\{0,n\}$, $\alpha_n=2\vec{e}_n$, and $\alpha_0=-2\vec{e}_1$.

The affine Weyl group $W(C_{n}^{(1)})$ is the Coxeter group generated by $\{s_0,s_1,\ldots,s_{n-1},s_n\}$ subject to the relations $s_i^2 = 1$ for all $i\in I$, $s_is_{i+1}s_i = s_{i+1}s_is_{i+1}$ for $1\leq i \leq n-2$, 
$s_is_j = s_js_i \text{ for all }|i-j|>1$, and $(s_0s_1)^4 = (s_{n-1}s_n)^4 = 1$. The level-zero action of $W(C_{n}^{(1)})$ on the lattice $\Lambda$ is defined by 
\[
s_i(\mu_1,\ldots,\mu_n) = \begin{cases}
    (\mu_1,\ldots,\mu_{i+1},\mu_i,\ldots,\mu_n) &\text{ if } i\in I\setminus\{0,n\}, \\
    (\mu_1,\ldots,\mu_{n-1},-\mu_n) &\text{ if } i=n \\
    (-\mu_1,\mu_2,\ldots,\mu_n) &\text{ if } i=0
\end{cases}
\]
Under this action, the generators $s_i$ correspond exactly to the $\tasepC$ transitions from \cref{def:TypeC-ParticleSystem}. 

As in type $A$, the usual type $C$ affine action differs by a translation from the level-zero action when defining $s_0$, and is given by $s_0(\mu_1,\ldots,\mu_n)=(-\mu_1+1,\mu_2,\ldots,\mu_n)$. The level-zero action is again obtained by taking the linear part.

We review the construction of a crystal of type $C_n^{(1)}$. For $n\geq 1$, denote the ordered signed alphabet $$\mathcal{A}_n^C =  1<2<\ldots<n<\overline{n} < \overline{n-1}<\ldots<\overline{2}<\overline{1}.$$

\begin{definition}\label{def:KN-cols-type-C}
 Let $k,n\in\NN$. A filling $T$ of a tableau of shape $1^k$ with entries from $\mathcal{A}^C_n$ is a \emph{KN column of type $C$} of size $k$ if 
\begin{itemize}
    \item the entries strictly increase from top to bottom according to the order of the alphabet $\mathcal{A}_n^C$, and 
    \item for every $m=1,2,\ldots,n$, we have $\left| \{ i\in T \colon i\leq m \text{ or } i\geq \overline{m}\} \right| \leq m.$
\end{itemize}
We denote the set of all such columns by $\KN_C(k,n).$ The weight map on $T\in\KN_C(k,n)$ is
\[\wt(T) = \sum_{j\in T}\vec{e}_j,\qquad\text{where }\quad\vec{e}_{\,\overline{j}} := -\vec{e}_{j}.\]
\end{definition}

Now we describe the affine crystal operators on tensor products of KN columns of type $C$. These operators can be interpreted through the realization of type $C$ crystals inside type $A$ crystals via virtualization, introduced in \cite{OSS-VirtualCrystals,Kashiwara96}. For a full construction and a more detailed exposition of the relationship between KR crystals of types $A$ and $C$, we refer the reader to \cite{OSS-VirtualCrystals}.

The subsequent definition of $\overline{\pth}_i(\mathbf{T})$ 
is equivalent to the following construction: read the entries in $T$ according to the order of $\Ac_n^C$, record a ``$)$'' for each $i$ and $\overline{i+1}$ and record a ``$($'' for each $i+1$ and $\overline{i}$. Then, use the bracketing rule to cancel adjacent matched pairs ``$()$''.

\begin{definition}[Type $C$ parenthesis word]
Let $T\in\KN_C(k,n)$ and let $0\leq i\leq n$. The reduced \emph{$i$th parenthesis word} $\overline{\pth}_i(T)$ is defined as
\begin{subequations}
\begin{empheq}[left={\overline{\pth}_i(T)=\empheqlbrace}]{align}
& )
&&\text{if } i\in T,\quad \bar i\notin T,\quad
(i+1\in T)\Longleftrightarrow(\overline{i+1}\in T)\label{d.1a}\tag{c.1a}\\[0.2cm]
&)
&&\text{if }\overline{i+1}\in T,\quad i+1\notin T,\quad
(i\in T)\Longleftrightarrow(\overline{i}\in T)
\label{d.1b}\tag{c.1b}\\[0.2cm]
& (
&&\text{if } \overline{i}\in T,\quad i\notin T,\quad
(i+1\in T)\Longleftrightarrow(\overline{i+1}\in T)\label{d.2a}\tag{c.2a}\\[0.2cm]
&(
&&\text{if }i+1\in T,\quad \overline{i+1}\notin T,\quad
(i\in T)\Longleftrightarrow(\overline{i}\in T)
\label{d.2b}\tag{c.2b}\\[0.2cm]
& ))
&&\text{if } i,\overline{i+1}\in T,\ \overline{i},i+1\not\in T
\label{d.3}\tag{c.3}\\[0.2cm]
& ((
&&\text{if } \overline{i},i+1\in T,\ i,\overline{i+1}\not\in T
\label{d.4}\tag{c.4}\\[0.2cm]
& )(
&&\text{if } i+1,\overline{i+1}\in T,\ i,\overline{i}\not\in T
\label{d.5}\tag{c.5}\\[0.2cm]
& \emptyset
&&\text{otherwise.}\label{d.6}\tag{c.6}
\end{empheq}
\end{subequations}
By convention, for a KN column $T$ with entries in $\Ac_n^C$, we assume $i,\bar i\not\in T$ for $i\in\{0,n+1\}$. 
\end{definition}

\begin{definition}[Affine crystal operators for type $C$ columns]
\label{def:crystal-operators-type-C}
    Let $T\in\KN_C(k,n)$, and let $0\leq i\leq n$. The action of the crystal operators $f_i$ and $e_i$ on $T$ is determined by $\pth_i(T)$: 
    \begin{itemize}
    \item If $\pth_i(T)$ has no closed parentheses, $f_i(T)=\mathbf{0}$. Otherwise, $f_i(T)$ is the KN column obtained by increasing the entry corresponding to the rightmost ``$)$'' by one position to the right in the $\mathcal{A}^C_n$ order if $i>0$, and by sending $\bar 1$ to $1$ and reordering the column if $i=0$. In particular, $\pth_i(f_i(T))$ is corresponds to switching the rightmost ``$)$'' to a ``$($''.
     \item If $\pth_i(T)$ has no open parentheses, $e_i(T)=\mathbf{0}$. Otherwise, $e_i(T)$ is the KN column obtained by decreasing the entry corresponding to the leftmost ``$)$'' by one position to the right in the $\mathcal{A}^C_n$ order if $i>0$, and by sending $1$ to $\bar 1$ and reordering the column if $i=0$. In particular, $\pth_i(e_i(T))$ is corresponds to switching the leftmost ``$($'' to a ``$)$''.
     \end{itemize}
\end{definition}

For a tensor product of type $C$ KN columns $\mathbf{T}=T_1\otimes\cdots\otimes T_m$, the crystal operators $\widetilde{f}_i$ and $\widetilde{e}_i$ are defined as in \cref{def:crystal-operators-parentheses} by constructing the $i$th parentheses word of $\mathbf{T}$: let $\overline{\pth}_i(\mathbf{T})$ be the reduced parentheses word obtained from applying the bracketing rule to $\pth_i(\mathbf{T})=\overline{\pth}_i(T_1)\cdots\overline{\pth}_i(T_m)$.
Then, 
\begin{equation}\label{eq:ftilde}
\widetilde{f}_i(\mathbf{T})=T_1\otimes\cdots\otimes f_i(T_k)\otimes\cdots \otimes T_m\qquad \text{ and }\qquad
\widetilde{e}_i(\mathbf{T})=T_1\otimes\cdots\otimes e_i(T_\ell)\otimes\cdots \otimes T_m,
\end{equation}
where $T_k$ (resp.~$T_\ell$) is the factor corresponding to the rightmost ``$)$'' (resp.~leftmost ``$($'') in $\overline{\pth}_i(\mathbf{T})$, if such exists, and we set $\widetilde{f}_i(\mathbf{T})=\mathbf{0}$ (resp.~$\widetilde{e}_i(\mathbf{T})=\mathbf{0}$) otherwise.

\begin{example}\label{ex:crystal-ops-type-C}
    Let $n=5$. Let $\mathbf{T}=T_1\otimes T_2\otimes T_3\otimes T_4$ be a tensor product of type $C$ KN columns:
    \[\mathbf{T}\;=\;\raisebox{11pt}{\tableau{2\\\overline{4}\\\overline{1}}}\;\otimes\;\raisebox{18pt}{\tableau{1\\3\\5\\\overline{3}}}\;\otimes\;\raisebox{11pt}{\tableau{2\\4\\5}}\;\otimes\;\raisebox{11pt}{\tableau{2\\3\\\overline{2}}}\,.\]
    To obtain $\widetilde{f}_2(\mathbf{T})$, $\widetilde{e}_2(B)$, $\widetilde{f}_5(\mathbf{T})$, and $\widetilde{f}_0(\mathbf{T})$, we compute the following (using ``$.$'' to separate the factors):
     \begin{alignat*}{2}
        \pth_2(w) &= \;)\;.\;)\;(\;.\;)\;.\;(\;\quad &&\longrightarrow\qquad \overline{\pth}_2(w) = \,)\;.\;)\,.\,.\,(\, \\
        \pth_5(w) &= \,.\,)\,\ \;.\,)\;.\,
        \quad &&\longrightarrow\qquad \overline{\pth}_5(w) =\,.\,)\,\ \;.\,)\;.\,\\
        \pth_0(w) &= \;)\,.\,(\;.\,.\,
        \quad &&\longrightarrow\qquad \overline{\pth}_0(w) =\;)\,.\,(\;.\,.\,
    \end{alignat*}
    In $\overline{\pth}_2(w)$ the rightmost ``$)$'' corresponds to the $\overline{3}$ in $T_2$, and the leftmost ``$($'' to the $\overline{2}$ in $T_4$. Thus 
    $$\widetilde{f_2}(B) \;=\;\raisebox{11pt}{\tableau{2\\\overline{4}\\\overline{1}}}\;\otimes\;\raisebox{18pt}{\tableau{1\\3\\5\\\mathbf{\overline{2}}}}\;\otimes\;\raisebox{11pt}{\tableau{2\\4\\5}}\;\otimes\;\raisebox{11pt}{\tableau{2\\3\\\overline{2}}} \qquad \text{and} \qquad \widetilde{e_2}(B) \;=\;\raisebox{11pt}{\tableau{2\\\overline{4}\\\overline{1}}}\;\otimes\;\raisebox{18pt}{\tableau{1\\3\\5\\\overline{3}}}\;\otimes\;\raisebox{11pt}{\tableau{2\\4\\5}}\;\otimes\;\raisebox{11pt}{\tableau{2\\3\\\mathbf{\overline{3}}}}\,.$$
    The rightmost ``$)$'' in $\overline{\pth}_5(w)$ and $\overline{\pth}_0(w)$ correspond to the $5$ in $T_3$ and  the $\overline{1}$ in $T_1$, respectively. Thus
    $$\widetilde{f_5}(B) \;=\;\raisebox{11pt}{\tableau{\mathbf{1}\\2\\\overline{4}}}\;\otimes\;\raisebox{18pt}{\tableau{1\\3\\5\\\overline{3}}}\;\otimes\;\raisebox{11pt}{\tableau{2\\4\\\mathbf{\overline{5}}}}\;\otimes\;\raisebox{11pt}{\tableau{2\\3\\\overline{2}}}\qquad\text{and}\qquad
    \widetilde{f_0}(B) \;=\;\raisebox{11pt}{\tableau{\mathbf{1}\\2\\\overline{4}}}\;\otimes\;\raisebox{18pt}{\tableau{1\\3\\5\\\overline{3}}}\;\otimes\;\raisebox{11pt}{\tableau{2\\4\\5}}\;\otimes\;\raisebox{11pt}{\tableau{2\\3\\\overline{2}}}\,.$$
\end{example}

The type $C$ multiline queue construction central to this paper is fundamentally based on \emph{virtualization}, first appearing in \cite{Kashiwara96} and defined explicitly for KR crystals in \cite{OSS-VirtualCrystals}. For KN columns, virtualization is realized by the \emph{splitting map} \cite{Baker00,Lecouvey2002TypeC}, which realizes type $C_n^{(1)}$ crystals as virtual crystals inside type $A_{2n-1}^{(1)}$. In particular, the virtual crystal operators are realized virtually as 
\[f^C_i\mapsto f^A_i f^A_{2n-i}\quad \text{for } 1\leq i\leq n-1,\qquad f^C_n\mapsto (f^A_n)^2,\qquad \text{and}\quad f^C_0\mapsto (f^A_0)^2,
\]where the superscripts denote the type $C$ and type $A$ operators, respectively. The same formulas hold for the operators $e_i$. 
We recall the splitting map below, and in \cref{sec:typeC-mlqs}, we will use it to define  (double-row) particle configurations which reinterpret the elements of the KR crystal.

\begin{definition}[Column splitting]\label{def:splitting}

    Let $T$ be a column in the alphabet $\mathcal{A}_n^C$. Let $D = \{d_1 > d_2 > \cdots > d_m\}$ be the set of unbarred letters such that both $d_i$ and $\overline{d_i}$ appear in $T$. Further, suppose that there exists a set of $k$ unbarred letters $S=\{s_1 > s_2 > \cdots > s_m\}$ such that: 
    \begin{itemize}
        \item $s_1$ is the greatest letter from $1$ to $n$ that satisfies $s_1 < d_1$ and $s_1,\overline{s_1} \notin T$, and 
        \item for $i=2,\ldots,m$, $s_i$ is the greatest letter from $1$ to $n$ that satisfies $s_i < \min(s_{i-1},d_i)$ and $s_i,\overline{s_i} \notin T$.
    \end{itemize}
    In this case we say that \emph{$T$ can be split}, and its \emph{splitting} is the double-column $T^LT^R$ given by:
    \begin{itemize}
        \item $T^L$ is the column obtained from $T$ by changing each $d_i$ into $s_i$ and reordering if necessary. 
        \item $T^R$ is the column obtained from $T$ by changing each $\overline{d_i}$ into $\overline{s_i}$ and reordering if necessary.
    \end{itemize}
    We denote by $\SplitC(k,n)$ the set of columns of height $k$ in the alphabet $\mathcal{A}_n^C$ that can be split. In particular, the set of KN columns for type $C$ corresponds to the set of columns that can be split and this defines a correspondence
    $$\KN_C(k,n) \leftrightharpoons \SplitC(k,n),$$
    which can be seen from comparing definitions.
\end{definition}

\begin{example}\label{ex:splitting}
    Consider the KN column $T \in \KN_C(7,8)$ shown below. Using the notation from \cref{def:splitting}, we have $d_1 = 5$ and $d_2=2$. Then $s_1 = 3$ is the greatest letter from $1$ to $8$ that satisfies $3<5$ and $3,\overline{3}\notin T.$ Moreover, $s_2 = 1$ is the greatest letter from $1$ to $8$ that satisfies $1<\min(3,2)$ and $1,\overline{1}\notin T.$ Thus, $T^L$ is obtained from $T$ by changing $5$ to $3$ and $2$ to $1$ and reordering the column. Similarly, $T^R$ is obtained from $T$ by changing $\overline{5}$ to $\overline{3}$ and $\overline{2}$ to $\overline{1}$; in this case reordering is not needed.
    $$T \, = \, \raisebox{42.5pt}{\tableau{2\\4\\ 5 \\ 8 \\ \overline{6} \\ \overline{5} \\ \overline{2}}} \hspace{1.25cm} \xleftrightarrow{\hspace{1cm}} \hspace{1cm} T^LT^R \, = \, \raisebox{42.5pt}{\tableau{1&2 \\ 3&4 \\ 4&5 \\ 8&8 \\ \overline{6}&\overline{6} \\ \overline{5}&\overline{3} \\ \overline{2}&\overline{1}}} $$
\end{example}

\subsection{Multiline queues of type $A$}

A combinatorial formula to compute the stationary distribution of the TASEP on a circle is given by the \emph{multiline queues} introduced by Ferrari and Martin \cite{FM07}. For a partition $\lambda$ with $L=\lambda_1$ and an integer $n\geq \ell(\lambda)$, define the set of multiline queues 
\[\MLQ_{A}(\lambda,n)=\{(B_1,\ldots,B_L)\colon B_i\subseteq[n]\,,\,|B_i|=\lambda'_i\ \text{ for } i=1,2,\ldots,L\}.
\]
By interpreting each $B_i$ as the set of column indices containing particles in a row of length $n$, we identify $\MLQ_{A}(\lambda,n)$ with the tensor product of type $A_{n-1}^{(1)}$ KR crystals $B_{\lambda}$. This diagrammatic interpretation is a $90^\circ$ rotation of the KN-column convention; hence from now on, we will refer to the tensor factors as \emph{type $A$ rows}.

The \emph{Ferrari--Martin (FM) algorithm} \cite{FM07} defines a projection map $$\proj_{\FM}\colon\MLQ_{A}(\lambda,n)\to\States_{A}(\lambda,n),$$ 
The algorithm proceeds by sequentially pairing particles in adjacent rows, from top to bottom, and labeling the particles being paired. The projection is a word read off the labels of the particles in the bottom row once the procedure has been completed.

\begin{definition}[Ferrari--Martin (FM) algorithm]\label{def:FM-algorithm}
Let $M=(B_1,\ldots,B_L)\in\MLQ_{A}(\lambda,n)$ be a multiline queue. We produce a labeling on $M$ 
as follows. 
\begin{enumerate}
\item For $r=L,L-1,\ldots,2$, begin by assigning to all unlabeled particles in row $r$ the label $r$.  Then, for $\ell=L,\ldots,r$, cylindrically pair each particle with the label $\ell$ to the first unlabeled particle in row $r-1$ weakly to its right, and assign the label $\ell$ to the paired particle. 
\item For $r=1$, assign any remaining unpaired particles in row 1 the label $1$. 
\end{enumerate}
The projection $\proj_{\FM}(M)$ is the word obtained by reading off the labels of the particles in row 1, from left to right, and recording a 0 for each vacancy.
\end{definition}

\begin{prop}[{\cite[Theorems 2.2 and 3.1]{FM07}}]
\label{prop:MLQs-typeA}
    There exists a Markov chain $\mlqAFM(\lambda,n)$ on $\MLQ_A(\lambda,n)$ with the following properties.
    \begin{itemize}
        \item[i.] The stationary distribution of $\mlqAFM(\lambda,n)$ is the uniform distribution.
        \item[ii.] $\tasepA(\lambda,n)$ is a lumping of $\mlqAFM(\lambda,n)$.
    \end{itemize}
    In particular, the stationary distribution of $\tasepA(\lambda,n)$ is given by
    \[\pi(\tau)=\frac{1}{|\MLQ_{A}(\lambda,n)|}|\{B\in\MLQ_{A}(\lambda,n)\colon\proj_{\FM}(B)=\tau\}|.\]
\end{prop}

\begin{example}
    For $\lambda = (2,1)$ and $n=3$, the projections of the elements in $\MLQ_{A}(\lambda,n)$ are shown in \cref{fig:exMLQs-typeA}, and the stationary distribution as given by \cref{prop:MLQs-typeA} indeed coincides with  he computation in \cref{ex:tasepA}.

    \begin{figure}
        \centering
        \resizebox{0.9\linewidth}{!}{
        \begin{tikzpicture}[scale=0.8]
            \def \w{1};
            \def \h{1};
            \def \r{0.25};
            \def \ww{4.5};

            \foreach \i in {0,...,4} {
            \draw[gray!50] (4.5*\i-.75,-3)--(4.5*\i-.75,3);
            }
        
            \node at (-3+\ww,3) {\large$(0,2,1)$};
            \node at (-3,3) {\large$(0,1,2)$};
            \node at (-3+5*\ww,3) {\large$(2,1,0)$};
            \node at (-3+3*\ww,3) {\large$(1,2,0)$};
            \node at (-3+2*\ww,3) {\large$(1,0,2)$};
            \node at (-3+4*\ww,3) {\large$(2,0,1)$};

            \begin{scope}[xshift=18cm,yshift=-3cm]
            \foreach \i in {0,...,2}
            {
            \draw[gray!50] (0,\i*\h)--(\w*3,\i*\h);
            }
            \foreach \i in {0,...,3}
            {
            \draw[gray!50] (\w*\i,0)--(\w*\i,2*\h);
            }
            \foreach \xx\yy\i in {0/0/2,1/0/1,0/1/2}
            {
            \draw (\w*.5+\w*\xx,\h*.5+\h*\yy) circle (\r cm);
            \node[font=\scriptsize] at (\w*.5+\w*\xx,\h*.5+\h*\yy) {\i};
            }
            \draw[black] (\w*0.5,\h*1.5-\r)--(\w*0.5,\h*.5+\r);
            \end{scope}
        
            \begin{scope}[xshift=13.5cm,yshift=0cm]
            \foreach \i in {0,...,2}
            {
            \draw[gray!50] (0,\i*\h)--(\w*3,\i*\h);
            }
            \foreach \i in {0,...,3}
            {
            \draw[gray!50] (\w*\i,0)--(\w*\i,2*\h);
            }
            \foreach \xx\yy\i in {0/0/2,2/0/1,0/1/2}
            {
            \draw (\w*.5+\w*\xx,\h*.5+\h*\yy) circle (\r cm);
            \node[font=\scriptsize] at (\w*.5+\w*\xx,\h*.5+\h*\yy) {\i};
            }
            \draw[black] (\w*0.5,\h*1.5-\r)--(\w*0.5,\h*.5+\r);
            \end{scope}
        
            \begin{scope}[xshift=0cm,yshift=0cm]
            \foreach \i in {0,...,2}
            {
            \draw[gray!50] (0,\i*\h)--(\w*3,\i*\h);
            }
            \foreach \i in {0,...,3}
            {
            \draw[gray!50] (\w*\i,0)--(\w*\i,2*\h);
            }
            \foreach \xx\yy\i in {1/0/2,2/0/1,0/1/2}
            {
            \draw (\w*.5+\w*\xx,\h*.5+\h*\yy) circle (\r cm);
            \node[font=\scriptsize] at (\w*.5+\w*\xx,\h*.5+\h*\yy) {\i};
            }
            \draw[black] (\w*.5,\h*1.5-\r)--(\w*.5,\h*0.9)--(\w*1.5,\h*0.9)--(\w*1.5,\h*.5+\r);
            \end{scope}
        
        
            \begin{scope}[xshift=9cm,yshift=0cm]
            \foreach \i in {0,...,2}
            {
            \draw[gray!50] (0,\i*\h)--(\w*3,\i*\h);
            }
            \foreach \i in {0,...,3}
            {
            \draw[gray!50] (\w*\i,0)--(\w*\i,2*\h);
            }
            \foreach \xx\yy\i in {0/0/1,1/0/2,1/1/2}
            {
            \draw (\w*.5+\w*\xx,\h*.5+\h*\yy) circle (\r cm);
            \node[font=\scriptsize] at (\w*.5+\w*\xx,\h*.5+\h*\yy) {\i};
            }
            \draw[black] (\w*1.5,\h*1.5-\r)--(\w*1.5,\h*.5+\r);
            \end{scope}
        
            \begin{scope}[xshift=4.5cm,yshift=0cm]
            \foreach \i in {0,...,2}
            {
            \draw[gray!50] (0,\i*\h)--(\w*3,\i*\h);
            }
            \foreach \i in {0,...,3}
            {
            \draw[gray!50] (\w*\i,0)--(\w*\i,2*\h);
            }
            \foreach \xx\yy\i in {0/0/1,2/0/2,1/1/2}
            {
            \draw (\w*.5+\w*\xx,\h*.5+\h*\yy) circle (\r cm);
            \node[font=\scriptsize] at (\w*.5+\w*\xx,\h*.5+\h*\yy) {\i};
            }
            \draw[black] (\w*1.5,\h*1.5-\r)--(\w*1.5,\h*0.9)--(\w*2.5,\h*0.9)--(\w*2.5,\h*.5+\r);
            \end{scope}
        
            \begin{scope}[xshift=0cm,yshift=-3cm]
            \foreach \i in {0,...,2}
            {
            \draw[gray!50] (0,\i*\h)--(\w*3,\i*\h);
            }
            \foreach \i in {0,...,3}
            {
            \draw[gray!50] (\w*\i,0)--(\w*\i,2*\h);
            }
            \foreach \xx\yy\i in {2/0/1,1/0/2,1/1/2}
            {
            \draw (\w*.5+\w*\xx,\h*.5+\h*\yy) circle (\r cm);
            \node[font=\scriptsize] at (\w*.5+\w*\xx,\h*.5+\h*\yy) {\i};
            }
            \draw[black] (\w*1.5,\h*1.5-\r)--(\w*1.5,\h*.5+\r);
            \end{scope}
        
        
            \begin{scope}[xshift=18cm,yshift=0cm]
            \foreach \i in {0,...,2}
            {
            \draw[gray!50] (0,\i*\h)--(\w*3,\i*\h);
            }
            \foreach \i in {0,...,3}
            {
            \draw[gray!50] (\w*\i,0)--(\w*\i,2*\h);
            }
            \foreach \xx\yy\i in {0/0/2,1/0/1,2/1/2}
            {
            \draw (\w*.5+\w*\xx,\h*.5+\h*\yy) circle (\r cm);
            \node[font=\scriptsize] at (\w*.5+\w*\xx,\h*.5+\h*\yy) {\i};
            }
            \draw[black,-stealth] (\w*2.5,\h*1.5-\r)--(\w*2.5,\h*0.9)--(\w*3.3,\h*0.9);
            \draw[black] (-.3,\h*0.9)--(\w*0.5,\h*0.9)--(\w*0.5,\h*.5+\r);
            \end{scope}
        
            \begin{scope}[xshift=4.5cm,yshift=-3cm]
            \foreach \i in {0,...,2}
            {
            \draw[gray!50] (0,\i*\h)--(\w*3,\i*\h);
            }
            \foreach \i in {0,...,3}
            {
            \draw[gray!50] (\w*\i,0)--(\w*\i,2*\h);
            }
            \foreach \xx\yy\i in {0/0/1,2/0/2,2/1/2}
            {
            \draw (\w*.5+\w*\xx,\h*.5+\h*\yy) circle (\r cm);
            \node[font=\scriptsize] at (\w*.5+\w*\xx,\h*.5+\h*\yy) {\i};
            }
            \draw[black] (\w*2.5,\h*1.5-\r)--(\w*2.5,\h*.5+\r);
            \end{scope}
        
            \begin{scope}[xshift=-4.5cm,yshift=0cm]
            \foreach \i in {0,...,2}
            {
            \draw[gray!50] (0,\i*\h)--(\w*3,\i*\h);
            }
            \foreach \i in {0,...,3}
            {
            \draw[gray!50] (\w*\i,0)--(\w*\i,2*\h);
            }
            \foreach \xx\yy\i in {1/0/1,2/0/2,2/1/2}
            {
            \draw (\w*.5+\w*\xx,\h*.5+\h*\yy) circle (\r cm);
            \node[font=\scriptsize] at (\w*.5+\w*\xx,\h*.5+\h*\yy) {\i};
            }
            \draw[black] (\w*2.5,\h*1.5-\r)--(\w*2.5,\h*.5+\r);
            \end{scope}
        
        \end{tikzpicture}
        }
        \caption{Elements of $\MLQ_{A}((2,1),3)$ with their corresponding projections obtained by the FM algorithm.}
        \label{fig:exMLQs-typeA}
    \end{figure}
\end{example}

The Ferrari--Martin algorithm admits a reformulation in terms of iterated combinatorial
$R$ matrices acting on the multiline queue viewed as a tensor product of type $A$ rows, as first observed by Kuniba--Maruyama--Okado \cite{KMO15}. More precisely, the cylindrical pairing rule for pairing particles with a fixed label between adjacent rows in \cref{def:FM-algorithm} yields exactly the same output as the Nakayashiki--Yamada rule for the combinatorial $R$ matrix on the corresponding KN columns \cref{def:NY-rule}. Iterating this operation from a given row of a multiline queue down to the bottom row, and repeating for row 1 through $L$, yields a triangular array of iterated combinatorial $R$ computations, which is called the \emph{corner transfer matrix} description of the projection.

\begin{defn}[Corner transfer matrix]\label{def:ctm_typeA}
     For a pair of columns $T$ and $S$ such that $R(T\otimes S)=T'\otimes S'$, denote the output of the first factor of the combinatorial $R$ by $R_{(1)}(T\otimes S)=T'$. This is depicted diagrammatically  with a crossing representing the action of $R$:
     \begin{center}
        \begin{tikzpicture}
            \draw[thick,->] (0,0)--(1,0); 
         \draw[thick,->] (.5,-.5)--(.5,.5);  
            \node at (-0.5,0) {$S$};
            \node at (0.5,-.75) {$T$};
            \node at (1.5,0) {$T'$};
            \node at (0.5,.75) {$S'$};
        \end{tikzpicture}
    \end{center}
      Let $M=(B_1,\ldots,B_L)$ be a multiline queue. Then define the computation at level $j$:
    \[\pi_1(M)=\wt(B_1)\quad \text{and}\quad \pi_j(M)=\wt\left(R_{(1)}(B_1\otimes\cdots R_{(1)}(B_{j-2}\otimes R_{(1)}(B_{j-1}\otimes B_j))\cdots)\right)\ \text{for $j\geq 2$}.\]
    Diagrammatically, this composition of combinatorial $R$ matrices is represented by the picture below, where $R$ is applied at each crossing:
    \begin{center}
        \begin{tikzpicture}
            \draw[thick] (0,0)--(1,0); 
            \draw[dashed] (1,0)--(2,0);
            \draw[thick,->] (2,0)--(4,0);  \draw[thick,->] (.5,-.5)--(.5,.5);  \draw[thick,->] (2.5,-.5)--(2.5,.5); 
             \draw[thick,->] (3.5,-.5)--(3.5,.5);
            \node at (-0.5,0) {$B_j$};
            \node at (0.5,-1) {$B_{j-1}$};
            \node at (3,-1) {$B_{2}$};
            \node at (4,-1) {$B_{1}$};
            \node at (5,0) {$\pi_j(M)$};
        \end{tikzpicture}
    \end{center}
    Define the projection map to be the sum of these computations over all levels:
    \[\pi(M)=\sum_{j=1}^{L}\pi_j(M).\]
\end{defn}

\begin{prop}[{\cite{KMO15}}]\label{prop:proj equals pi}
For a multiline queue $M$, one has
$\pi(M)=\proj_{\FM}(M)$.
\end{prop}

\begin{example}
    Let $M = \left(\{2,3,4,5\},\{1,2,5\},\{3,4,5\},\{1,3\}\right) \in \MLQ_{A}((4,4,3,1),5).$ From the FM algorithm  we obtain the labeling below, which yields $\proj_{\FM}(M)=(0,4,3,1,4)$.
    
    \begin{center}
        \resizebox{.3\linewidth}{!}{
        \begin{tikzpicture}[scale=0.7]
        \def \w{1};
        \def \h{1};
        \def \r{0.25};
        
        \foreach \i in {0,...,4}
        {
        \draw[gray!50] (0,\i*\h)--(\w*5,\i*\h);
        }
        \foreach \i in {0,...,5}
        {
        \draw[gray!50] (\w*\i,0)--(\w*\i,4*\h);
        }
        \foreach \xx\yy\i in {
        1/0/4,2/0/3,3/0/1,4/0/4,
        0/1/4,1/1/3,4/1/4,
        2/2/4,3/2/4,4/2/3,
        0/3/4,2/3/4}
        {
        \draw (\w*.5+\w*\xx,\h*.5+\h*\yy) circle (\r cm);
        \node[font=\scriptsize] at (\w*.5+\w*\xx,\h*.5+\h*\yy) {\i};
        }
    
        \draw[blue] (\w*.5,\h*1.5-\r)--(\w*.5,\h*0.9)--(\w*1.5,\h*0.9)--(\w*1.5,\h*.5+\r);
        \draw[red] (\w*1.5,\h*1.5-\r)--(\w*1.5,\h*1.1)--(\w*2.5,\h*1.1)--(\w*2.5,\h*.5+\r);
        \draw[black] (\w*4.5,\h*1.5-\r)--(\w*4.5,\h*.5+\r);
    
        \draw[black] (\w*2.5,\h*2.5-\r)--(\w*2.5,\h*1.9)--(\w*4.5,\h*1.9)--(\w*4.5,\h*1.5+\r);
        \draw[blue,-stealth] (\w*3.5,\h*2.5-\r)--(\w*3.5,\h*2)--(\w*5.3,\h*2);
        \draw[blue] (-.3,\h*2)--(\w*0.5,\h*2)--(\w*0.5,\h*1.5+\r);
        \draw[red,-stealth] (\w*4.5,\h*2.5-\r)--(\w*4.5,\h*2.15)--(\w*5.3,\h*2.15);
        \draw[red] (-.3,\h*2.15)--(\w*1.5,\h*2.15)--(\w*1.5,\h*1.5+\r);
    
        \draw[black] (\w*0.5,\h*3.5-\r)--(\w*0.5,\h*2.9)--(\w*2.5,\h*2.9)--(\w*2.5,\h*2.5+\r);
        \draw[blue] (\w*2.5,\h*3.5-\r)--(\w*2.5,\h*3.1)--(\w*3.5,\h*3.1)--(\w*3.5,\h*2.5+\r);
    
        \end{tikzpicture}
        }
    \end{center}   
    
    For the same multiline queue, we show the computation of the projection map $\pi(M)$ via the corner transfer matrix. The computations at level $j$, for $j=2,3,4$, are performed as follows:
    \allowdisplaybreaks
    \begin{align*}
        \pi_2(M) &= \wt(R_{(1)}(B_1\otimes B_2)) = \wt(\{2,3,5\})
        &= (0,1,1,0,1)\\
        \pi_3(M) &= \wt(R_{(1)}(B_1\otimes R_{(1)}(B_2\otimes B_3))) = \wt(R_{(1)}(B_1\otimes\{1,2,5\})) 
        &= (0,1,1,0,1)\\
        \pi_4(M) &= \wt(R_{(1)}(B_1\otimes R_{(1)}(B_2\otimes R_{(1)}(B_3\otimes B_4)))) &\\
        &= \wt(R_{(1)}(B_1\otimes R_{(1)}(B_2\otimes \{ 3,4 \}))) = \wt(R_{(1)}(B_1\otimes \{ 1,5 \}))
        &= (0,1,0,0,1)
    \end{align*}
    With $\wt(B_1)= (0,1,1,1,1)$, the corner transfer matrix projection of $M$ is thus
    \begin{align*}
        \pi(M) = (0,1,1,1,1) + (0,1,1,0,1)+ (0,1,1,0,1)+(0,1,0,0,1)
        = \proj_{\FM}(M).
    \end{align*}
\end{example}

\begin{remark}
   It is useful to interpret $\pi$ as follows: if $b$ is the set of particles in row $r$ having labels at least $\ell$, then $R_{(1)}(B_{r-1}\otimes b)$ is exactly the set of particles in row $r-1$ with labels at least $\ell$. Iterating downward from row $j$ to row $1$ therefore yields the set of sites containing particles with labels at least $\ell$ in the output of $\pi$. Repeating for all rows $j$ thus recovers the output of the FM algorithm. 
\end{remark}

\section{Diagrammatic interpretation for KN columns of type $C$ }\label{sec:typeC-mlqs}

By interpreting the splitting of a KN column as a configuration of particles, we construct \emph{type $C$ rows}, which will be used as the building blocks of a \emph{type $C$ multiline queue}.

\begin{definition}\label{def:TypeC-Rows-particles} 
    Let $k$ and $n$ be natural numbers with $n\geq k$. A \emph{configuration of type $C$ of size $n$} is a $2\times n$ array 
    \[Q=\begin{pmatrix}a_1&a_2&\cdots &a_n\\ b_1&b_2&\cdots& b_n\end{pmatrix},\qquad a_i,b_i\in\{\,\cir\;,\,\squ\;,\,\Cdot\,\}\text{ for each }i\in[n].
    \] 
    If a site contains `$\cdot$', it is consider empty, and otherwise it contains a particle, represented by $\cir$ or $\squ$. If $Q$ is a configuration of type $C$, we say it is a \emph{type $C$ row} if it satisfies the following conditions: \\
    
    \begin{enumerate}[itemsep=3mm,leftmargin=3em]
        \item[(R1)] Each column is either fully vacant or contains two particles: $a_i\in\{\cir,\squ\}$ if and only if $b_i\in\{\cir,\squ\}$ for all $i$. 
        \item[(R2)] Let $\pth(Q)=w_1\cdots w_n$ be the word in the symbols $\{\,(\,,\,)\,,0\,,\,\emptyset\}$ defined as
        \[w_i=(\text{ if }\binom{a_i}{b_i}=\sqcir,\qquad
            w_i=)\text{ if }\textbf{}=\cirsq,\qquad
            w_i=0\text{ if }\binom{a_i}{b_i}=\emp,
        \]
        and $w_i=\emptyset$ otherwise.
        Then, after matching the open and closed parentheses according to the the bracketing rule from \cref{def:bracketing-rule}, we require that all parentheses in $\pth(Q)$ are matched, and there are no $0$'s in between pairs of matched parentheses.
    \end{enumerate}
    \vspace{1em}
    We denote the set of all type $C$ rows of size $n$ with exactly $k$ nonempty columns by $\RowC(k,n).$
\end{definition}

For 
$Q=\begin{pmatrix}
    a_1&\cdots&a_n\\
    b_1&\cdots&b_n\\
\end{pmatrix}\in\RowC(k,n)$, the \emph{top} and \emph{bottom} components are denoted $Q_T=a_1\cdots a_n$ and $Q_B=b_1\cdots b_n$, respectively.  The \emph{column} $j$ in $Q$ refers to the pair of particles $\begin{pmatrix}
    a_j\\
    b_j\\
\end{pmatrix}$, and we denote by $Q_{T,j}=a_j$ and $Q_{B,j}=b_j$ the corresponding entries in this column. For $X\in\{T,B\}$, if $Q_{X,i}=\cdot$, we say the site is \emph{empty} or \emph{vacant}, and if $Q_{X,i}\in\{\cir,\squ\}$, we say it is occupied by a particle. 

\begin{definition}\label{def:row}
    Define the map 
    \[\row\colon\KN_C(k,n)\to\RowC(k,n)
    \]
    as follows. Let $T^LT^R=\spl(T)$ be the splitting of $T\in\KN_C(k,n)$. Then define $Q\in\RowC(k,n)$ by setting $Q_B=(b_1^L,\ldots, b_n^L)$ and $Q_T=(b_1^R,\ldots,b_n^R)$, where 
    \[b_i^X=\begin{cases}
        \cir&\text{if }i\in T^X\\\squ&\text{if }\bar i\in T^X\\\Cdot&\text{else}
    \end{cases}\,,\qquad\qquad X\in\{L,R\}.
    \]
    Define the inverse map $\row^{-1}$ by constructing a set $S\subseteq \Ac_n^C$ by scanning $Q=(b_1,\ldots,b_n)\in\RowC(k,n)$, and for $1\leq i\leq n$, if $b_{i,T}=\cir$, add $i$ to $S$, and if $b_{i,B}=\squ$, add $\bar{i}$ to $S$.
Then, $\row^{-1}(Q)$ is obtained by rearranging the letters in $S$ according to the order on $\Ac_n^C$.
\end{definition}

\begin{lemma}\label{lem:bijection_KN-typeC_rows}
    The maps $\row$ and $\row^{-1}$ are well-defined, and are mutual inverses.
\end{lemma}

\begin{proof}

    Let $T\in \KN_C(k,n)$. We first show that $Q = \row(T) \in \RowC(k,n)$. Recall the sets $D$ and $S$ from \cref{def:splitting}: $D=\{1\leq i\leq n:i,\bar i\in T\}$, and $S$ is the set of the letters to which the letters of $D$ are changed to produce $\spl(T)=T^LT^R$, by mapping $d_i\mapsto s_i$ to obtain $T^L$ from $T$, and $\overline{d_i}\mapsto \overline{s_i}$ to obtain $T^R$ from $T$.  
    We analyze the cases for the different letters appearing in $T^LT^R$:
    \begin{enumerate}[label=(\alph*)]
        \item If a letter is in $T$ but not in $D$, the corresponding symbol, $\cir$ or $\squ$, appears in both $Q_B$ and $Q_T$. 
        \item For each $d_k \in D$, $T^L$ contains $\overline{d_k}$ and $s_k$, while $T^R$ contains $d_k$ and $\overline{s_k}$; this means $Q_B$ has a $\squ$ in position $d_k$ and a $\cir$ in position $s_k$, while $Q_T$ has a $\squ$ in position $s_k$ and a $\cir$ in position $d_k$. 
    \end{enumerate}
    This shows that $Q$ satisfies (R1) from \cref{def:TypeC-Rows-particles}. Moreover, from the definition of the splitting procedure, notice that the columns of $Q$ with the configurations $\cirsq$ and $\sqcir$ correspond to the elements of $D$. Indeed, each of these elements creates one such pair as described in (b) from the cases before. Thus, all parenthesis from $\pth(Q)$ are matched. Furthermore, since the values of $s_i$ are chosen to be maximally to the right and to the left of $d_i$, there are no empty columns in between such pairs of columns. Therefore $Q$ satisfies (R2) from \cref{def:TypeC-Rows-particles}.

    Now let $Q\in \RowC(k,n).$ We show that $T = \row^{-1}(Q) \in \KN_C(k,n).$ Condition (R2) on $Q$ ensures that both $j$ and $\overline{j}$ are added to $T$ only when there is a $\sqcir$ paired with a $\cirsq$ at positions $\ell,j$ respectively, with $\ell<j$. In particular, this implies neither $\ell$ or $\overline{\ell}$ is in $T$. Since this is true for every such pair $j,\overline j$, this guarantees that, for each $k$, the number of letters less than or equal to $k$ in absolute value does not exceed $k$, and hence $T$ satisfies the KN condition from \cref{def:KN-cols-type-C}. 

    We show $\row^{-1}$ is the inverse of $\row$ (and leave the other direction to the reader). Let $T\in\KN_C(k,n)$ with splitting $T^LT^R$ and let $Q=(b_1,\ldots,b_n)=\row(T)$. The analysis above implies $b_i=\cirsq$ if and only if $i,\bar i\in T$. Moreover, $b_i=\circir$ if and only if $i\in T^L,T^R$, which means $i\in T$ and $\bar i\not\in T$; similarly, $b_i=\sqsq$ if and only if $\bar i\in T^L,T^R$, which means $\bar i\in T$ and $i\not\in T$. Thus $\row^{-1}(Q)=T$, as desired.  
\end{proof}

\begin{example}
    Consider the column $T\in\KN_C(7,8)$ and its splitting $T^LT^R$ from \cref{ex:splitting}. 
    \[T \, = \, \raisebox{42.5pt}{\tableau{2\\4\\ 5 \\ 8 \\ \overline{6} \\ \overline{5} \\ \overline{2}}} \hspace{.25cm} \xleftrightarrow{\hspace{.5cm}} \hspace{.25cm} T^LT^R \, = \, \raisebox{42.5pt}{\tableau{1&2 \\ 3&4 \\ 4&5 \\ 8&8 \\ \overline{6}&\overline{6} \\ \overline{5}&\overline{3} \\ \overline{2}&\overline{1}}} \hspace{.25cm}\xlongrightarrow{\quad\row\quad}\hspace{.25cm} \begin{pNiceMatrix}[first-row,last-row]
    {\scriptstyle \bar 1} & {\scriptstyle 2} & {\scriptstyle \bar 3} & {\scriptstyle 4} &
    {\scriptstyle 5} & {\scriptstyle \bar 6} &  & {\scriptstyle 8} \\
    \squ & \cir & \squ & \cir & \cir & \squ & \cdot & \cir \\
    \cir & \squ & \cir & \cir & \squ & \squ & \cdot & \cir \\
    {\scriptstyle 1} & {\scriptstyle \bar2} & {\scriptstyle 3} & {\scriptstyle 4} &
    {\scriptstyle \bar 5} & {\scriptstyle \bar 6} & & {\scriptstyle 8}
\end{pNiceMatrix}\]
    Indeed, the columns containing $\cir$ and $\squ$ in $(\row(T)_B,\row(T)_T)$ are in correspondence with $(T^L,T^R)$, in accordance with \cref{def:row}.
\end{example}

\begin{definition}[Minimal balanced block]\label{def:min_balanced_block}
    Let $Q=(c_1,\ldots,c_n)$ be a type $C$ row. An interval $[i,j]$ of columns of $Q$ with $i\leq j$ is \emph{balanced} if the particle configuration obtained by replacing the columns $c_k$ with $\emp$ for $i\leq k\leq j$ still has the properties of being a type $C$ row. 
    A balanced interval is a \emph{minimal balanced block} if it is nonempty and does not contain a strictly smaller balanced interval.
    Denote the indices of the leftmost and rightmost columns of the block containing column $j$ in $Q$ by $\ell_j(Q)$ and $r_j(Q)$, respectively.
\end{definition}

\begin{example}\label{ex:blocks}
    Consider the type $C$ row $Q$ from our running example, and its minimal balanced blocks: 
    \vspace{1em}
    \[
    Q=\begin{pNiceMatrix}
        \squ & \cir & \squ & \cir & \cir & \squ & \cdot & \cir \\
        \cir & \squ & \cir & \cir & \squ & \squ & \cdot & \cir
        \CodeAfter
        \OverBrace{1-1}{2-2}{\text{\scriptsize Block 1}}
        \UnderBrace{1-3}{2-5}{\text{\scriptsize Block 2}}
        \OverBrace{1-6}{2-6}{\text{\scriptsize Block 3}}
        \UnderBrace{1-8}{2-8}{\text{\scriptsize Block 4}}
    \end{pNiceMatrix}\;.
    \]
    
    \vspace{1em}
    \noindent Thus we obtain the following values for $\ell_j(Q)$ and $r_j(Q)$: 
    \begin{center}
    \begin{tabular}{c|cccccccc}
         $j$ & 1 & 2 & 3 & 4 & 5 & 6 & 7 & 8 \\[1pt]
         \hline
         \rule{0pt}{4ex}
         $\left(\ell_j(Q),r_j(Q)\right)$ & (1,2) & (1,2) & (3,5) & (3,5) & (3,5) & (6,6) & (7,7) & (8,8) 
    \end{tabular}
\end{center}
\end{example}
For an interval $[a,b]\subseteq[n]$ and a configuration $Q=(c_1,\ldots,c_n)$, define the \emph{imbalance counter}
    \[\lh_{[a,b]}(Q)\coloneq \left|\left\{j\in [a,b]:c_j= \cirsq\right\}\right| \,-\, \left|\left\{j\in [a,b] \colon c_j=\sqcir\right\}\right|.
    \]
It is immediate that every minimal balanced block $[a,b]$ of $Q$ must satisfy $$\lh_{[1,a-1]}(Q)=\lh_{[a,b]}(Q)=\lh_{[b+1,n]}(Q)=0.$$ Moreover, if $j\in[n]$ and $\ell_j(Q)=r_j(Q)=j$, then column $j$ of $Q$ is contained in the set $\left\{\emp,\circir,\sqsq\right\}$. Otherwise, columns $\ell_j(Q)$ and $r_j(Q)$ are equal to $\sqcir$ and $\cirsq$, respectively. 

\begin{lemma}\label{lem:balanced equiv}
    Suppose the interval $[a,b]\subseteq[n]$ contains no empty columns in $Q\in\RowC(k,n)$ and 
    $\lh_{[1,b]}(Q)=\lh_{[a,n]}
    (Q)=0$. The following are equivalent:
    \begin{itemize}
        \item[i.]  The interval $[a,b]$ in $Q$ is a concatenation of minimal balanced blocks.
        \item[ii.] $\lh_{[1,d]}(Q)\leq 0$ for all $a\leq c<b$.
        \item[iii.]$\lh_{[c,n]}(Q)\geq 0$ for all $a< c\leq b$.
    \end{itemize}
    Moreover, $[a,b]$ is a minimal balanced block if and only if equality is never reached in either (ii) or (iii). 
\end{lemma}

\begin{proof}
    When the interval $[a,b]$ contains no empty columns, conditions (ii) and (iii) are equivalent to condition (R2) in \cref{def:TypeC-Rows-particles}, which holds if and only if (i) is true. 

    We prove the second statement. If $a=b$, $[a,b]$ is a minimal balanced block, so the claim is vacuously true. Assume $a<b$. There exists $a<c\leq b$ such that $\lh_{[c,n]}(Q)=0$ if and only if condition (ii) holds for the interval $[c,b]$; equivalently, if and only if condition (iii) holds for $[a,c-1]$. The latter is equivalent to $[a,b]$ containing more than one minimal balanced block. This proves our claim.
\end{proof}

\begin{example}
    For the type $C$ row $Q$ from \cref{ex:blocks}, the values of $\lh_{[1,j]}$ and $\lh_{[j,8]}$ are shown in the table below. 

    \begin{center}
        \begin{tabular}{c|cccccccc}
             $j$ & 1 & 2 & 3 & 4 & 5 & 6 & 7 & 8 \\[3pt]
             \hline
             \rule{0pt}{4ex}
             $\lh_{[1,j]}(Q)$ & -1 & 0 & -1 & -1 & 0 & 0 & 0 & 0 \\[3pt]
             \rule{0pt}{3ex}
             $\lh_{[j,8]}(Q)$ & 0 & 1 & 0 & 1 & 1 & 0 & 0 & 0\\[5pt]
        \end{tabular}
    \end{center}
    
    For an example of \cref{lem:balanced equiv}, consider the interval $[3,6]$, which is balanced, not minimal, and contains no empty columns. We verify the first two conditions of the lemma. Indeed, since $\lh_{[1,6]}(Q) =\lh_{[3,8]}(Q) =  0$, it satisfies the hypothesis of the lemma. In $\lh_{[1,j]}(Q)$ the entries between columns $a=3$ and $b-1=5$ are all nonpositive, which verifies condition (ii). In $\lh_{[j,n]}(Q)$, the entries between columns $a+1 = 4$ and $b=6$ are all nonnegative, verifying (iii). Finally, we see that the interval is not minimal since in both rows, the corresponding intervals contain $0$'s. In particular, one can recover the block decomposition of  \cref{ex:blocks} from either the values $\{\lh_{[1,j]}(Q)\}$ or $\{\lh_{[j,n]}(Q)\}$.
\end{example}

To finish this section, we show that type $C$ rows satisfy a certain \emph{square-circle symmetry}.
\begin{definition}
    For
    $Q=(a_1,\ldots,a_n)\in\RowC(k,n)$, define the \emph{reflection} of $B$ by 
    \[
    \refl(B)\coloneq(\overline{a_n},\ldots,\overline{a_1}),
    \qquad
    \overline{\begin{pmatrix}x\\y\end{pmatrix}}
    \coloneq
    \begin{pmatrix}\overline{x}\\ \overline{y}\end{pmatrix},
    \qquad
    \overline{j}\coloneq \begin{cases}
        \squ&\text{ if }j=\cir\\
        \cir&\text{ if }j=\squ\\
        \cdot&\text{ if }j=\cdot
    \end{cases} \qquad \text{for }j\in\{\cir\,,\,\squ\,,\,\Cdot\,\}.
\]
\end{definition}

\begin{lemma}\label{lem:square circle symmetry}
    Reflection is an involution on type $C$ rows:
    \[\refl\colon\RowC(k,n)\to\RowC(k,n).
    \]
\end{lemma}

\begin{proof}
    By construction, $\refl$ is an involution that preserves the number of particles. We show $B'=\refl(B)\in\RowC(k,n)$. Let $[a,b]$ be a maximal interval containing no empty columns in $B'$. Then $[n-b+1,n-a+1]$ is a maximal interval containing no empty columns in $B$, and is hence balanced. The interval $[n-b+1,n-a+1]$ in $B$ therefore satisfies item (ii) of \cref{lem:balanced equiv}, and since $\lh_{[i,j]}(B)=\lh_{[n-j+1,n-i+1]}(B')$, the interval $[a,b]$ in $B'$ satisfies item (iii). Since this holds for every maximal interval with no empty columns in $B'$, the claim is proved. 
\end{proof}

\begin{remark}
    Reflection is equivalent to taking the dual of the corresponding type $C$ crystal. The dual crystal is obtained by swapping $f_i$ and $e_i$ and reversing the order of the tensor factors \cite{BumpSchilling17}. Under the map $\row$, this agrees with reversing the type $C$ row and swapping the roles of the $\squ$s and $\cir$s.
\end{remark}

\subsection{Tensor products of type $C$ rows}

By \cref{lem:bijection_KN-typeC_rows}, tuples of type $C$ rows may be identified with tensor products of KN columns of type $C$, from which we obtain a crystal structure equivalent to the one on tensors of KN columns via the bijection $\row$. We will give a purely diagrammatic description in terms of our objects.

\begin{definition}
    Let $\alpha=(\alpha_1,\ldots,\alpha_L)$ be a composition and $n\geq 1$ be an integer. Denote by $\RowC(\alpha,n)$ the tensor product of type $C$ rows in $n$ columns with the numbers of particles in each factor given by $\alpha$: $$\RowC(\alpha,n) \coloneqq \bigotimes_{i=1}^L \RowC(\alpha_i,n) = \RowC(\alpha_1,n) \otimes\RowC(\alpha_2,n) \otimes \cdots \otimes \RowC(\alpha_L,n).$$
\end{definition}

With these tensor products we can now define our main objects of study: type $C$ multiline queues. 

\begin{definition}[Type $C$ multiline queues]\label{def:typeC-MLQs}
    Let $n\geq 1$, and let $\lambda$ be a partition with $\lambda_1 = L$. A \emph{multiline queue of type $C$ of shape $\lambda$ on $n$ columns} is a $2L\times n$ array
\[
M=(b_1,b_1',b_2,b_2',\ldots,b_{L},b_L'),\qquad b_j,b_j'\in\left\{\cir,\squ,\cdot\right\}^n\quad\text{for } 1\leq j\leq L,
\]
obtained by stacking the $L$ type $C$ rows in an element of $\RowC(\lambda',n)$. We denote the set of such multiline queues by $\MLQC(\lambda,n)$.
\end{definition}

There is a natural identification $\MLQC(\lambda,n) \cong \RowC(\lambda',n)$ under which an element 
\[B_1\otimes\cdots\otimes B_L\in\RowC(\lambda',n),\qquad B_j=\begin{pmatrix}b_{j}'\\b_{j}\end{pmatrix}\in\RowC(\lambda'_j,n),
\]
is identified with the array $M=(b_1,b_1',\ldots,b_{L},b_L')$.
Although the two models are equivalent, we keep the distinction between the sets $\MLQC(\lambda,n)$ and $\RowC(\lambda',n)$ because they serve different purposes. On $\MLQC(\lambda,n)$, we will define a type $C$ analogue of the FM procedure (see \cref{fig:TypeC-MLQ-pairing}). On $\RowC(\lambda',n)$, we will define a projection using a corner transfer matrix built from the combinatorial R matrix. Both will yield a projection to states of the multispecies open TASEP. Unlike in type $A$, the two constructions do not coincide on the nose: the corner transfer matrix formulation on $\RowC(\lambda,n)$ is not the same as the pairing procedure on $\MLQC(\lambda,n)$. 

We define the notion of \emph{type $C$ pairing} for a pair of rows $(a,b)$ with $a,b\in\{\cir,\squ,\,\cdot\,\}^n$, analogous to the cylindrical type $A$ pairing given by the NY rule. If we consider that $b$ is on top of $a$, the pairing determines a subset of the particles in $a$ that are paired with particles in $b$. If $a$ has more particles than $b$, each particle in $b$ is sequentially paired to an unpaired particle in $a$, with $\cir$'s pairing weakly to the left and $\squ$'s pairing weakly to the right, and with particles reflecting off the boundaries and reversing direction and preference. We will use this rule to define pairing on $\MLQC(\lambda,n)$.

\begin{definition}[Type $C$ pairing]\label{def:type-c-pairing}
    Let $a,b\in\left\{\cir\,,\,\squ\,,\,\cdot\right\}^n$ be two words where $a$ and $b$ have $m$ and $\ell$ particles, respectively. Define $\cPair(a,b)\subseteq a$ as follows. If $m\leq\ell$, set $\cPair(a, b)=a$. Otherwise, consider the particles in $b$ in positions $\{j_1,\ldots,j_\ell\}$. Set $a'=a$ and $c=(\,\cdot\,)^n$. For $i=1,2,\ldots,\ell$, let $j=j_i$, and define $$k=\begin{cases}
        k_{\squ}&\text{if }b_j=\squ\\
        k_{\cir}&\text{if }b_j=\cir\end{cases}$$
        where
    \begin{align*}k_{\cir}&=\begin{cases}
        \max\{u\in[j]:a'_u=\cir\}&\text{if }\cir\in a'_1\cdots a'_j\\
        \min\{u\in[n]:a'_u=\squ\}&\text{if }\cir\not\in a'_1\cdots a'_j,\ \squ\in a'\\
        \max\{u\in[n]:a'_u=\cir\}& \text{else},
    \end{cases}\\
  k_{\squ}&=\begin{cases}
        \min\{u\in[j,n]:a'_u=\squ\}&\text{if }\squ\in a'_j\cdots a'_n\\
        \max\{u\in[n]:a'_u=\cir\}&\text{if }\squ\not\in a'_j\cdots a'_n,\ \cir\in a'\\
        \min\{u\in[n]:a'_u=\squ\}&\text{else}.
    \end{cases}
    \end{align*}
    Then, set $c_{k}=a_{k}$ and $a'_k=\,\cdot\,$. Once all particles of $b$ have been processed, set $\cPair(a, b)\coloneqq c$.
\end{definition}

\begin{example}\label{ex:typeC-pairing}
    Let $a = \left( \squ\,,\,\cdot\,,\,\squ\,,\,\squ\,,\,\squ\,,\,\cdot\,,\,\cir\right)$ and $b = \left( \cdot\,,\,\squ\,,\,\cdot\,,\,\cir\,,\,\cdot\,,\,\squ\,,\,\squ\right)$. The positions of the $\squ$ in $b$ from left to right are $j_1=2$, $j_2=6$, $j_3=7$ and the position of $\cir$ in $b$ is $j_4=4$. We show the iterations of the type $C$ pairing algorithm in \cref{table}. At each step, we show the updated $a'$ and $c$ given the information from the row above. Also, we record the choices of elements of $a'$ in a diagram that we will consider as the multiline queue together with its pairing. Explicitly, for each particle in $b$, we draw a string from it to the corresponding particle in $a$ that gets picked by the $K$, and hence replaced in $a'$ by $\cdot$ and $c$ by the corresponding type of particle. 
\end{example}

\begin{remark}\label{rem:type C via type A}
Type $C$ pairing can be equivalently described in terms of cylindrical type $A$ pairing on the ordered alphabet
\[
\Ac^C_n=1<2<\cdots<n<\overline{n}<\cdots<\overline{1}.
\]
For a word $a\in\{\cir,\squ,\cdot\}^n$, define
\[\widetilde a=(\widetilde a_1,\ldots,\widetilde a_n,\widetilde a_{\bar n},\ldots,\widetilde a_{\bar 1}),\qquad \widetilde a_j=\begin{cases}
    1&\text{if }(j\leq n\text{ and }a_j=\cir)\text{, or }(j\geq \bar n\text{ and }a_{|j|}=\squ)\\
    0&\text{else}.
\end{cases} 
\]
We analyze a single step in computing $\cPair(a,b)=c$. Let $a'$ denote the set of unpaired particles in $a$ at that step, and let $\widetilde a'$ be its embedding in $\Ac^C_n$. If $b_j=\cir$ (resp.~$b_j=\squ$), then $k=k_{\cir}$ (resp.~$k=k_{\squ}$) is exactly the position of the rightmost $1$ in $\widetilde{a}'$ that is (cyclically) weakly to the left of site $j$ (resp.~site $\bar j$). Writing $\widetilde b$ for the embedding of $b$ in $\Ac^C_n$, we then have that $k$ is equal to the position selected by cylindrical type-$A$ pairing of the corresponding letter of $\widetilde{b}$ with $\widetilde{a}'$. Thus the paired positions given by the NY rule for the type $A$ pairing of $\widetilde b$ to $\widetilde a$ is the embedding of $c=\cPair(a,b)$ in $\Ac^C_n$.

\end{remark}

    \begin{table}[t]
    \centering
        \setlength{\tabcolsep}{5pt}
        \begin{tabular}{c | c | c | c}
        Step & $a'$ & $c$ & Diagram of $a'\otimes b$\\[5pt]
        \hline
        \rule{0pt}{5ex}
         -  & $\left( \squ\,,\,\Cdot\,,\,\squ\,,\,\squ\,,\,\squ\,,\,\Cdot\,,\,\cir\right)$  & $\left( \Cdot\,,\,\Cdot\,,\,\Cdot\,,\,\Cdot\,,\,\Cdot\,,\,\Cdot\,,\,\Cdot\right)$ 
           & $\vcenter{\hbox{\resizebox{!}{.65in}{
        \begin{tikzpicture}[scale=0.5]
        \def \w{1.3};
        \def \h{1};
        \def \r{0.325};
        \draw[gray!60] (0,0)--(0,3);
        \draw[gray!60] (7*\w,0)--(7*\w,3);
        \foreach \xx\yy in {
        6/0,3/2}
        {
        \draw (\w*.5+\w*\xx,\h*.5+\h*\yy) circle (\r cm);
        }
        \foreach \xx\yy in {
        0/0,2/0,3/0,4/0,1/2,5/2,6/2}
        {
        \node at (\w*.5+\w*\xx,\h*.5+\h*\yy) {$\squa{0.325}$};
        }
        \foreach \xx\yy in {
        1/0,5/0,2/2,4/2,0/2}
        {
        \node at (\w*.5+\w*\xx,\h*.5+\h*\yy) {$\cdot$};
        }
        \end{tikzpicture}
        }}}$\\[6ex]
        \hline
        \rule{0pt}{5ex}
        $j_1=2$& $\left( \squ\,,\,\Cdot\,,\,\tcr{\squ}\,,\,\squ\,,\,\squ\,,\,\Cdot\,,\,\cir\right)$ &$\left( \Cdot\,,\,\Cdot\,,\,\squ\,,\,\Cdot\,,\,\Cdot\,,\,\Cdot\,,\,\Cdot\right)$ 
         & $\vcenter{\hbox{\resizebox{!}{.65in}{
        \begin{tikzpicture}[scale=0.5]
        \def \w{1.3};
        \def \h{1};
        \def \r{0.325};
        \draw[gray!60] (0,0)--(0,3);
        \draw[gray!60] (7*\w,0)--(7*\w,3);
        \foreach \xx\yy in {
        6/0,3/2}
        {
        \draw (\w*.5+\w*\xx,\h*.5+\h*\yy) circle (\r cm);
        }
        \foreach \xx\yy in {
        0/0,2/0,3/0,4/0,1/2,5/2,6/2}
        {
        \node at (\w*.5+\w*\xx,\h*.5+\h*\yy) {$\squa{0.325}$};
        }
        \foreach \xx\yy in {
        1/0,5/0,2/2,4/2,0/2}
        {
        \node at (\w*.5+\w*\xx,\h*.5+\h*\yy) {$\cdot$};
        }
        \draw[black,thick] (\w*1.5,\h*2.5-\r)--(\w*1.5,\h*2)--(\w*2.5,\h*2)--(\w*2.5,\h*.5+\r);
        \end{tikzpicture}
        }}}$\\[6ex]
        \hline
        \rule{0pt}{5ex}
        $j_2=6$& $\left( \squ\,,\,\Cdot\,,\,\Cdot\,,\,\squ\,,\,\squ\,,\,\Cdot\,,\,\tcr{\cir}\right)$ &$\left( \Cdot\,,\,\Cdot\,,\,\squ\,,\,\Cdot\,,\,\Cdot\,,\,\Cdot\,,\,\cir\right)$ 
          & $\vcenter{\hbox{\resizebox{!}{.65in}{
        \begin{tikzpicture}[scale=0.5]
        \def \w{1.3};
        \def \h{1};
        \def \r{0.325};
        \draw[gray!60] (0,0)--(0,3);
        \draw[gray!60] (7*\w,0)--(7*\w,3);
        \foreach \xx\yy in {
        6/0,3/2}
        {
        \draw (\w*.5+\w*\xx,\h*.5+\h*\yy) circle (\r cm);
        }
        \foreach \xx\yy in {
        0/0,2/0,3/0,4/0,1/2,5/2,6/2}
        {
        \node at (\w*.5+\w*\xx,\h*.5+\h*\yy) {$\squa{0.325}$};
        }
        \foreach \xx\yy in {
        1/0,5/0,2/2,4/2,0/2}
        {
        \node at (\w*.5+\w*\xx,\h*.5+\h*\yy) {$\cdot$};
        }
        \draw[black] (\w*1.5,\h*2.5-\r)--(\w*1.5,\h*2)--(\w*2.5,\h*2)--(\w*2.5,\h*.5+\r);
        \draw[blue,thick] (\w*5.5,\h*2.5-\r)--(\w*5.5,\h*2.25-\r)--(7*\w,1.5)--(\w*6.5,\h*0.75+\r)--(\w*6.5,\h*.5+\r);
        \end{tikzpicture}
        }}}$\\[6ex]
        \hline
        \rule{0pt}{5ex}
        $j_3=7$& $\left( \tcr{\squ}\,,\,\Cdot\,,\,\Cdot\,,\,\squ\,,\,\squ\,,\,\Cdot\,,\,\Cdot\right)$ &$\left( \squ\,,\,\Cdot\,,\,\squ\,,\,\Cdot\,,\,\Cdot\,,\,\Cdot\,,\,\cir\right)$ 
         & $\vcenter{\hbox{\resizebox{!}{.65in}{
        \begin{tikzpicture}[scale=0.5]
        \def \w{1.3};
        \def \h{1};
        \def \r{0.325};
        \draw[gray!60] (0,0)--(0,3);
        \draw[gray!60] (7*\w,0)--(7*\w,3);
        \foreach \xx\yy in {
        6/0,3/2}
        {
        \draw (\w*.5+\w*\xx,\h*.5+\h*\yy) circle (\r cm);
        }
        \foreach \xx\yy in {
        0/0,2/0,3/0,4/0,1/2,5/2,6/2}
        {
        \node at (\w*.5+\w*\xx,\h*.5+\h*\yy) {$\squa{0.325}$};
        }
        \foreach \xx\yy in {
        1/0,5/0,2/2,4/2,0/2}
        {
        \node at (\w*.5+\w*\xx,\h*.5+\h*\yy) {$\cdot$};
        }
        \draw[black] (\w*1.5,\h*2.5-\r)--(\w*1.5,\h*2)--(\w*2.5,\h*2)--(\w*2.5,\h*.5+\r);
        \draw[blue] (\w*5.5,\h*2.5-\r)--(\w*5.5,\h*2.25-\r)--(7*\w,1.5)--(\w*6.5,\h*0.75+\r)--(\w*6.5,\h*.5+\r);
        \draw[red,thick] (\w*6.5,\h*2.5-\r)--(\w*6.5,\h*2.25-\r)--(6.925*\w,1.8)--(0,1.2)--(\w*0.5,\h*0.75+\r)--(\w*0.5,\h*.5+\r);
        \end{tikzpicture}
        }}}$\\[6ex]
        \hline
        \rule{0pt}{5ex}
        $j_4=4$& $\left( \Cdot\,,\,\Cdot\,,\,\Cdot\,,\,\tcr{\squ}\,,\,\squ\,,\,\Cdot\,,\,\Cdot\right)$ &$\left( \squ\,,\,\Cdot\,,\,\squ\,,\,\squ\,,\,\Cdot\,,\,\Cdot\,,\,\cir\right)$  & $\vcenter{\hbox{\resizebox{!}{.65in}{
        \begin{tikzpicture}[scale=0.5]
        \def \w{1.3};
        \def \h{1};
        \def \r{0.325};
        \draw[gray!60] (0,0)--(0,3);
        \draw[gray!60] (7*\w,0)--(7*\w,3);
        \foreach \xx\yy in {
        6/0,3/2}
        {
        \draw (\w*.5+\w*\xx,\h*.5+\h*\yy) circle (\r cm);
        }
        \foreach \xx\yy in {
        0/0,2/0,3/0,4/0,1/2,5/2,6/2}
        {
        \node at (\w*.5+\w*\xx,\h*.5+\h*\yy) {$\squa{0.325}$};
        }
        \foreach \xx\yy in {
        1/0,5/0,2/2,4/2,0/2}
        {
        \node at (\w*.5+\w*\xx,\h*.5+\h*\yy) {$\cdot$};
        }
        \draw[black] (\w*1.5,\h*2.5-\r)--(\w*1.5,\h*2)--(\w*2.5,\h*2)--(\w*2.5,\h*.5+\r);
        \draw[violet,thick] (\w*3.5,\h*2.5-\r)--(\w*3.5,\h*1.8)--(0,1.5)--(\w*3.5,\h*1.2)--(\w*3.5,\h*.5+\r);
        \draw[blue] (\w*5.5,\h*2.5-\r)--(\w*5.5,\h*2.25-\r)--(7*\w,1.5)--(\w*6.5,\h*0.75+\r)--(\w*6.5,\h*.5+\r);
        \draw[red] (\w*6.5,\h*2.5-\r)--(\w*6.5,\h*2.25-\r)--(6.925*\w,1.8)--(0,1.2)--(\w*0.5,\h*0.75+\r)--(\w*0.5,\h*.5+\r);
        \end{tikzpicture}
        }}}$\\[6ex]
    \end{tabular}
    \caption{The iterations of the type $C$ pairing algorithm for \cref{ex:typeC-pairing}.}\label{table}
    \end{table}

\begin{example}

Consider \cref{ex:typeC-pairing} with the vectors 
$$\widetilde{a} = (0,0,0,0,0,0,1,0,0,1,1,1,0,1) \quad \text{and}\quad\widetilde{b} = (0,0,0,1,0,0,0,1,1,0,0,0,1,0).$$
We represent the pairing of $\widetilde{a}$ and $\widetilde{b}$ below, where we label the particles corresponding to $j_1,\ldots,j_4$ in \cref{ex:typeC-pairing} in the same order. Each particle is paired in this order to the nearest unpaired particle weakly to its left (cyclically). Notice that this diagram coincides with the paired multiline queue from \cref{ex:typeC-pairing} after folding the diagram along the middle dashed line and turning the right-hand side circles into squares, and hence we recover the type $C$ pairing. 

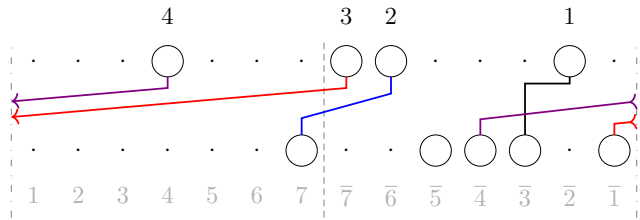
\begin{center}
    \resizebox{.5\linewidth}{!}{
    \begin{tikzpicture}[scale=0.7]
    \def \w{1};
    \def \h{1};
    \def \r{0.35};
    \draw[gray,densely dashed] (7,-1)--(7,3);
    \draw[gray,loosely dashdotted] (0,-1)--(0,3);
    \draw[gray,loosely dashdotted] (14,-1)--(14,3);
    \node at (\w*12.5,\h*3.5) {1};
    \node at (\w*8.5,\h*3.5) {2};
    \node at (\w*7.5,\h*3.5) {3};
    \node at (\w*3.5,\h*3.5) {4};

    \foreach \xx\yy\ii in {
    0/-1/1,1/-1/2,2/-1/3,3/-1/4,4/-1/5,5/-1/6,6/-1/7}
    {
    \node[gray!60] at (\w*.5+\w*\xx,\h*.5+\h*\yy) {\small $\ii$};
    }
    \foreach \xx\yy\ii in {
    7/-1/7,8/-1/6,9/-1/5,10/-1/4,11/-1/3,12/-1/2,13/-1/1}
    {
    \node[gray!60] at (\w*.5+\w*\xx,\h*.5+\h*\yy) {\small $\overline{\ii}$};
    }
    \foreach \xx\yy in {
    11/0,6/0,9/0,10/0,13/0,
    3/2,7/2,8/2,12/2}
    {
    \draw (\w*.5+\w*\xx,\h*.5+\h*\yy) circle (\r cm);
    }
    \foreach \xx\yy in {
    0/0,1/0,2/0,3/0,4/0,5/0,7/0,8/0,12/0,
    0/2,1/2,2/2,4/2,5/2,6/2,9/2,10/2,11/2,13/2}
    {
    \node at (\w*.5+\w*\xx,\h*.5+\h*\yy) {$\cdot$};
    }
    \draw[black,thick] (\w*12.5,\h*2.5-\r)--(\w*12.5,\h*2)--(\w*11.5,\h*2)--(\w*11.5,\h*.5+\r);
    \draw[blue,thick] (\w*8.5,\h*2.5-\r)--(\w*8.5,\h*2.125-\r)--(\w*6.5,\h*0.875+\r)--(\w*6.5,\h*.5+\r);
    \draw[red,->,thick] (\w*7.5,\h*2.5-\r)--(\w*7.5,\h*2.25-\r)--(0,1.25);
    \draw[red,>-,thick] (14,1.15)--(\w*13.5,\h*0.75+\r)--(\w*13.5,\h*.5+\r);
    \draw[violet,->,thick] (\w*3.5,\h*2.5-\r)--(\w*3.5,\h*1.9)--(0,1.6);
    \draw[violet,>-,thick] (14,1.6)--(\w*10.5,\h*1.2)--(\w*10.5,\h*.5+\r);
    \end{tikzpicture}
    }
\end{center} 
\end{example}

\begin{lemma}\label{lem:type C order independent}
The set $\cPair(a,b)$ is independent of the order in which the particles of $b$ are processed.
\end{lemma}

\begin{proof}
Following \cref{rem:type C via type A}, type $C$ pairing is equivalent to cylindrical
type $A$ pairing (the NY rule) on the alphabet $\Ac_n^C.$ Thus the claim follows from the fact that the NY rule is a realization of the combinatorial $R$ matrix, whose output is independent of the order of individual pairings. 
\end{proof}

\subsubsection{Crystal operators for tensor products of type $C$ rows}\label{sec:crystal-ops-type-c-rows}

\begin{definition}[The operators $f_i$]\label{def:f_i}
Define the operators $f_i\colon\RowC(k,n)\to\RowC(k,n)\cup \{\mathbf{0}\}$ for $0\leq i\leq n$ as follows. Let $Q=(c_1,\ldots,c_n)\in\RowC(k,n)$.   

For $1\leq i\leq n-1$, the operator $f_i$ affects the columns $i,i+1$ and leaves all other columns unchanged. We describe the action of $f_i$ on those two columns. There are three possible (disjoint) cases:
\vspace{1em}
\begin{enumerate}
    \item \emph{Twisting}: if $(c_i,c_{i+1})=\begin{pmatrix}
        \cir&\squ\\\cir&\squ
    \end{pmatrix}$, then $f_i((c_i,c_{i+1}))=\begin{pmatrix}
        \squ&\cir\\\cir&\squ
    \end{pmatrix}$.
    \item \emph{Untwisting}: if $(c_i,c_{i+1})=\begin{pmatrix}
        \squ&\cir\\\cir&\squ
    \end{pmatrix}$, then $f_i((c_i,c_{i+1}))=\begin{pmatrix}
        \squ&\cir\\\squ&\cir
    \end{pmatrix}$.
    \item \emph{Swapping}: if $(c_i,c_{i+1}) \in \left\{\begin{pmatrix}
        \cir&\cdot\\\cir&\cdot
    \end{pmatrix},\begin{pmatrix}
        \cdot&\squ\\\cdot&\squ
    \end{pmatrix},\begin{pmatrix}
        \cir&\squ\\\cir&\cir
    \end{pmatrix},\begin{pmatrix}
        \cir&\cir\\\cir&\squ
    \end{pmatrix},\begin{pmatrix}
        \cir&\squ\\\squ&\squ
    \end{pmatrix},\begin{pmatrix}
        \squ&\squ\\\cir&\squ
    \end{pmatrix} \right\}$ 
    then \\ $f_i((c_i,c_{i+1})) = (c_{i+1},c_{i})$. 
\end{enumerate}
\vspace{1em}
The left and right boundary operators $f_0$ and $f_n$ act on the columns $1$ and $n$, respectively:
\vspace{1em}
\begin{enumerate}
    \item[4.] \emph{Left boundary reflection}: if $c_1=\sqsq$, then $f_0(c_1)=\circir$.
    \item[5.] \emph{Right boundary reflection}: if $c_n=\circir$, then $f_n(c_n)=\sqsq$.
\end{enumerate}
\vspace{1em}
In all other cases, $f_i(Q)=\mathbf{0}$. 
\end{definition}

\begin{definition}[The operators $e_i$]
    \label{def:e_i}
    For $0\leq i\leq n$, the operator $e_i\colon\RowC(k,n)\to
    \RowC(k,n)\cup \{\mathbf{0}\}$ is defined as follows. For $Q\in\RowC(k,n)$, define $e_i(Q)=Q'$,
    where $Q'\in\RowC(k,n)$ satisfies $f_i(Q')=Q$; if no such $Q'$ exists, then $e_i(Q) = \mathbf{0}$.
\end{definition}

\begin{definition}
    For a type $C$ row $Q$, define 
    \[\varphi_i(Q)=\max\Big\{j:f_i^{j}(Q)\neq \mathbf{0}\Big\}\quad\text{and}\quad \varepsilon_i(Q)=\max\Big\{j:e_i^{j}(Q)\neq \mathbf{0}\Big\}.\]
   By convention, if $f_i(Q)=\mathbf{0}$ (resp.~$e_i(Q)=0$), then $\varphi_i(Q)=0$ (resp.~$\varepsilon_i(Q)=0$).
\end{definition}

Note that when $f_i$ acts on $Q$ by twisting, then $\varphi_i(Q)=2$; in all other cases above, $\varphi_i(Q)=1$, and if neither case occurs, then $\varphi_i(Q)=0$. It is straightforward to verify the following.
\begin{prop}\label{prop:row}
    The map $\row$ from \cref{lem:bijection_KN-typeC_rows} intertwines the operators $f_i$ and $e_i$ with the crystal operators on KN columns of type $C$.
\end{prop}

\begin{proof}
    Let $T\in \KN_C(k,n)$ and $1\leq i\leq n-1$. We denote the crystal operators on KN columns from \cref{def:crystal-operators-type-C} by $F_i$ and compare $f_i(\row(T))$ to $\row(F_i(T))$ in all cases where $F_i(T)\neq\mathbf{0}$. 
    \begin{enumerate}[itemsep=2pt]
        \item If $F_i(T)$ is obtained by replacing $i$ with $i+1$ in $T$, we must have $i\in T$, $i+1 \notin T$, and one of the following cases for the columns $i,i+1$ in $\row(T)$:\vspace{2pt}
        \begin{itemize}[itemsep=2pt]
            \item \textbf{Case 1:   $\overline{i+1},\overline{i}\notin T$} 
            \[\row(T)_{i,i+1}=\begin{pmatrix}
                \cir & a \\
                \cir & b \\
            \end{pmatrix}\implies \row(F_i(T))_{i,i+1}=\begin{pmatrix}
                a & \cir \\
                b & \cir \\
            \end{pmatrix},  \qquad  \begin{pmatrix}
                a \\
                b
            \end{pmatrix}\in\left\{\begin{pmatrix}
                \cdot\\\cdot\end{pmatrix},\begin{pmatrix} \squ \\
                \cir 
            \end{pmatrix}\right\}.\] 

            \item \textbf{Case 2: $\overline{i+1}\in T$ and $\overline{i}\notin T$}
            \[\row(T)_{i,i+1}=\begin{pmatrix}
                \cir & \squ \\
                \cir & \squ \\
            \end{pmatrix}\implies \row(F_i(T))_{i,i+1}=\begin{pmatrix}
                \squ & \cir \\
                \cir & \squ \\
            \end{pmatrix}.
            \]
            \item \textbf{Case 3: $\overline{i+1},\overline{i}\in T$}
            \[\row(T)_{i,i+1}=\begin{pmatrix}
                \cir & \squ \\
                \squ & \squ \\
            \end{pmatrix}\implies\row(F_i(T))_{i,i+1}=\begin{pmatrix}
                \squ & \cir \\
                \squ & \squ \\
            \end{pmatrix}.\] 
        \end{itemize}
        Cases 1 and 3 correspond to $f_i$ the swapping the columns, and Case 2 corresponds to twisting.
        
        \item If $F_i(T)$ is obtained by replacing $\overline{i+1}$ with $\overline{i}$ in $T$, we must have $\overline{i+1}\in T$, $\overline{i} \notin T$, and one of the following cases for the columns $i,i+1$ in $\row(T)$:\vspace{2pt}
        \begin{itemize}[itemsep=2pt]
            \item \textbf{Case 1: $i,i+1\notin T$}
            \[\row(T)_{i,i+1}=\begin{pmatrix}
                a & \squ \\
                b & \squ \\
            \end{pmatrix}\implies\row(F_i(T))_{i,i+1} =\begin{pmatrix}
                \squ & a \\
                \squ & b \\
            \end{pmatrix},  \qquad  \begin{pmatrix}
                a \\
                b
            \end{pmatrix}\in\left\{\begin{pmatrix}
                \cdot\\\cdot\end{pmatrix},\begin{pmatrix} \squ \\
                \cir 
            \end{pmatrix}\right\}.\]  

            \item \textbf{Case 2: $i\notin T$ and $i+1\in T$}
            \[\row(T)_{i,i+1}=\begin{pmatrix}
                \squ & \cir \\
                \cir & \squ \\
            \end{pmatrix}\implies \row(F_i(T))_{i,i+1}=\begin{pmatrix}
                \squ & \cir \\
                \squ & \cir \\
            \end{pmatrix}.
            \]
            \item \textbf{Case 3: $i,i+1\in T$}

            \[\row(T)_{i,i+1}=\begin{pmatrix}
                \cir & \cir \\
                \cir & \squ \\
            \end{pmatrix}\implies\row(F_i(T))_{i,i+1}=\begin{pmatrix}
                \cir & \cir \\
                \squ & \cir \\
            \end{pmatrix}.\] 
        \end{itemize}

        Cases 1 and 3 correspond to $f_i$ swapping the columns, and Case 2 corresponds to untwisting.

        \item For $i=0$, if $F_0(T)\neq\mathbf{0}$, then $\overline{1}\in T$, $1\not\in T$, and $F_0(T)$ replaces the $\overline{1}$ with $1$. Then 
        \[\row(T)_1=\sqsq\implies\row(F_0(T))_1=\circir,
        \]
        giving the left boundary reflection $f_0$. 
        Similarly, for $i=n$, if $F_n(T)\neq\mathbf{0}$, then $q\in T$ and $\overline{1}\not\in T$, in which case $F_n$ corresponds to the right boundary reflection $f_n$. 
    \end{enumerate}
    Thus in all cases, $f_i(\row(T))=\row(F_i(T))$, as desired.
\end{proof}

The crystal operators on type $C$ rows naturally extend to operators on their tensor products using the signature rule. We give a diagrammatic description of how these operators act. 
We define \emph{$i$-column-unmatched particles}, which arise from the tensor product of type $C$ rows
$\Qb=Q_1\otimes\cdots\otimes Q_L$. These determine which factor $Q_j$ the operator $\widetilde f_i$ acts on to obtain $\widetilde f_i(\Qb)$. 

We give a description for identifying the rows containing the $i$-column-unmatched particles directly from the tensor product of rows by stacking them from bottom to top, i.e., row $Q_1$ is in the bottom and $Q_{i}$ is on top of $Q_{i-1}$ for $i=2\ldots,L$. The output is equivalent to taking the set of unmatched particles given by the $i$-signature rule for the corresponding tensor product of KN columns. 
\begin{definition}[$i$th reduced parentheses word of a type $C$ row]\label{def:par_i_typeC_rows}
   For $Q=(c_1,\ldots,c_n)\in\RowC(k,n)$ and $0\leq i\leq n$, define 

\begin{subequations}
\begin{empheq}[left={\pth_i(Q)=\empheqlbrace}]{align}
& )
&&\text{if } c_i=\circir,\ c_{i+1}\notin\{\circir,\sqsq\}
\text{ or } c_{i+1}=\sqsq,\ c_i\notin\{\circir,\sqsq\},
\label{c.1}\\
& (
&&\text{if } c_i=\sqsq,\ c_{i+1}\notin\{\circir,\sqsq\}
\text{ or } c_{i+1}=\circir,\ c_i\notin\{\circir,\sqsq\},
\label{c.2}\\
& ))
&&\text{if } (c_i,c_{i+1})=
\begin{pmatrix}
\cir&\squ\\
\cir&\squ
\end{pmatrix},
\label{c.3}\\
& ((
&&\text{if } (c_i,c_{i+1})=
\begin{pmatrix}
\squ&\cir\\
\squ&\cir
\end{pmatrix},
\label{c.4}\\
& )(
&&\text{if } (c_i,c_{i+1})=
\begin{pmatrix}
\squ&\cir\\
\cir&\squ
\end{pmatrix},
\label{c.5}\\
& \emptyset
&&\text{otherwise.}\label{c.6}
\end{empheq}
\end{subequations}

We set the convention $c_0=c_{n+1}=\emp$ for defining $\pth_0(Q)$ and $\pth_n(Q)$, respectively.
   
Then, for any composition $\alpha$ and $\Qb=Q_1\otimes\cdots\otimes Q_L \in \RowC(\alpha,n)$, define the parentheses word as the concatenation 
\[\pth_i(\Qb)=\pth_i(Q_1)\cdots \pth_i(Q_L).
\]
\end{definition}

\begin{lemma}\label{lem:par equivalent}
    Let $Q$ be a type $C$ row and $T = \row(Q)$ the corresponding KN column. Then $\pth_i(Q)=\overline{\pth}_i(T)$. 
\end{lemma}

\begin{proof}
   We compare the cases that determine the two definitions in the following table:
   \[\begin{array}{c|c|c|c|c|c|c}
       \pth_i(Q)&\eqref{c.1}&\eqref{c.2}&\eqref{c.3}&\eqref{c.4}&\eqref{c.5}&\eqref{c.6}\\[3pt]
\hline
       \overline{\pth}_i(T)&\eqref{d.1a},\eqref{d.1b}&\eqref{d.2a},\eqref{d.2b}&\eqref{d.3}&\eqref{d.4}&\eqref{d.5}&\eqref{d.6}
   \end{array}
   \]
   One may verify that in all cases, the output of $\pth_i(Q)$ matches that of $\overline{\pth}_i(T)$.
\end{proof}

\begin{remark}
Comparing with \cref{def:f_i}, the parenthesis word from \cref{def:par_i_typeC_rows} can be equivalently described as follows:
\[\pth_i(\Qb)=)^{\varphi_i(Q_1)}(^{\varepsilon_i(Q_1)}\cdots )^{\varphi_i(Q_L)}(^{\varepsilon_i(Q_L)},\qquad\qquad \pth_i(Q_j)=)^{\varphi_i(Q_j)}(^{\varepsilon_i(Q_j)}.\]
\end{remark}

By \cref{prop:row,lem:par equivalent}, we define the crystal operators 
\[\widetilde{f}_i,\widetilde{e}_i\colon\RowC(\lambda,n)\to\RowC(\lambda,n)\sqcup \{\mathbf{0}\}
\]
by identifying $\Qb\in\RowC(\lambda,n)$ with the corresponding tensor product of KN columns $\row^{-1}(\Qb)$, and applying \eqref{eq:ftilde}. 

The following lemma is immediate from our construction, \cref{lem:par equivalent}, and comparisons with \cref{def:crystal-operators-type-C} and \cref{def:splitting}.

\begin{lemma}\label{lemma:row is split}
For $k,n\in\NN$, the map $\row$ gives a bijection 
\[
    \row:B^k\simeq\RowC(k,n),
\]
where $B^{k}=B^{k,1}$ is the type $C_n^{(1)}$ KR crystal. Under this bijection, the particle configuration $Q\in\RowC(k,n)$ is precisely the double row encoding of the splitting of the KN column $\row^{-1}(Q)$. 

The same statement holds component-wise for tensor products: for a partition $\lambda$ with $\lambda_1=L$, the map $\row$ induces the bijection
\[ B_\lambda\coloneq B^{\lambda'_1}\otimes\cdots\otimes B^{\lambda'_{L}}\simeq \RowC(\lambda,n).
\]
In particular, the crystal operators defined on tensor products of type $C$ rows agree, under this bijection, with the crystal operators on $B_\lambda$.
\end{lemma}

\section{TASEP dynamics induced by crystal operators}\label{sec:crystal-MCs}

In this section, we show that crystal operators induce TASEP behavior on the image of the projection map. We begin with a description of the simpler type $A$ case, and then show that the same properties hold in type $C$.

Throughout this section, we identify each $M\in\MLQ_A(\lambda,n)$ with its corresponding tensor product of KR columns.

\subsection{The crystal Markov chain in type $A$}\label{sec:type $A$}

Define a Markov chain on $\MLQ_A(\lambda,n)$, denoted $\mlqA(\lambda,n)$, as follows. For a multiline queue $M\in\MLQ_A(\lambda,n)$, let
\[E(M)\coloneq\{e_i^j\colon1\leq j\leq \varepsilon_i(M),\ 0\leq i\leq n-1\}.\]
Then define the set of transitions, all with uniform probability, by:
\[M\to e_i^j M,\qquad e_i^j\in E(M).\]
 We call this process the \emph{type $A$ crystal Markov chain}. 
 
\begin{lemma}\label{lem:irreducible A}
The type $A$ crystal Markov process is irreducible.
\end{lemma}

\begin{proof}
    Denote the lowest weight element $M_{\delta}\in\MLQ_A(\lambda,n)$ with weight $(\lambda_{n},\ldots,\lambda_2,\lambda_1)$ where we set $\lambda_{\ell(\lambda)+1}=\cdots=\lambda_n=0$. As the KR crystal $B_\lambda$ equivalent to $\mlqA(\lambda,n)$ is connected, any $M'\in\MLQ_A(\lambda,n)$ can be written as $e_{j_1}\cdots e_{j_s}M_\delta$, and by a result of Kashiwara \cite{kashiwara02}, for any $M\in \mlqA(\lambda,n)$ there exists a path $e_{i_1}\cdots e_{i_r}$ in the affine crystal from $M$ to $M_\delta$. Thus $e_{j_1}\cdots e_{j_s}e_{i_1}\cdots e_{i_r}$ is a path from $M$ to $M'$ passing through $M_\delta$. As this holds for any $M,M'\in\MLQ_A(\lambda,n)$, the Markov process is irreducible.
\end{proof}

See \cref{ex:typeA crystal} for examples of transitions, and \cref{ex:21} for the full Markov chain $\mlqA((2,1),3)$.

\begin{example}\label{ex:typeA crystal}
    For $\lambda=(4,4,3,2)$ and $n=6$, let 
    \[M = (\{1,2,4,5\},\{2,4,5,6\},\{3,4,5\},\{4,6\})\in \MLQ(\lambda,n).\] 
    The set of transitions $M\to M'$ in $\mlqA(\lambda,n)$ is given by $E(M)=\{e_0,e_1,e_3,e_3^2,e_3^3\}$. We show all transitions from $M$ (top center) below, with displaced balls highlighted.
 
    
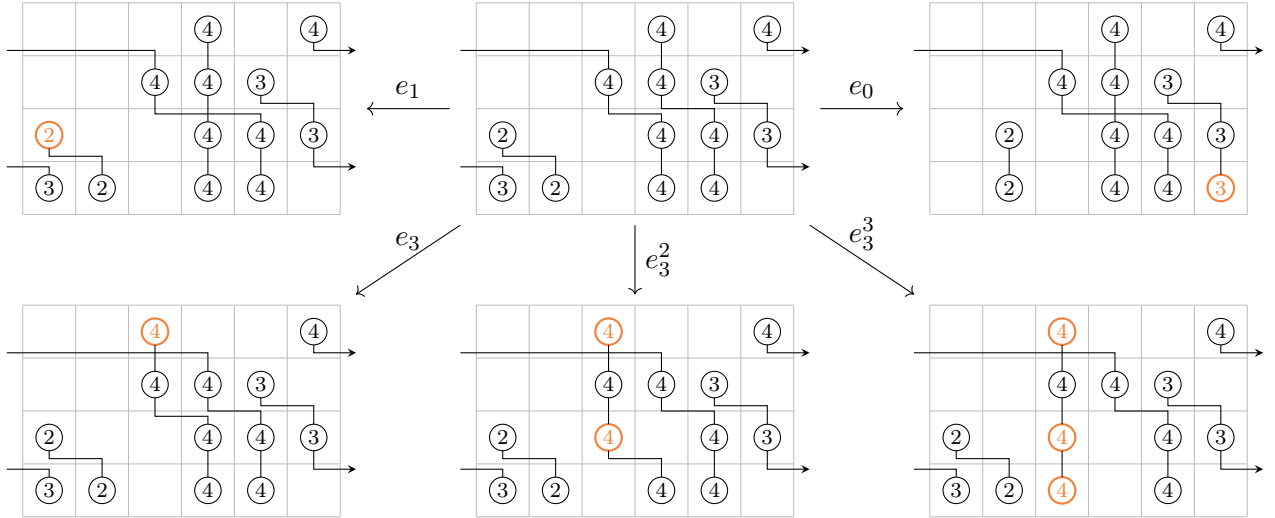
\begin{center}
        \begin{tikzpicture}
            \node (M) at (0,0) {\begin{tikzpicture}[scale=0.7]
            \def \w{1};
            \def \h{1};
            \def \r{0.25};
            
            \foreach \i in {0,...,4}
            {
            \draw[gray!50] (0,\i*\h)--(\w*6,\i*\h);
            }
            \foreach \i in {0,...,6}
            {
            \draw[gray!50] (\w*\i,0)--(\w*\i,4*\h);
            }
            \foreach \xx\yy\i in {
            0/0/3,1/0/2,3/0/4,4/0/4,
            0/1/2,3/1/4,4/1/4,5/1/3,
            2/2/4,3/2/4,4/2/3,
            3/3/4,5/3/4}
            {
            \draw (\w*.5+\w*\xx,\h*.5+\h*\yy) circle (\r cm);
            \node[font=\scriptsize] at (\w*.5+\w*\xx,\h*.5+\h*\yy) {\i};
            }
        
            \draw[black] (\w*0.5,\h*1.5-\r)--(\w*0.5,\h*1.1)--(\w*1.5,\h*1.1)--(\w*1.5,\h*0.5+\r);
            \draw[black] (\w*3.5,\h*1.5-\r)--(\w*3.5,\h*.5+\r);
            \draw[black] (\w*4.5,\h*1.5-\r)--(\w*4.5,\h*.5+\r);
            \draw[black,-stealth] (\w*5.5,\h*1.5-\r)--(\w*5.5,\h*0.9)--(\w*6.3,\h*0.9);
            \draw[black] (-.3,\h*0.9)--(\w*0.5,\h*0.9)--(\w*0.5,\h*0.5+\r);
        
            \draw[black] (\w*2.5,\h*2.5-\r)--(\w*2.5,\h*1.9)--(\w*3.5,\h*1.9)--(\w*3.5,\h*1.5+\r);
            \draw[black] (\w*3.5,\h*2.5-\r)--(\w*3.5,\h*2)--(\w*4.5,\h*2)--(\w*4.5,\h*1.5+\r);
            \draw[black] (\w*4.5,\h*2.5-\r)--(\w*4.5,\h*2.1)--(\w*5.5,\h*2.1)--(\w*5.5,\h*1.5+\r);
        
            \draw[black] (\w*3.5,\h*3.5-\r)--(\w*3.5,\h*2.5+\r);
            \draw[black,-stealth] (\w*5.5,\h*3.5-\r)--(\w*5.5,\h*3.1)--(\w*6.3,\h*3.1);
            \draw[black] (-.3,\h*3.1)--(\w*2.5,\h*3.1)--(\w*2.5,\h*2.5+\r);
            \end{tikzpicture}};

        \node (E1) at (-6,0) {
            \begin{tikzpicture}[scale=0.7]
            \def \w{1};
            \def \h{1};
            \def \r{0.25};
            
            \foreach \i in {0,...,4}
            {
            \draw[gray!50] (0,\i*\h)--(\w*6,\i*\h);
            }
            \foreach \i in {0,...,6}
            {
            \draw[gray!50] (\w*\i,0)--(\w*\i,4*\h);
            }
            \foreach \xx\yy\i in {
            0/0/3,1/0/2,3/0/4,4/0/4,
            3/1/4,4/1/4,5/1/3,
            2/2/4,3/2/4,4/2/3,
            3/3/4,5/3/4}
            {
            \draw (\w*.5+\w*\xx,\h*.5+\h*\yy) circle (\r cm);
            \node[font=\scriptsize] at (\w*.5+\w*\xx,\h*.5+\h*\yy) {\i};
            }
            \foreach \xx\yy\i in {
            0/1/2}
            {
            \draw[Orange,thick] (\w*.5+\w*\xx,\h*.5+\h*\yy) circle (\r cm);
            \node[font=\scriptsize,Orange,thick] at (\w*.5+\w*\xx,\h*.5+\h*\yy) {\i};
            }
        
            \draw[black] (\w*0.5,\h*1.5-\r)--(\w*0.5,\h*1.1)--(\w*1.5,\h*1.1)--(\w*1.5,\h*0.5+\r);
            \draw[black] (\w*3.5,\h*1.5-\r)--(\w*3.5,\h*.5+\r);
            \draw[black] (\w*4.5,\h*1.5-\r)--(\w*4.5,\h*.5+\r);
            \draw[black,-stealth] (\w*5.5,\h*1.5-\r)--(\w*5.5,\h*0.9)--(\w*6.3,\h*0.9);
            \draw[black] (-.3,\h*0.9)--(\w*0.5,\h*0.9)--(\w*0.5,\h*0.5+\r);
        
            \draw[black] (\w*2.5,\h*2.5-\r)--(\w*2.5,\h*1.9)--(\w*4.5,\h*1.9)--(\w*4.5,\h*1.5+\r);
            \draw[black] (\w*3.5,\h*2.5-\r)--(\w*3.5,\h*1.5+\r);
            \draw[black] (\w*4.5,\h*2.5-\r)--(\w*4.5,\h*2.1)--(\w*5.5,\h*2.1)--(\w*5.5,\h*1.5+\r);
        
            \draw[black] (\w*3.5,\h*3.5-\r)--(\w*3.5,\h*2.5+\r);
            \draw[black,-stealth] (\w*5.5,\h*3.5-\r)--(\w*5.5,\h*3.1)--(\w*6.3,\h*3.1);
            \draw[black] (-.3,\h*3.1)--(\w*2.5,\h*3.1)--(\w*2.5,\h*2.5+\r);
            \end{tikzpicture}
        };

        \node (E0) at (6,0) {
            \begin{tikzpicture}[scale=0.7]
            \def \w{1};
            \def \h{1};
            \def \r{0.25};
            
            \foreach \i in {0,...,4}
            {
            \draw[gray!50] (0,\i*\h)--(\w*6,\i*\h);
            }
            \foreach \i in {0,...,6}
            {
            \draw[gray!50] (\w*\i,0)--(\w*\i,4*\h);
            }
            \foreach \xx\yy\i in {
            1/0/2,3/0/4,4/0/4,
            1/1/2,3/1/4,4/1/4,5/1/3,
            2/2/4,3/2/4,4/2/3,
            3/3/4,5/3/4}
            {
            \draw (\w*.5+\w*\xx,\h*.5+\h*\yy) circle (\r cm);
            \node[font=\scriptsize] at (\w*.5+\w*\xx,\h*.5+\h*\yy) {\i};
            }
    
            \foreach \xx\yy\i in {
            5/0/3}
            {
            \draw[Orange,thick] (\w*.5+\w*\xx,\h*.5+\h*\yy) circle (\r cm);
            \node[font=\scriptsize,Orange,thick] at (\w*.5+\w*\xx,\h*.5+\h*\yy) {\i};
            }
        
            \draw[black] (\w*1.5,\h*1.5-\r)--(\w*1.5,\h*.5+\r);
            \draw[black] (\w*3.5,\h*1.5-\r)--(\w*3.5,\h*.5+\r);
            \draw[black] (\w*4.5,\h*1.5-\r)--(\w*4.5,\h*.5+\r);
            \draw[black] (\w*5.5,\h*1.5-\r)--(\w*5.5,\h*.5+\r);
        
            \draw[black] (\w*2.5,\h*2.5-\r)--(\w*2.5,\h*1.9)--(\w*4.5,\h*1.9)--(\w*4.5,\h*1.5+\r);
            \draw[black] (\w*3.5,\h*2.5-\r)--(\w*3.5,\h*1.5+\r);
            \draw[black] (\w*4.5,\h*2.5-\r)--(\w*4.5,\h*2.1)--(\w*5.5,\h*2.1)--(\w*5.5,\h*1.5+\r);
        
            \draw[black] (\w*3.5,\h*3.5-\r)--(\w*3.5,\h*2.5+\r);
            \draw[black,-stealth] (\w*5.5,\h*3.5-\r)--(\w*5.5,\h*3.1)--(\w*6.3,\h*3.1);
            \draw[black] (-.3,\h*3.1)--(\w*2.5,\h*3.1)--(\w*2.5,\h*2.5+\r);
    
            \end{tikzpicture}
        };
        \draw[->] (M) to node[above] {$e_1$} (E1);
        \draw[->] (M) to node[above] {$e_0$} (E0);

        \node (E31) at (-6,-4) {
            \begin{tikzpicture}[scale=0.7]
            \def \w{1};
            \def \h{1};
            \def \r{0.25};
            
            \foreach \i in {0,...,4}
            {
            \draw[gray!50] (0,\i*\h)--(\w*6,\i*\h);
            }
            \foreach \i in {0,...,6}
            {
            \draw[gray!50] (\w*\i,0)--(\w*\i,4*\h);
            }
            \foreach \xx\yy\i in {
            0/0/3,1/0/2,3/0/4,4/0/4,
            0/1/2,3/1/4,4/1/4,5/1/3,
            2/2/4,3/2/4,4/2/3,
            5/3/4}
            {
            \draw (\w*.5+\w*\xx,\h*.5+\h*\yy) circle (\r cm);
            \node[font=\scriptsize] at (\w*.5+\w*\xx,\h*.5+\h*\yy) {\i};
            }
    
            \foreach \xx\yy\i in {
            2/3/4}
            {
            \draw[Orange,thick] (\w*.5+\w*\xx,\h*.5+\h*\yy) circle (\r cm);
            \node[font=\scriptsize,Orange,thick] at (\w*.5+\w*\xx,\h*.5+\h*\yy) {\i};
            }
        
            \draw[black] (\w*0.5,\h*1.5-\r)--(\w*0.5,\h*1.1)--(\w*1.5,\h*1.1)--(\w*1.5,\h*0.5+\r);
            \draw[black] (\w*3.5,\h*1.5-\r)--(\w*3.5,\h*.5+\r);
            \draw[black] (\w*4.5,\h*1.5-\r)--(\w*4.5,\h*.5+\r);
            \draw[black,-stealth] (\w*5.5,\h*1.5-\r)--(\w*5.5,\h*0.9)--(\w*6.3,\h*0.9);
            \draw[black] (-.3,\h*0.9)--(\w*0.5,\h*0.9)--(\w*0.5,\h*0.5+\r);
        
            \draw[black] (\w*2.5,\h*2.5-\r)--(\w*2.5,\h*1.9)--(\w*3.5,\h*1.9)--(\w*3.5,\h*1.5+\r);
            \draw[black] (\w*3.5,\h*2.5-\r)--(\w*3.5,\h*2)--(\w*4.5,\h*2)--(\w*4.5,\h*1.5+\r);
            \draw[black] (\w*4.5,\h*2.5-\r)--(\w*4.5,\h*2.1)--(\w*5.5,\h*2.1)--(\w*5.5,\h*1.5+\r);
        
            \draw[black] (\w*2.5,\h*3.5-\r)--(\w*2.5,\h*2.5+\r);
            \draw[black,-stealth] (\w*5.5,\h*3.5-\r)--(\w*5.5,\h*3.1)--(\w*6.3,\h*3.1);
            \draw[black] (-.3,\h*3.1)--(\w*3.5,\h*3.1)--(\w*3.5,\h*2.5+\r);
            \end{tikzpicture}
        };

        \node (E32) at (0,-4) {
            \begin{tikzpicture}[scale=0.7]
            \def \w{1};
            \def \h{1};
            \def \r{0.25};
            
            \foreach \i in {0,...,4}
            {
            \draw[gray!50] (0,\i*\h)--(\w*6,\i*\h);
            }
            \foreach \i in {0,...,6}
            {
            \draw[gray!50] (\w*\i,0)--(\w*\i,4*\h);
            }
            \foreach \xx\yy\i in {
            0/0/3,1/0/2,3/0/4,4/0/4,
            0/1/2,4/1/4,5/1/3,
            2/2/4,3/2/4,4/2/3,
            5/3/4}
            {
            \draw (\w*.5+\w*\xx,\h*.5+\h*\yy) circle (\r cm);
            \node[font=\scriptsize] at (\w*.5+\w*\xx,\h*.5+\h*\yy) {\i};
            }
    
            \foreach \xx\yy\i in {
            2/1/4,2/3/4}
            {
            \draw[Orange,thick] (\w*.5+\w*\xx,\h*.5+\h*\yy) circle (\r cm);
            \node[font=\scriptsize,Orange,thick] at (\w*.5+\w*\xx,\h*.5+\h*\yy) {\i};
            }
        
            \draw[black] (\w*0.5,\h*1.5-\r)--(\w*0.5,\h*1.1)--(\w*1.5,\h*1.1)--(\w*1.5,\h*0.5+\r);
            \draw[black] (\w*2.5,\h*1.5-\r)--(\w*2.5,\h*1.1)--(\w*3.5,\h*1.1)--(\w*3.5,\h*0.5+\r);
            \draw[black] (\w*4.5,\h*1.5-\r)--(\w*4.5,\h*.5+\r);
            \draw[black,-stealth] (\w*5.5,\h*1.5-\r)--(\w*5.5,\h*0.9)--(\w*6.3,\h*0.9);
            \draw[black] (-.3,\h*0.9)--(\w*0.5,\h*0.9)--(\w*0.5,\h*0.5+\r);
        
            \draw[black] (\w*2.5,\h*2.5-\r)--(\w*2.5,\h*1.5+\r);
            \draw[black] (\w*3.5,\h*2.5-\r)--(\w*3.5,\h*2)--(\w*4.5,\h*2)--(\w*4.5,\h*1.5+\r);
            \draw[black] (\w*4.5,\h*2.5-\r)--(\w*4.5,\h*2.1)--(\w*5.5,\h*2.1)--(\w*5.5,\h*1.5+\r);
        
            \draw[black] (\w*2.5,\h*3.5-\r)--(\w*2.5,\h*2.5+\r);
            \draw[black,-stealth] (\w*5.5,\h*3.5-\r)--(\w*5.5,\h*3.1)--(\w*6.3,\h*3.1);
            \draw[black] (-.3,\h*3.1)--(\w*3.5,\h*3.1)--(\w*3.5,\h*2.5+\r);
            \end{tikzpicture}
        };

        \node (E33) at (6,-4) {
            \begin{tikzpicture}[scale=0.7]
            \def \w{1};
            \def \h{1};
            \def \r{0.25};
            
            \foreach \i in {0,...,4}
            {
            \draw[gray!50] (0,\i*\h)--(\w*6,\i*\h);
            }
            \foreach \i in {0,...,6}
            {
            \draw[gray!50] (\w*\i,0)--(\w*\i,4*\h);
            }
            \foreach \xx\yy\i in {
            0/0/3,1/0/2,4/0/4,
            0/1/2,4/1/4,5/1/3,
            2/2/4,3/2/4,4/2/3,
            5/3/4}
            {
            \draw (\w*.5+\w*\xx,\h*.5+\h*\yy) circle (\r cm);
            \node[font=\scriptsize] at (\w*.5+\w*\xx,\h*.5+\h*\yy) {\i};
            }
    
            \foreach \xx\yy\i in {
            2/0/4,2/1/4,2/3/4}
            {
            \draw[Orange,thick] (\w*.5+\w*\xx,\h*.5+\h*\yy) circle (\r cm);
            \node[font=\scriptsize,Orange,thick] at (\w*.5+\w*\xx,\h*.5+\h*\yy) {\i};
            }
        
            \draw[black] (\w*0.5,\h*1.5-\r)--(\w*0.5,\h*1.1)--(\w*1.5,\h*1.1)--(\w*1.5,\h*0.5+\r);
            \draw[black] (\w*2.5,\h*1.5-\r)--(\w*2.5,\h*0.5+\r);
            \draw[black] (\w*4.5,\h*1.5-\r)--(\w*4.5,\h*.5+\r);
            \draw[black,-stealth] (\w*5.5,\h*1.5-\r)--(\w*5.5,\h*0.9)--(\w*6.3,\h*0.9);
            \draw[black] (-.3,\h*0.9)--(\w*0.5,\h*0.9)--(\w*0.5,\h*0.5+\r);
        
            \draw[black] (\w*2.5,\h*2.5-\r)--(\w*2.5,\h*1.5+\r);
            \draw[black] (\w*3.5,\h*2.5-\r)--(\w*3.5,\h*2)--(\w*4.5,\h*2)--(\w*4.5,\h*1.5+\r);
            \draw[black] (\w*4.5,\h*2.5-\r)--(\w*4.5,\h*2.1)--(\w*5.5,\h*2.1)--(\w*5.5,\h*1.5+\r);
        
            \draw[black] (\w*2.5,\h*3.5-\r)--(\w*2.5,\h*2.5+\r);
            \draw[black,-stealth] (\w*5.5,\h*3.5-\r)--(\w*5.5,\h*3.1)--(\w*6.3,\h*3.1);
            \draw[black] (-.3,\h*3.1)--(\w*3.5,\h*3.1)--(\w*3.5,\h*2.5+\r);
            \end{tikzpicture}
        };
        \draw[->] (M) to node[above]{$e_3$} (E31);
        \draw[->] (M) to node[right]{$e_3^2$} (E32);
        \draw[->] (M) to node[above]{$e_3^3$} (E33);
        \end{tikzpicture}
    \end{center}
    
\end{example}

\begin{lemma}
    \label{thm:balancing-condition-typeA}
    The type $A$ crystal process $\mlqA(\lambda,n)$ has uniform stationary distribution.
\end{lemma}
\begin{proof}
  For $M\in\MLQ_A(\lambda,n)$, let
\[F(M)\coloneq\{f_i^j\colon1\leq j\leq \varphi_i(M),\ 0\leq i\leq n-1\}.\]
The set of uniform transitions $M\to f_i^j M$ over all $f_i^j\in F(M)$ defines a time reversal process of $\mlqA(\lambda,n)$.
To prove $\mlqA(\lambda,n)$ has uniform distribution, we verify the balance condition for each $M\in\MLQ_A(\lambda,n)$: 
\[|E(M)|= |\{M'\in\MLQ_A(\lambda,n)\colon\exists e_i^j\in E(M')\ \text{such that }M=e_i^j M'\}|=|F(M)|.
\]
Hence the balance condition, implied by the equality $|E(M)|= |F(M)|$, follows from the following sum, with the convention $0\equiv n$:
\[ |F(M)|-|E(M)|=\sum_{i=0}^{n-1} \varphi_i(M)-\varepsilon_i(M)=\sum_{i=0}^{n-1} |\{r:i\in B_r\}|-|\{r:i+1\in B_r\}|=0.\]
We justify the second equality by computing $\varphi_i(M)-\varepsilon_i(M)$ directly. By definition, $\varphi_i(M)$ and $\varepsilon_i(M)$ are the numbers of unmatched $i$'s and $i+1$'s, respectively, in $\pth_i(M)$, which is indeed given by the difference in the number of particles between columns $i$ and $i+1$, since the matched particles contribute equally on both sides. 
\end{proof}

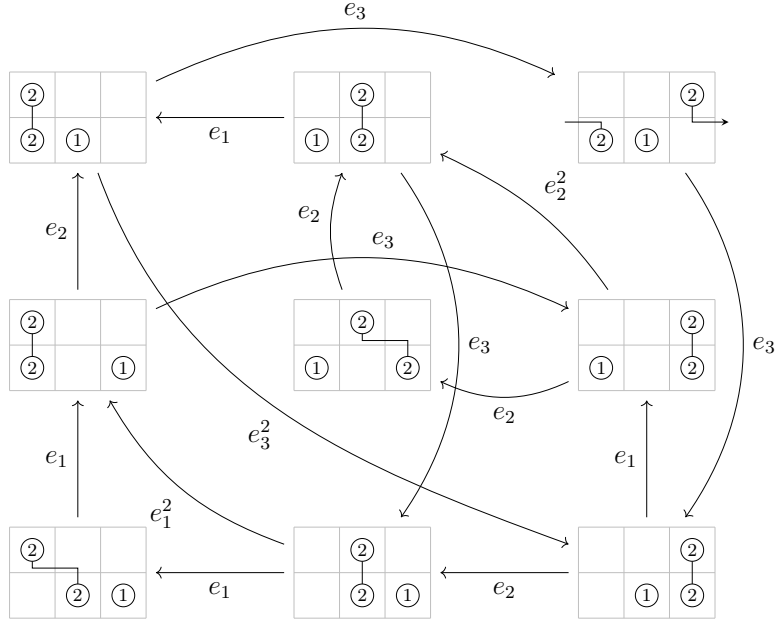
\begin{figure}[ht]
\centering
        \resizebox{0.55\linewidth}{!}{\begin{tikzpicture}[scale=0.875]
            \node (A) at (0,0)
            {\begin{tikzpicture}[scale=0.7]
            \def \w{1};
            \def \h{1};
            \def \r{0.25};
            
            \foreach \i in {0,...,2}
            {
            \draw[gray!50] (0,\i*\h)--(\w*3,\i*\h);
            }
            \foreach \i in {0,...,3}
            {
            \draw[gray!50] (\w*\i,0)--(\w*\i,2*\h);
            }
            \foreach \xx\yy\i in {
            0/0/2,1/0/1,0/1/2}
            {
            \draw (\w*.5+\w*\xx,\h*.5+\h*\yy) circle (\r cm);
            \node[font=\scriptsize] at (\w*.5+\w*\xx,\h*.5+\h*\yy) {\i};
            }
        
            \draw[black] (\w*0.5,\h*1.5-\r)--(\w*0.5,\h*0.5+\r);
            \end{tikzpicture}
            };

            \node (B) at (5,0)
            {\begin{tikzpicture}[scale=0.7]
            \def \w{1};
            \def \h{1};
            \def \r{0.25};
            
            \foreach \i in {0,...,2}
            {
            \draw[gray!50] (0,\i*\h)--(\w*3,\i*\h);
            }
            \foreach \i in {0,...,3}
            {
            \draw[gray!50] (\w*\i,0)--(\w*\i,2*\h);
            }
            \foreach \xx\yy\i in {
            0/0/1,1/0/2,1/1/2}
            {
            \draw (\w*.5+\w*\xx,\h*.5+\h*\yy) circle (\r cm);
            \node[font=\scriptsize] at (\w*.5+\w*\xx,\h*.5+\h*\yy) {\i};
            }
        
            \draw[black] (\w*1.5,\h*1.5-\r)--(\w*1.5,\h*0.5+\r);
            \end{tikzpicture}
            };

            \node (C) at (10,0)
            {\begin{tikzpicture}[scale=0.7]
            \def \w{1};
            \def \h{1};
            \def \r{0.25};
            
            \foreach \i in {0,...,2}
            {
            \draw[gray!50] (0,\i*\h)--(\w*3,\i*\h);
            }
            \foreach \i in {0,...,3}
            {
            \draw[gray!50] (\w*\i,0)--(\w*\i,2*\h);
            }
            \foreach \xx\yy\i in {
            0/0/2,1/0/1,2/1/2}
            {
            \draw (\w*.5+\w*\xx,\h*.5+\h*\yy) circle (\r cm);
            \node[font=\scriptsize] at (\w*.5+\w*\xx,\h*.5+\h*\yy) {\i};
            }

            \draw[black,-stealth] (\w*2.5,\h*1.5-\r)--(\w*2.5,\h*0.9)--(\w*3.3,\h*0.9);
            \draw[black] (-.3,\h*0.9)--(\w*0.5,\h*0.9)--(\w*0.5,\h*0.5+\r);
            \end{tikzpicture}
            };

            \node (D) at (0,-4)
            {\begin{tikzpicture}[scale=0.7]
            \def \w{1};
            \def \h{1};
            \def \r{0.25};
            
            \foreach \i in {0,...,2}
            {
            \draw[gray!50] (0,\i*\h)--(\w*3,\i*\h);
            }
            \foreach \i in {0,...,3}
            {
            \draw[gray!50] (\w*\i,0)--(\w*\i,2*\h);
            }
            \foreach \xx\yy\i in {
            0/0/2,2/0/1,0/1/2}
            {
            \draw (\w*.5+\w*\xx,\h*.5+\h*\yy) circle (\r cm);
            \node[font=\scriptsize] at (\w*.5+\w*\xx,\h*.5+\h*\yy) {\i};
            }
        
            \draw[black] (\w*0.5,\h*1.5-\r)--(\w*0.5,\h*0.5+\r);
            \end{tikzpicture}
            };

            \node (E) at (5,-4)
            {\begin{tikzpicture}[scale=0.7]
            \def \w{1};
            \def \h{1};
            \def \r{0.25};
            
            \foreach \i in {0,...,2}
            {
            \draw[gray!50] (0,\i*\h)--(\w*3,\i*\h);
            }
            \foreach \i in {0,...,3}
            {
            \draw[gray!50] (\w*\i,0)--(\w*\i,2*\h);
            }
            \foreach \xx\yy\i in {
            0/0/1,2/0/2,1/1/2}
            {
            \draw (\w*.5+\w*\xx,\h*.5+\h*\yy) circle (\r cm);
            \node[font=\scriptsize] at (\w*.5+\w*\xx,\h*.5+\h*\yy) {\i};
            }
        
            \draw[black] (\w*1.5,\h*1.5-\r)--(\w*1.5,\h*1.1)--(\w*2.5,\h*1.1)--(\w*2.5,\h*0.5+\r);
            \end{tikzpicture}
            };

            \node (F) at (10,-4)
            {\begin{tikzpicture}[scale=0.7]
            \def \w{1};
            \def \h{1};
            \def \r{0.25};
            
            \foreach \i in {0,...,2}
            {
            \draw[gray!50] (0,\i*\h)--(\w*3,\i*\h);
            }
            \foreach \i in {0,...,3}
            {
            \draw[gray!50] (\w*\i,0)--(\w*\i,2*\h);
            }
            \foreach \xx\yy\i in {
            0/0/1,2/0/2,2/1/2}
            {
            \draw (\w*.5+\w*\xx,\h*.5+\h*\yy) circle (\r cm);
            \node[font=\scriptsize] at (\w*.5+\w*\xx,\h*.5+\h*\yy) {\i};
            }

            \draw[black] (\w*2.5,\h*1.5-\r)--(\w*2.5,\h*0.5+\r);
            \end{tikzpicture}
            };

            \node (G) at (0,-8)
            {\begin{tikzpicture}[scale=0.7]
            \def \w{1};
            \def \h{1};
            \def \r{0.25};
            
            \foreach \i in {0,...,2}
            {
            \draw[gray!50] (0,\i*\h)--(\w*3,\i*\h);
            }
            \foreach \i in {0,...,3}
            {
            \draw[gray!50] (\w*\i,0)--(\w*\i,2*\h);
            }
            \foreach \xx\yy\i in {
            1/0/2,2/0/1,0/1/2}
            {
            \draw (\w*.5+\w*\xx,\h*.5+\h*\yy) circle (\r cm);
            \node[font=\scriptsize] at (\w*.5+\w*\xx,\h*.5+\h*\yy) {\i};
            }
        
            \draw[black] (\w*0.5,\h*1.5-\r)--(\w*0.5,\h*1.1)--(\w*1.5,\h*1.1)--(\w*1.5,\h*0.5+\r);
            \end{tikzpicture}
            };

            \node (H) at (5,-8)
            {\begin{tikzpicture}[scale=0.7]
            \def \w{1};
            \def \h{1};
            \def \r{0.25};
            
            \foreach \i in {0,...,2}
            {
            \draw[gray!50] (0,\i*\h)--(\w*3,\i*\h);
            }
            \foreach \i in {0,...,3}
            {
            \draw[gray!50] (\w*\i,0)--(\w*\i,2*\h);
            }
            \foreach \xx\yy\i in {
            1/0/2,2/0/1,1/1/2}
            {
            \draw (\w*.5+\w*\xx,\h*.5+\h*\yy) circle (\r cm);
            \node[font=\scriptsize] at (\w*.5+\w*\xx,\h*.5+\h*\yy) {\i};
            }
        
            \draw[black] (\w*1.5,\h*1.5-\r)--(\w*1.5,\h*0.5+\r);
            \end{tikzpicture}
            };

            \node (I) at (10,-8)
            {\begin{tikzpicture}[scale=0.7]
            \def \w{1};
            \def \h{1};
            \def \r{0.25};
            
            \foreach \i in {0,...,2}
            {
            \draw[gray!50] (0,\i*\h)--(\w*3,\i*\h);
            }
            \foreach \i in {0,...,3}
            {
            \draw[gray!50] (\w*\i,0)--(\w*\i,2*\h);
            }
            \foreach \xx\yy\i in {
            1/0/1,2/0/2,2/1/2}
            {
            \draw (\w*.5+\w*\xx,\h*.5+\h*\yy) circle (\r cm);
            \node[font=\scriptsize] at (\w*.5+\w*\xx,\h*.5+\h*\yy) {\i};
            }

            \draw[black] (\w*2.5,\h*1.5-\r)--(\w*2.5,\h*0.5+\r);
            \end{tikzpicture}
            };

            \draw[->,out = 25, in = 155] (A) to node[above] {$e_3$} (C);
            \draw[->,out = 290, in = 160] (A) to node[below left] {$e_3^2$} (I);

            \draw[->,out = 305, in = 55] (B) to node[right] {$e_3$} (H);
            \draw[->,out = 180, in = 0] (B) to node[below] {$e_1$} (A);

            \draw[->,out = 305, in = 55] (C) to node[right] {$e_3$} (I);

            \draw[->,out = 90, in = 270] (D) to node[left] {$e_2$} (A);
            \draw[->,out = 25, in = 155] (D) to node[above right] {$e_3$} (F);

            \draw[->,out = 110, in = 250] (E) to node[above left] {$e_2$} (B);

            \draw[->,out = 125, in = 335] (F) to node[above right] {$e_2^2$} (B);
            \draw[->,out = 205, in = 335] (F) to node[below] {$e_2$} (E);

            \draw[->,out = 90, in = 270] (G) to node[left] {$e_1$} (D);
            
            \draw[->,out = 180, in = 0] (H) to node[below] {$e_1$} (G);
            \draw[->,out = 160, in = 300] (H) to node[below left] {$e_1^2$} (D);

            \draw[->,out = 180, in = 0] (I) to node[below] {$e_2$} (H);
            \draw[->,out = 90, in = 270] (I) to node[left] {$e_1$} (F);

        \end{tikzpicture}
        }
       \caption{The state diagram of the crystal Markov chain $\mlqA((2,1),3).$}\label{ex:21}
       \end{figure}

\begin{theorem}\label{thm:crystal ops MC A}
    The type $A$ crystal process on $\MLQ_A(\lambda,n)$ projects to the TASEP chain on $\States(\lambda,n)$ via the map $\pi$.
\end{theorem}

\begin{proof}
    For a vector $v=(v_1,\ldots,v_n)$, define
\[s_i\cdot v=\begin{cases}(v_1,\ldots,v_{i+1},v_i,\ldots,v_n)&1\leq i\leq n-1,\\
(v_n,v_2,\ldots,v_{n-1},v_1)&i=0.
\end{cases}
\]
    
    According to \cref{rem:equiv_conds_lumping}, to prove the claim it suffices to verify that for each $M\in\MLQ_A(\lambda,n)$, the following two properties hold for each $i\in\{0,\ldots,n-1\}$ and $1\leq j\leq \varepsilon_i(M)$. \\

    \begin{enumerate}
        \item[(A)] If $j<\varepsilon_i(M)$ or $\pi(M)_i\geq \pi(M)_{i+1}$, then $\pi(e_i^j M)=\pi(M)$.
        \item[(B)] If $j=\varepsilon_i(M)$ and $\pi(M)_i<\pi(M)_{i+1}$, then $\pi(e_i^j M)=s_i\;\cdot\;\pi(M)$.
    \end{enumerate}
    \vspace{1em}
    Both properties follow from \cite{MNS26}: property (A) is due to \cite[Lemmas~3.2~and~3.6]{MNS26}, whereas property (B) follows from \cite[Lemma~3.10]{MNS26}.
    However, we shall give an independent proof via a crystal argument, as this approach extends to the type $C$ setting for an analogous projection map defined by a corner transfer matrix. Let $M=(b_1,\ldots,b_L)$.

    For $2\leq k \leq L$, define the following maps 
       \begin{align*}
       Q_{[1,k-1]}(M) &\coloneq R(b_1\otimes\cdots R(b_{k-2}\otimes R(b_{k-1}\otimes b_k))\cdots)\;,\\
 R_{[1,k-1]}(M) &\coloneq R_{(1)}(b_1\otimes\cdots R_{(1)}(b_{k-2}\otimes R_{(1)}(b_{k-1}\otimes b_k))\cdots)\;,
    \end{align*}
    and for $k=1$ define $R_{[1,0]}(M) \coloneq b_1.$ In particular, $R_{[1,k-1]}(M)$ is equal to the first factor of $Q_{[1,k-1]}(M)$. 
    Recall from \cref{def:ctm_typeA} that the corner transfer matrix projection is given by 
    \[\pi(M)=\sum_{k=1}^{L}\pi_k(M),\qquad\qquad \pi_k(M)=\wt\left(R_{[1,k-1]}(M)\right)\ \text{for $k\geq 2$},\]
    with $\pi_1(M) = \wt(b_1).$ To simplify notation we write $u_k = \pi_k(M)$. Recall that these indicator vectors are nested, that is, $u_L\subseteq \cdots\subseteq u_1$. We will compare these indicator vectors with the corresponding computations for $e_i^j M$, denoted $u_k' = \pi_k(e_i^j M)$, with $\pi(e_i^j M)=u_1'+\cdots+u_L'$.

    Since the combinatorial R matrix commutes with the crystal operators, for all $1\leq k \leq L $ and fixed $ 0\leq i \leq n-1 $ we have $\varepsilon_i\left(Q_{[1,k-1]}(M)\right)=\varepsilon_i(M)$, and
    \begin{align*}
        R_{[1,k-1]}\left(e_i^j M\right)
        &=\begin{cases}
        R_{[1,k-1]}(M)&e_i^j\ \text{does not act on the first factor of }Q_{[1,k-1]}(M)\;,\\&\\
        e_i\Big(R_{[1,k-1]}(M)\Big)&e_i^j\ \text{acts on the first factor of }Q_{[1,k-1]}(M).
    \end{cases}
    \end{align*}
    The first scenario implies $u'_k=u_k$, while the second implies $u'_k=s_i\;\cdot\;u_k$. We use the convention $(u_k)_0 \coloneq (u_k)_{n}$. Now, the first case occurs if either $((u_k)_i,(u_k)_{i+1})=(0,1)$ with $j<\varepsilon_i(M)$ (meaning, $e_i^j$ acts strictly above the first factor of $Q_{[1,k-1]}(M)$) or $(u_k)_i\geq (u_k)_{i+1}$ (meaning, $e_i$ cannot act on the first factor of $Q_{[1,k-1]}(M)$). Conversely, the second case requires $j=\varepsilon_i(M)$ and $((u_k)_i,(u_k)_{i+1})=(0,1)$.
    Since this holds for all $1\leq k\leq L$ and the $u_k$'s and $u'_k$'s are nested, we obtain
    \[\pi(e_i^j M)=\begin{cases}
        \pi(M)&j<\varepsilon_i(M)\quad\text{ or }\quad\pi(M)_i\geq\pi(M)_{i+1}\;,\\
        s_i\;\cdot\;\pi(M)&j=\varepsilon_i(M)\quad\text{ and }\quad\pi(M)_i<\pi(M)_{i+1}.
    \end{cases}\]
\end{proof}

\begin{example}
   We illustrate the properties (A) and (B) in the proof of \cref{thm:crystal ops MC A} with \cref{ex:typeA crystal}. Recall from \cref{prop:proj equals pi} that $\pi=\proj_{\FM}$.
   
   \begin{itemize}
   \item For the transition $M\to e_1 M$, since $\proj_{\FM}(M)=\proj_{\FM}(e_1 M)= (3,2,0,4,4,0)$, the transition is invisible in the TASEP chain. Indeed, the state $(3,2,0,4,4,0)$ does not admit a transition between sites 1 and 2, showing the case $\tau_i\geq \tau_{i+1}$ in property (A). 
   \item For $M\to e_3 M$ and $M\to e_3^2 M$, we have $\proj_{\FM}(M)=\proj_{\FM}(e_3 M)=\proj_{\FM}(e_3^2 M)$, showing the case $p<\varepsilon_i(M)$ in (A). 
   \item On the other hand, $\proj_{\FM}(e_0 M)=(0,2,5,0,4,3)=s_0\cdot \proj_{\FM}(M)$, and $\proj_{\FM}(e_3^3 M)=(3,2,5,0,4,0)=s_3\cdot \proj_{\FM}(M)$, both corresponding to valid transitions in the TASEP chain. Since $\varepsilon_0(M)=1$ and $\varepsilon_3(M)=3$, this shows property (B).
   \end{itemize}
\end{example}

Notably, the type $A$ crystal process on $\MLQ(\lambda,n)$ answers in the affirmative to a question posed in \cite[Question 5.10]{ManScrim24} of whether there exists a Markov chain on multiline queues that intertwines the combinatorial $R$ matrix.

\subsection{Projection map on type $C$ multiline queues}

In this section, we define a type $C$ analogue of the FM algorithm on type $C$ multiline queues to obtain a projection map $\projC\colon\MLQ_C(\lambda,n)\to \StatesC(\lambda,n)$ that lumps the type $C$ multiline queue chain to the open boundary multispecies TASEP. Our procedure closely follows \cref{def:FM-algorithm}, but with type $C$ pairing from \cref{def:type-c-pairing} replacing type $A$ pairing (the NY rule).

\begin{definition}[Queueing procedure for a pair of rows in type $C$]\label{def:queueing}
Fix the label set $\{1,\ldots,L\}$ and let $r\in[L-1]$. Suppose $(\ab, \bb)$ is a pair of rows of size $n$ with the particles $\{\cir\,,\,\squ\}$, with row $\ab$ having at least as many particles as row $\bb$. Suppose the particles and vacancies in row $\bb$ are labeled with $u=(u_1,\ldots,u_n)$ where $u_i \in \{r+1,r+2,\ldots,L\}$ if $b_i\in\{\cir\,,\,\squ\}$ and $u_i=0$ otherwise. We give a procedure to produce a labeling $v=(v_1,\ldots,v_n)$ on the sites of $\ab$ as follows. 

For $\ell=L,L-1,\ldots,r+1$, in that order, let $\bb^{(\ell)}$ be the set of particles with label $\ell$ in $\bb$, and let $\ab^{(\ell)}$ be the set of unlabeled particles in $\ab$ at that stage (sites $i$ such that $a_i\in\{\cir\,,\,\squ\}$ where $v_i$ has not yet been set). Then, for each $k\in \cPair\left( \ab^{(\ell)},\bb^{(\ell)} \right)$, set $v_k = \ell$. Set $v_i=0$ for all vacant sites $i$ in $\ab$.

Once this procedure is completed, for any remaining sites $i$ such that $a_i\in\{\cir,\squ\}$ and $v_i$ has not been set, set $v_i=r$. Then $v$ is the labeling on $\ab$ produced by the queueing process.   
\end{definition}

\begin{remark}
    Although we assign all labels of a given type simultaneously using the pairing map $\cPair$, one may equivalently present the queueing procedure using queueing language, by treating the particles in $\bb$ as arrivals and the particles of $\ab$ as possible departures. The labels impose a priority order in which arrivals with higher labels are processed first, and among all arrivals with the same label $\ell$, any order can be chosen. Following \cref{rem:type C via type A}, type $C$ pairing has the same order-independence property as type $A$: any order of processing particles that respects priority of higher labels produces the same set of departures in $\ab$, and hence the same labeling $v$, although the individual pairings may differ.
\end{remark}

\begin{example}\label{ex:queueing-labels-type-C}
    We show an example of the queueing process for $\bb = \left( \cir\,,\,\squ\,,\,\cdot\,,\,\cir\,,\,\cdot\,,\,\cir\,,\,\squ\,,\,\cdot\,,\,\cdot\,,\,\squ \right)$ and $\ab = (\cdot\,,\,\squ\,,\,\cir\,,\,\squ\,,\,\cdot\,,\,\squ\,,\,\cir\,,\,\cir\,,\,\cir\,,\,\squ)$. Let $r=1$, and let $u=(3,4,0,4,0,6,6,0,0,4)$ be the labels of $\bb$. Type-$C$ pairing particles in the priority-respecting order marked below (sites $6,7,2,4,10,1$), we obtain the labeling  $v=(0,3,5,3,0,2,1,1,3,5)$ on $\ab$.
    \def\sca{0.7}
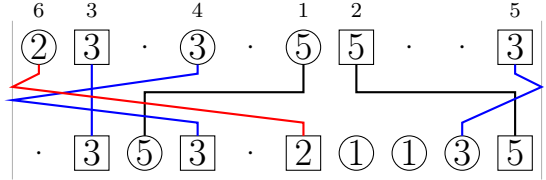
    \begin{center}

        \begin{tikzpicture}[scale=\sca]
        \def \w{1};
        \def \h{1};
        \def \r{0.325};
        \draw[gray!60] (0,0)--(0,3);
        \draw[gray!60] (10,0)--(10,3);

        \foreach \xx\yy\ll in {
        0/2/2,3/2/3,5/2/5}
        {
        \draw (\w*.5+\w*\xx,\h*.5+\h*\yy) circle (\r cm);
        \node at (\w*.5+\w*\xx,\h*.5+\h*\yy) {\large \ll};
        }
        \foreach \xx\yy\ll in {
        1/2/3,6/2/5,9/2/3}
        {
        \node at (\w*.5+\w*\xx,\h*.5+\h*\yy) {$\squa{2*\r*\sca}$};
        \node at (\w*.5+\w*\xx,\h*.5+\h*\yy) {\large \ll};
        }

        \foreach \xx\yy\ll in {
        2/0/5,6/0/1,7/0/1,8/0/3}
        {
        \draw (\w*.5+\w*\xx,\h*.5+\h*\yy) circle (\r cm);
        \node at (\w*.5+\w*\xx,\h*.5+\h*\yy) {\large \ll};
        }
        \foreach \xx\yy\ll in {
        1/0/3,3/0/3,5/0/2,9/0/5}
        {
        \node at (\w*.5+\w*\xx,\h*.5+\h*\yy) {$\squa{2*\r*\sca}$};
        \node at (\w*.5+\w*\xx,\h*.5+\h*\yy) {\large \ll};
        }
        
        \foreach \xx\yy in {
        0/0,4/0,2/2,4/2,7/2,8/2}
        {
        \node at (\w*.5+\w*\xx,\h*.5+\h*\yy) {$\cdot$};
        }

        \draw[black,thick] (\w*6.5,\h*2.5-\r)--(\w*6.5,\h*1.65)--(\w*9.5,\h*1.65)--(\w*9.5,\h*.5+\r);
        \draw[black,thick] (\w*5.5,\h*2.5-\r)--(\w*5.5,\h*1.65)--(\w*2.5,\h*1.65)--(\w*2.5,\h*.5+\r);
        
        \draw[blue,thick] (\w*1.5,\h*2.5-\r)--(\w*1.5,\h*.5+\r);
        \draw[blue,thick] (\w*3.5,\h*2.5-\r)--(\w*3.5,\h*2)--(\w*0,\h*1.5)--(\w*3.5,\h*0.75+\r)--(\w*3.5,\h*.5+\r);
        \draw[blue,thick] (\w*9.5,\h*2.5-\r)--(\w*9.5,\h*2)--(\w*10,\h*1.75)--(\w*8.5,\h*0.75+\r)--(\w*8.5,\h*.5+\r);

        \draw[red,thick] (\w*0.5,\h*2.5-\r)--(\w*0.5,\h*2)--(\w*0,\h*1.75)--(\w*5.5,\h*0.75+\r)--(\w*5.5,\h*.5+\r);

        \foreach \i\s in {6/1,7/2,2/3,4/4,10/5,1/6}
        {
        \node at (\w*\i-\w*0.5,\h*2+1.2) {\scriptsize $\s$};
        }
        \end{tikzpicture}
    \end{center}
    
\end{example}

Define the map $\iota\colon\{\,\cir\,,\,\squ\,,\Cdot\}^* \to \{1,-1,0\}^*$ by sending the letters $(\,\cir\,,\,\squ\,,\Cdot)$ to $(1,-1,0)$.

\begin{definition}[Projection of type $C$ MLQs]\label{def:projectionMLQS-C-TASEP-C}
    We define the \emph{projection map} 
    \[\projC\colon\MLQ_C(\lambda,n)\to \StatesC(\lambda,n)
    \]
    as follows. Let $M=(\bb_1,\bb_2,\ldots,\bb_{2L})\in\MLQC(\lambda,n)$ where $L=\lambda_1$. (We relabel the rows from $1,1',\ldots,L,L'$ to $1,\ldots,2L$ to simplify notation.) We produce a labeling $\textbf{L}=(U_1,U_2,\ldots,U_{2L})$ of $M$ with $U_j = (u_{j1},\ldots,u_{jn})$ labeling row $\bb_j$ of $M$.
    
    Initialize $U_{2L}$ to have the label $L$ at each site $i\in \bb_L$, and 0 otherwise. Next, the labels are propagated downward via the queueing procedure in \cref{def:queueing}: at each iteration $h = 2L-1,\ldots,2,1$, given the labeling $U_{h+1}$ on row $\bb_{h+1}$, the labeling $U_{h}$ on the row $\bb_{h}$ is completed by setting any remaining unlabeled particles to have the label $r=h/2$. Note that this is well-defined since only even numbered rows can have remaining unlabeled particles, since the number of particles in rows $b_{2j}$ and $b_{2j-1}$ is equal for $1\leq j\leq L$. Finally, set $(e_1,\ldots,e_n) \coloneq \iota(\bb_1)\in\{1,-1,0\}^n$ and define $$\projC(M) = (e_1\,u_{11},e_2\,u_{12},\ldots,e_n\,u_{1n}).$$ That is, each entry of $U_1$ is given a plus or minus sign according to whether the it is labeling a $\cir$ or $\squ.$
\end{definition}

\begin{lemma}
    For $M=(\bb_1,\bb_1',\ldots,\bb_{L},\bb_{L}')\in\MLQ_C(\lambda,n)$, the projection $\projC(M)$ is well-defined as an element of $\StatesC(\lambda,n)$. 
\end{lemma}
\begin{proof}
    For $1\leq j\leq L$, let $\mathcal{L}_j$ be the multiset of labels of row $\bb_j$ (equivalently, row $\bb_j'$) of $M$ according to \cref{def:projectionMLQS-C-TASEP-C}.     Denote by $\{\ell^{m}\}$ the multiset containing the label $\ell$ with multiplicity $m$. It suffices to show that for $1\leq j\leq L$, the label $j$ occurs in $\mathcal{L}_1$ with multiplicity $m_j(\lambda)=\lambda_j'-\lambda_{j+1}'$, where $\lambda_{L+1}'=0$.

    For the base case $j=L$, $\mathcal{L}_L=\{L^{\lambda_L'}\}$. For $1\leq j<L$, the label $j$ is added to the multiset of labels during the pairing from row $\bb_{j+1}$ to $\bb_j'$; thus $\mathcal{L}_j=\mathcal{L}_{j+1}\sqcup\{j^{\lambda_j'-\lambda_{j+1}'}\}$, and hence by induction,
    \[\mathcal{L}_1=\{L^{\lambda_{L}'}\}\cup\{(L-1)^{\lambda_{L-1}'-\lambda_{L}'}\}\cup\cdots\cup\{1^{\lambda_1'-\lambda_2'}\},\]
    as desired.
\end{proof}

\begin{example}
    \cref{fig:TypeC-MLQ-pairing} shows a type $C$ multiline queue $M\in\MLQC(\lambda,10)$ for $\lambda = (5^2,3^3,2^1,1^2)$. The pairing between rows $2$ and $1'$ corresponds to the one explained in \cref{ex:queueing-labels-type-C}. We have $$\projC(M) = (0,0,5,\overline{3},\overline{3},0,1,1,3,\overline{2},\overline{5}) \in \StatesC(\lambda,10).$$ 

    \def\sca{0.7}
    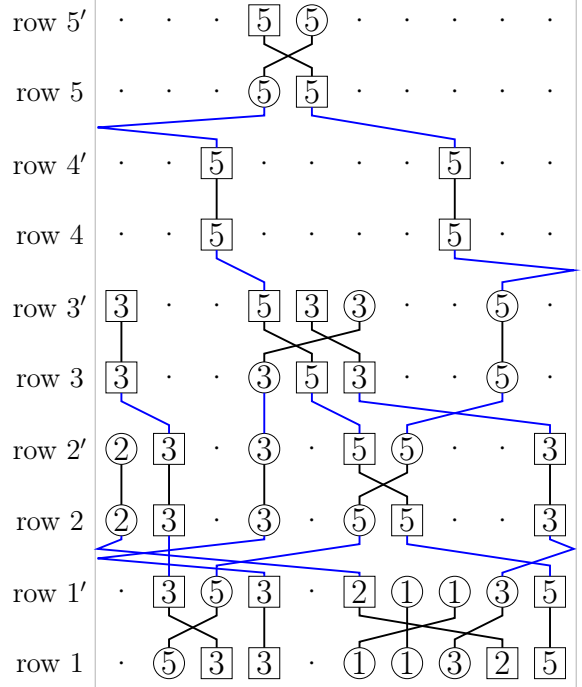
\begin{figure}[t]
        \centering
        \scalebox{0.9}{
        \begin{tikzpicture}[scale=\sca]
        \def \w{1};
        \def \h{1};
        \def \r{0.325};
        \draw[gray!60] (-0.05,0)--(-0.05,14.5);
        \draw[gray!60] (10.05,0)--(10.05,14.5);

        \foreach \xx\yy\ll in {
        1/0/5,5/0/1,6/0/1,7/0/3,
        2/1.5/5,6/1.5/1,7/1.5/1,8/1.5/3,
        0/3/2,3/3/3,5/3/5,
        0/4.5/2,3/4.5/3,6/4.5/5,
        3/6/3,8/6/5,
        8/7.5/5,5/7.5/3,
        3/12/5,4/13.5/5}
        {
        \draw (\w*.5+\w*\xx,\h*.5+\h*\yy) circle (\r cm);
        \node at (\w*.5+\w*\xx,\h*.5+\h*\yy) {\large \ll};
        }
        \foreach \xx\yy\ll in { 
        2/0/3,3/0/3,8/0/2,9/0/5,
        1/1.5/3,3/1.5/3,5/1.5/2,9/1.5/5,
        1/3/3,6/3/5,9/3/3,
        1/4.5/3,5/4.5/5,9/4.5/3,
        0/6/3,4/6/5,5/6/3,
        0/7.5/3,3/7.5/5,4/7.5/3,
        2/9/5,7/9/5,
        2/10.5/5,7/10.5/5,
        4/12/5,3/13.5/5
        }
        {
        \node at (\w*.5+\w*\xx,\h*.5+\h*\yy) {$\squa{2*\r*\sca}$};
        \node at (\w*.5+\w*\xx,\h*.5+\h*\yy) {\large \ll};
        }
        \foreach \xx\yy in {
        0/0,0/1.5,4/0,4/1.5,
        2/3,4/3,7/3,8/3,2/4.5,4/4.5,7/4.5,8/4.5,
        1/6,2/6,6/6,7/6,9/6,1/7.5,2/7.5,6/7.5,7/7.5,9/7.5,
        0/9,1/9,3/9,4/9,5/9,6/9,8/9,9/9,0/10.5,1/10.5,3/10.5,4/10.5,5/10.5,6/10.5,8/10.5,9/10.5,
        0/12,1/12,2/12,5/12,6/12,7/12,8/12,9/12,0/13.5,1/13.5,2/13.5,5/13.5,6/13.5,7/13.5,8/13.5,9/13.5
        }
        {
        \node at (\w*.5+\w*\xx,\h*.5+\h*\yy) {$\cdot$};
        }

        \foreach \yy in {1,2,3,4,5}
        {
        \node at (-1,3*\h*\yy-2.5*\h) {row $\yy$};
        \node at (-1,\h*3*\yy-1*\h) {row $\yy'$};
        }


        \draw[black,thick] (\w*1.5,\h*2-\r)--(\w*1.5,\h*1.5)--(\w*2.5,\h*1)--(\w*2.5,\h*.5+\r);
        \draw[black,thick] (\w*2.5,\h*2-\r)--(\w*2.5,\h*1.5)--(\w*1.5,\h*1)--(\w*1.5,\h*.5+\r);
        \draw[black,thick] (\w*3.5,\h*2-\r)--(\w*3.5,\h*.5+\r);
        \draw[black,thick] (\w*5.5,\h*2-\r)--(\w*5.5,\h*1.5)--(\w*8.5,\h*1)--(\w*8.5,\h*.5+\r);
        \draw[black,thick] (\w*6.5,\h*2-\r)--(\w*6.5,\h*.5+\r);
        \draw[black,thick] (\w*7.5,\h*2-\r)--(\w*7.5,\h*1.45)--(\w*5.5,\h*0.95)--(\w*5.5,\h*.5+\r);
        \draw[black,thick] (\w*8.5,\h*2-\r)--(\w*8.5,\h*1.5)--(\w*7.5,\h*1)--(\w*7.5,\h*.5+\r);
        \draw[black,thick] (\w*9.5,\h*2-\r)--(\w*9.5,\h*.5+\r);
        \draw[black,thick] (\w*0.5,\h*5-\r)--(\w*0.5,\h*3.5+\r);
        \draw[black,thick] (\w*1.5,\h*5-\r)--(\w*1.5,\h*3.5+\r);
        \draw[black,thick] (\w*3.5,\h*5-\r)--(\w*3.5,\h*3.5+\r);
        \draw[black,thick] (\w*6.5,\h*5-\r)--(\w*6.5,\h*4.5)--(\w*5.5,\h*4)--(\w*5.5,\h*3.5+\r);
        \draw[black,thick] (\w*5.5,\h*5-\r)--(\w*5.5,\h*4.5)--(\w*6.5,\h*4)--(\w*6.5,\h*3.5+\r);
        \draw[black,thick] (\w*9.5,\h*5-\r)--(\w*9.5,\h*3.5+\r);
        \draw[black,thick] (\w*0.5,\h*8-\r)--(\w*0.5,\h*6.5+\r);
        \draw[black,thick] (\w*3.5,\h*8-\r)--(\w*3.5,\h*7.5)--(\w*4.5,\h*7)--(\w*4.5,\h*6.5+\r);
        \draw[black,thick] (\w*4.5,\h*8-\r)--(\w*4.5,\h*7.5)--(\w*5.5,\h*7)--(\w*5.5,\h*6.5+\r);
        \draw[black,thick] (\w*5.5,\h*8-\r)--(\w*5.5,\h*7.5)--(\w*3.5,\h*7)--(\w*3.5,\h*6.5+\r);
        \draw[black,thick] (\w*8.5,\h*8-\r)--(\w*8.5,\h*6.5+\r);
        \draw[black,thick] (\w*2.5,\h*11-\r)--(\w*2.5,\h*9.5+\r);
        \draw[black,thick] (\w*7.5,\h*11-\r)--(\w*7.5,\h*9.5+\r);
        \draw[black,thick] (\w*3.5,\h*14-\r)--(\w*3.5,\h*13.5)--(\w*4.5,\h*13)--(\w*4.5,\h*12.5+\r);
        \draw[black,thick] (\w*4.5,\h*14-\r)--(\w*4.5,\h*13.5)--(\w*3.5,\h*13)--(\w*3.5,\h*12.5+\r);

        \draw[blue,thick] (\w*1.5,\h*3.5-\r)--(\w*1.5,\h*2+\r);
        \draw[blue,thick] (\w*3.5,\h*3.5-\r)--(\w*3.5,\h*3.1)-- (0,2.7) --(\w*3.5,\h*2.4)--(\w*3.5,\h*2+\r);
        \draw[blue,thick] (\w*0.5,\h*3.5-\r)--(\w*0.5,\h*3.1)-- (0,2.9) --(\w*5.5,\h*2.4)--(\w*5.5,\h*2+\r);
        \draw[blue,thick] (\w*5.5,\h*3.5-\r)--(\w*5.5,\h*3)--(\w*2.5,\h*2.55)--(\w*2.5,\h*2+\r);
        \draw[blue,thick] (\w*6.5,\h*3.5-\r)--(\w*6.5,\h*3)--(\w*9.5,\h*2.55)--(\w*9.5,\h*2+\r);
        \draw[blue,thick] (\w*9.5,\h*3.5-\r)--(\w*9.5,\h*3.1)-- (10,2.9) --(\w*8.5,\h*2.4)--(\w*8.5,\h*2+\r);

        \draw[blue,thick] (\w*0.5,\h*6.5-\r)--(\w*0.5,\h*6)--(\w*1.5,\h*5.5)--(\w*1.5,\h*5+\r);
        \draw[blue,thick] (\w*3.5,\h*6.5-\r)--(\w*3.5,\h*5+\r);
        \draw[blue,thick] (\w*4.5,\h*6.5-\r)--(\w*4.5,\h*6)--(\w*5.5,\h*5.5)--(\w*5.5,\h*5+\r);
        \draw[blue,thick] (\w*5.5,\h*6.5-\r)--(\w*5.5,\h*6)--(\w*9.5,\h*5.5)--(\w*9.5,\h*5+\r);
        \draw[blue,thick] (\w*8.5,\h*6.5-\r)--(\w*8.5,\h*6)--(\w*6.5,\h*5.5)--(\w*6.5,\h*5+\r);

        \draw[blue,thick] (\w*2.5,\h*9.5-\r)--(\w*2.5,\h*9)--(\w*3.5,\h*8.5)--(\w*3.5,\h*8+\r);
        \draw[blue,thick] (\w*7.5,\h*9.5-\r)--(\w*7.5,\h*9)-- (10,8.75) --(\w*8.5,\h*8.5)--(\w*8.5,\h*8+\r);

        \draw[blue,thick] (\w*4.5,\h*12.5-\r)--(\w*4.5,\h*12)--(\w*7.5,\h*11.5)--(\w*7.5,\h*11+\r);
        \draw[blue,thick] (\w*3.5,\h*12.5-\r)--(\w*3.5,\h*12)-- (0,11.75) --(\w*2.5,\h*11.5)--(\w*2.5,\h*11+\r);
        \end{tikzpicture}
        }
        \caption{Pairing procedure for the multiline queue $M\in\MLQC(\lambda,10)$ for $\lambda = (5^23^32^11^2)$. Pairings within type $C$ rows (between rows $r'$ and $r$) are shown in black, while pairings between a pair of type $C$ rows are shown in blue.}
        \label{fig:TypeC-MLQ-pairing}
    \end{figure}
    
\end{example}

The queueing procedure defining the map $\projC$ has a corresponding corner transfer matrix description analogous to \cref{def:ctm_typeA} for $\proj_{\FM}$ that comes from the action of the combinatorial $R$ in type $C$. From the equivalence of KN columns of type $C$ and type $C$ rows, we state the definition of the corner transfer matrix in terms of type $C$ rows identified with their corresponding to KN columns.

\begin{definition}[Corner transfer matrix in type $C$]
    \label{def:ctm-type-C}
     Recall that if $R(T\otimes S)=T'\otimes S'$, we denote the first factor of the output as $R_{(1)}(T\otimes S)\coloneq T'$. Recall also that the bottom component of a type $C$ row $Q$ is denoted $Q_B$, and that $\iota\colon\{\,\cir\,,\,\squ\,,\Cdot\}^* \to \{1,-1,0\}^*$ is the map sending a particle configuration to a signed indicator vector. Let  $\Qb=Q_1\otimes\cdots\otimes Q_L\in\RowC(\alpha,n)$ be a tensor product of type $C$ rows.      Set $P^{(1)}=Q_1$ and for $2\leq j\leq L-1$, define:
    \[P^{(j)}=R_{(1)}(Q_1\otimes R_{(1)}(Q_2\otimes \cdots(R_{(1)}(Q_{j-1}\otimes Q_j)\cdots)))\,.\]
    Then the corner transfer matrix projection is,defined by
    \[\pi(\Qb)=\sum_{j=1}^{L}\pi_j(\Qb),\qquad \pi_j(\Qb) = \iota\left (P^{(j)}_B \right).\]
\end{definition}

\subsection{The crystal Markov chain on $\MLQ_C(\lambda,n)$}

Remarkably, our strategy in \cref{sec:type $A$} naturally extends to the type $C$ case, as well: we define a Markov chain on $\MLQC(\lambda,n)$ whose dynamics are given by powers of crystal operators, after identifying the multiline queue $M\in\MLQ_C(\lambda,n)$ with its corresponding tensor product on type $C$ rows. Hence the definitions of the crystal operators $\widetilde{f}_i$ and $\widetilde{e}_i$ and the string lengths $\varphi_i$ and $\varepsilon_i$ extend to $M\in\MLQC(\lambda,n)$. 

We will show that this Markov chain has uniform stationary distribution and projects to $\tasepC$ via the map $\projC$.

\begin{definition}
     For a multiline queue $M\in\MLQC(\lambda,n)$, define
     \[F(M)=\{\widetilde{f}_i^{j}\colon1\leq j\leq \varphi_i(M),\ 0\leq i\leq n\}\qquad\text{and}\qquad E(M)=\{\widetilde{e}_i^{j}\colon1\leq j\leq \varepsilon_i(M),\ 0\leq i\leq n\}.\]
\end{definition}

\begin{definition}
Define the \emph{type $C$ crystal Markov chain} $\mlqC(\lambda,n)$ to be the process on the states $\MLQC(\lambda,n)$ with dynamics given by the following set of transitions, all occurring with uniform rate: 
\begin{center}
    $M\longmapsto \widetilde{f}_i^{j}M$ for $\widetilde{f}_i^{j}\in F(M)$.
\end{center}
\end{definition}

\begin{theorem}
    \label{thm:balancing-condition-typeC}
     The Markov chain $\mlqC(\lambda,n)$ is irreducible and has uniform stationary distribution.
\end{theorem}
\begin{proof}
    Irreducibility follows from the same argument as \cref{lem:irreducible A} for the type $A$ case after replacing $e_i$s by $\widetilde{f}_i$s and the lowest weight element by the highest weight element with weight $\lambda$.
    
    Define the time-reversal process of $\mlqC(\lambda,n)$ given by the set of uniform transitions
\[M\to \widetilde{e}_i^{j} M,\qquad \widetilde e_i^j\in E(M).\]
As with the proof of \cref{thm:crystal ops MC A}, we must show the balance condition holds for each $M\in\MLQC(\lambda,n)$, namely that $|F(M)|=|E(M)|$. We compute the difference
\[|F(M)|-|E(M)|=\sum_{i=0}^n \varphi_i(M)-\varepsilon_i(M)\]
as follows.  Fix $1\leq i\leq n-1$ and consider 
the configuration in columns $i,i+1$ in a pair of adjacent rows $(b_{k},b_{k}')$ for some $1\leq k\leq L$. From \cref{def:par_i_typeC_rows}, we can interpret the differences of the sum on the right hand side as follows. We say that a $\cir$ contributes $\frac{1}{2}$, a $\squ$ contributes $-\frac{1}{2}$, and a vacancy contributes 0, and then the total contribution to $\varphi_i(M)-\varepsilon_i(M)$ from this pair of rows is precisely the weight of column $i$ minus the weight of column $i+1$. For instance, if the configuration is $\begin{pmatrix}\squ&\cir\\\cir&\squ\end{pmatrix}$, the contribution from column $i$ is $\frac{1}{2}-\frac{1}{2}$, while that of column $i+1$ is $-\frac{1}{2}+\frac{1}{2}$, yielding a difference of 0. On the other hand, if the configuration is $\begin{pmatrix}\cir&\squ\\\cir&\squ\end{pmatrix}$, the contribution from column $i$ is $\frac{1}{2}+\frac{1}{2}$, while that of column $i+1$ is $-\frac{1}{2}-\frac{1}{2}$, yielding a difference of 2: indeed, $\varphi_i-\varepsilon_i=2$ for this configuration. 

Denote by $\omega(C)$ the weight of a column $C$ of $M$ computed by summing the contributions of $\cir$ and $\squ$ over the column as explained before. The previous discussion implies that summing over all rows of $M$, we obtain
\[\varphi_i(M)-\varepsilon_i(M)=\omega(\text{column $i$ of $M$})-\omega(\text{column $i+1$ of $M$})\]
for $1\leq i\leq n-1$. Similarly, for $i=0$ and $i=n$, we have
\begin{align*}\varphi_0(M)-\varepsilon_0(M)&=-\omega(\text{column $1$ of $M$})\;,\\
\varphi_n(M)-\varepsilon_n(M)&=\omega(\text{column $n$ of $M$}).
\end{align*}
Summing these differences over $0\leq i\leq n$, all contributions cancel, and we obtain the desired equality.
\end{proof}

\begin{remark}
    The proofs of \cref{thm:balancing-condition-typeA,thm:balancing-condition-typeC} can be done using the structure of the simple roots of the affine root systems $A_{n-1}^{(1)}$ and $C_n^{(1)}$. Recall that for an element $x\in B$ in the root system $\mathfrak{g}\in\{A_{n-1}^{(1)},C_n^{(1)}\}$, we have
    \[\varphi_i(x)-\varepsilon_i(x)=\langle\wt(x),\alpha_i^\vee\rangle=\begin{cases}\wt(x)_i-\wt(x)_{i+1}& \text{if } 1\leq i\leq n-1\\
   \wt(x)_n-\wt(x)_1&\ \text{if }i=0 \text{ and }\mathfrak{g}=A_{n-1}^{(1)}\\
   \wt(x)_n &\ \text{if }i=n \text{ and }\mathfrak{g}=C_{n}^{(1)}\\
    -\wt(x)_1&\ \text{if }i=0 \text{ and }\mathfrak{g}=C_{n}^{(1)}
    \end{cases}
    \]
    This matches our expression, since the $i$th component $\wt(x)_i$ of the weight of a multiline queue $x\in\MLQC(\lambda,n)$ is precisely its column content $\omega(\text{column $i$ of $x$})$. Thus the telescoping cancellation in the proof above is exactly the statement that in a level-zero realization of the affine roots of a KR crystal,
    \[
    \sum_{i\in I}\langle \mu,\alpha_i^\vee\rangle=\sum_{i\in I}\alpha_i^\vee = 0,\qquad \forall \mu\in\Lambda.
    \]
\end{remark}

\begin{theorem}\label{thm:lumping}
    $\mlqC(\lambda,n)$ projects to $\tasepC(\lambda,n)$ via the map $\projC$.
\end{theorem}
\begin{proof}
    The proof is analogous to that of \cref{thm:crystal ops MC A}. For $v=(v_1,\ldots,v_n)\in\Ac_C^n$, recall that the level-zero action of the type $C$ Weyl group is given by adjacent entry swaps $s_i$ for $1\leq i \leq n-1$ and by
    \begin{align*}s_0\cdot v&=(-v_1,v_2,\ldots,v_n)\\
    s_n\cdot v&=(v_1,\ldots,v_{n-1},-v_n).
    \end{align*}
    It suffices to prove the following properties for $M\in\MLQ_C(\lambda,n)$. 
    
    \begin{enumerate}
        \item[(A)] If $j<\varphi_i(M)$, then for $0\leq i\leq n$,  $\projC\left(\widetilde{f}_i^j M\right)=\projC(M)$.
        \item[(B)] If $j=\varphi_i(M)$, then
         \begin{itemize}
            \item For $1\leq i \leq n-1$,  $\projC\left(\widetilde{f}_i^j M\right)=s_i\cdot\projC(M)$ if and only if $\projC(M)_i>\projC(M)_{i+1}$.
            \item If $i=0$, then $\projC\left(\widetilde{f_0}^j M\right)=s_0 \cdot\projC(M)$ if and only if $\projC(M)_1<0$.
           \item If $i=n$, then $\projC\left(\widetilde{f_n}^j M\right)=s_n\cdot \projC(M)$ if and only if $\projC(M)_n>0$.
            \end{itemize}
    \end{enumerate}\vspace{1em}
    One could use \cref{def:ctm-type-C} to invoke almost exactly the same arguments as in the proof of \cref{thm:crystal ops MC A}. However, since we do not prove that the corner transfer matrix in type $C$ coincides with the projection obtained from the queueing procedure in this article, we will keep our proofs self-contained and use arguments analogous to \cite[Lemmas 3.2, 3.6, and 3.10]{MNS26}. 
    
    Let $\mathbf{L}=(U_1,U_2,\ldots,U_{2L})$ be the labels on the rows of $M=(\bb_1,\bb_2,\ldots,\bb_{2L})$ according to the queueing procedure. For each row $\bb_k$, define $\bb_k^{(\ell)}$ by $(\bb_k^{(\ell)})_i=(\bb_k)_i$ if $(U_k)_i\geq \ell$, and $(\bb_k^{(\ell)})_i=0$ otherwise. That is, $\bb_k^{(\ell)}$ is a subword of $\bb_k$ consisting of particles whose labels in $U_k$ are greater than or equal to $\ell$. Define $u_k^{(\ell)}\coloneq \iota(\bb_k^{(\ell)})$, additionally setting the convention $(u_k^{(\ell)})_0=(u_k^{(\ell)})_{n+1}=0$. By construction, we have $U_k=|u_k^{(1)}|+\cdots+|u_k^{(L)}|$.

    For example, $\bb_3$, which corresponds to row $2$ in the multiline queue $M$ in \cref{fig:TypeC-MLQ-pairing}, has the labeling $U_3=(2,3,0,3,0,5,5,0,0,3)$, which decomposes as $U_3=|u_3^{(1)}|+\cdots+|u_3^{(5)}|$, where
    \begin{align*}\bb_3^{(1)}=\bb_3^{(1)}&=(\cir,\squ,\Cdot,\cir,\Cdot,\cir,\squ,\Cdot,\Cdot,\squ)&\implies \quad u_3^{(1)}=\;&u_3^{(2)}=(1,\bar 1,0,1,0,1,\bar 1,0,0,\bar 1)\;,\\ \bb_3^{(3)}=\bb_3^{(4)}&=(\Cdot,\squ,\Cdot,\cir,\Cdot,\cir,\squ,\Cdot,\Cdot,\squ)&\implies\quad  u_3^{(3)}=\;&u_3^{(4)}=(0,\bar 1,0,1,0,1,\bar 1,0,0,\bar 1)\;,\\
    \bb_3^{(5)}&=(\Cdot,\Cdot,\Cdot,\Cdot,\Cdot,\cir,\squ,\Cdot,\Cdot,\Cdot)&\implies \qquad \qquad\;& u_3^{(5)}=(0,0,0,0,0,1,\bar 1,0,0,0)\;.
    \end{align*}
    We also extend the definition of $s_i$ to configurations $\bb\in\{\cir,\squ,\Cdot\}^n$, by defining $s_i \bb\coloneq \iota^{-1}(s_i\cdot\iota(\bb))$.
    
    Fix $\ell\leq L$ and $0\leq i\leq n$. We have that $\bb_k^{(\ell)}$ is the output of the sequence of type-C pairings 
    \[\bb_k^{(\ell)} = \cPair(\bb_k,\cPair(\bb_{k+1},\cdots,\cPair(\bb_{2\ell-1},\bb_{2\ell})\cdots)) = \cPair\left(\bb_k,\bb_{k+1}^{(\ell)}\right).\]
    Define the analogous set of 
    vectors $\{v_k^{(\ell)}\}$ for the multiline queue $\widetilde{f}_i^p M = (\bc_1,\bc_2,\ldots,\bc_{2L})$ with corresponding particle configurations $\{\bc_k^{(\ell)}\}$ defined by $\bc_k^{(\ell)}\coloneq \iota^{-1}(v_k^{(\ell)})$. Then $\bc_j\in\{\bb_j,s_i\bb_j\}$, depending on how $\widetilde{f}_i^p$ acts on the factor containing $\bb_j$ in $M$. We suppose by induction, with the simply checked base case $j=2\ell$, that $\bc_j^{(\ell)}\in\{\bb_j^{(\ell)},s_i\bb_j^{(\ell)}\}$ for all $j>k$. Then $\bc_k^{(\ell)}=\cPair(\bc_k,\bc_{k+1}^{(\ell)})$ has the following three possibilities: 
    \begin{enumerate}[leftmargin=2cm]
        \item[(Case 1).] $\bc_k^{(\ell)} = \cPair\left( s_i\bb_k, \bb_{k+1}^{(\ell)}\right).$
        \item[(Case 2).] $\bc_k^{(\ell)} = \cPair\left( s_i\bb_k, s_i\bb_{k+1}^{(\ell)}\right).$
        \item[(Case 3).] $\bc_k^{(\ell)} = \cPair\left( \bb_k, s_i\bb_{k+1}^{(\ell)}\right).$
    \end{enumerate}
    We have that $\bc_k\neq \bb_k$ only if one of the following cases occurs:
    \begin{itemize}
        \item[(i)] $(\bb_k)_i=\cir$ and $((\bb_k^{(\ell)})_i,(\bb_k^{(\ell)})_{i+1})=(\cir,\Cdot)$, so there is a paired $\cir$ at site $i$ in $\cPair(\bb_k,\bb_{k+1}^{(\ell)})$.
        \item[(ii)] $(\bb_k)_{i+1}=\squ$ and $((\bb_k^{(\ell)})_i,(\bb_k^{(\ell)})_{i+1})=(\Cdot,\squ)$, so there is a paired $\squ$ at site $i+1$ in $\cPair(\bb_k,\bb_{k+1}^{(\ell)})$.
        \item[(iii)] $((\bb_k)_i,(\bb_k)_{i+1})=(\cir,\squ)$ and 
        $((\bb_k^{(\ell)})_i,(\bb_k^{(\ell)})_{i+1})\in\Big\{(\cir,\Cdot),(\Cdot,\squ),(\cir,\squ)\Big\}$, so one or both of the particles at sites $i,i+1$ are paired in $\cPair(\bb_k,\bb_{k+1}^{(\ell)})$.
    \end{itemize}
    Note that the boundary cases $i\in\{0,n\}$ are captured by the convention $(\bb_k)_0=(\bb_k)_n=\Cdot$ in case (ii) or case (i), respectively.
    
    Now we compare $\bb_k^{(\ell)}$ and $\bc_k^{(\ell)}$ in each of the three cases, assuming $(\bb_k^{(\ell)})_i\neq (\bb_k^{(\ell)})_{i+1}$.
    We first analyze (Case 1). Since the operators $\widetilde{f}_i$ act on rows from top to bottom,
    \[((\bb_k)_i,(\bb_k)_{i+1})\in \Big\{(\cir,\Cdot),(\Cdot,\squ),(\cir,\squ)\Big\}\qquad\text{and}\qquad  ((\bc_k)_i,(\bc_k)_{i+1}) \in \Big\{(\Cdot,\cir),(\squ,\Cdot),(\squ,\cir)\Big\},\]
while both $((\bb_{k+1})_i,(\bb_{k+1})_{i+1})$ and $((\bb_{k+1}^{(\ell)})_i,(\bb_{k+1}^{(\ell)})_{i+1})$ belong to the set
\[\Big\{(\cir,\cir),(\squ,\squ),(\Cdot,\Cdot),(\Cdot,\cir),(\squ,\Cdot),(\squ,\cir)\Big\}.\]
If $(\bb_k^{(\ell)})_i=\cir$, for any case of $\bb_{k+1}^{(\ell)}$ above, the corresponding particle in $s_i\bb_k$ remains paired in $\cPair(s_i\bb_k,\bb_{k+1}^{(\ell)})$, since $\cir$s are paired weakly to the left. The same holds for the corresponding particle if $(\bb_k^{(\ell)})_{i+1}=\squ$, since $\squ$s are paired weakly to the right. Moreover, for the same reason, all particles outside of sites $i,i+1$ maintain their pairing status. Thus in all cases, $\bc_k^{(\ell)}=s_i\bb_k^{(\ell)}$.

Now we analyze (Case 2), where we assume $(\bb_k)_i>(\bb_k)_{i+1}$ and $(\bb_{k+1}^{(\ell)})_i>(\bb_{k+1}^{(\ell)})_{i+1}$, and $(\bb_{k+1})_i\geq (\bb_{k+1})_{i+1}$. Because of the directions of type $C$ pairing, any pair of paired particles in $\cPair(\bb_k,\bb_{k+1}^{(\ell)})$ necessarily remains paired in $\cPair(s_i\bb_k,s_i\bb_{k+1}^{(\ell)})$, so we again have $\bc_k=s_i\bb_k$.

It remains to analyze the slightly more intricate (Case 3), where $(\bb_{k+1}^{(\ell)})_i>(\bb_{k+1}^{(\ell)})_{i+1}$, and where $(s_i\bb_{k+1}^{(\ell)})_i<(s_i\bb_{k+1}^{(\ell)})_{i+1}$. We consider the possible cases for $\bb_k$, first assuming $(\bb_k^{(\ell)})_i\neq(\bb_k^{(\ell)})_{i+1}$:
\begin{itemize}
    \item If $(\bb_k)_i\neq(\bb_k)_{i+1}$, because $\cir$ pairs weakly to the left and $\squ$ pairs weakly to the right, any particles that are paired in $\cPair(\bb_k,\bb_{k+1}^{(\ell)})$ remain paired in $\cPair(\bb_k,s_i\bb_{k+1}^{(\ell)})$, so $\bc_k^{(\ell)}=\bb_k^{(\ell)}$.
    \item If $(\bb_k)_i=(\bb_k)_{i+1}=\cir$, we must have $((\bb_{k+1}^{(\ell)})_i,(\bb_{k+1}^{(\ell)})_{i+1})=(\cir,x)$ with $x<\cir$ and \linebreak[4] $((\bb_k^{(\ell)})_i,(\bb_k^{(\ell)})_{i+1})=(\cir,\Cdot)$. Thus in this case, $\bc_k^{(\ell)}=s_i\bb_k^{(\ell)}$.
    \item If $(\bb_k)_i=(\bb_k)_{i+1}=\squ$, the analysis is symmetric to the previous case so that $((\bb_{k+1}^{(\ell)})_i,(\bb_{k+1}^{(\ell)})_{i+1})=(x,\squ)$ with $x>\squ$ and $((\bb_k^{(\ell)})_i,(\bb_k^{(\ell)})_{i+1})=(\Cdot,\squ)$, and again $\bc_k^{(\ell)}=s_i\bb_k^{(\ell)}$.
\end{itemize}
Finally, it is straightforward to check that if $(\bb_k^{(\ell)})_i=(\bb_k^{(\ell)})_{i+1}\in\{\cir,\squ\}$, both particles will remain paired in $\cPair(\bb_k,s_i\bb_{k+1}^{(\ell)})$ in all cases.
    
    We summarize the previous analysis as follows. We have $v_k^{(\ell)}\in\{u_k^{(\ell)},s_i\cdot u_k^{(\ell)}\}$. More precisely, 
 (Case 1) and (Case 2) above correspond to (Ia.) and (Case 3) splits into the cases (Ib.) and (II) below: 
    \begin{enumerate}
    \item[(I)] $v_k^{(\ell)}=s_i \cdot u_k^{(\ell)}$ if one of the following occurs:
    \begin{itemize}
    \item[a.] $\widetilde{f}_i^p$ acts by $s_i$ on the sub-row $\bb_k$ within the component containing $\bb_k$ and $(u_k^{(\ell)})_i>(u_k^{(\ell)})_{i+1}.$ 
    \item[b.] $v_{k+1}^{(\ell)}=s_i \cdot u_{k+1}^{(\ell)}$ and $(u_{k+1}^{(\ell)})_i>(u_{k+1}^{(\ell)})_{i+1}$, with 
    \[(\bb_k)_i=(\bb_k)_{i+1}=\begin{cases}\cir& \text{if}\ (u_{k+1}^{(\ell)})_i=1\;,\\
    \squ&\text{if}\ (u_{k+1}^{(\ell)})_{i+1}=\bar{1}.\end{cases}
    \]
    \end{itemize}
    \item[(II)] Otherwise, $v_k^{(\ell)}=u_k^{(\ell)}$.
    \end{enumerate}
    
    Let $k$ be the bottommost row at which $f_i^p$ acts on $M$. 
    Therefore, we have that for $1\leq \ell\leq L$, $v_1^{(\ell)}=s_i\cdot u_1^{(\ell)}$ if and only if $(u_{k}^{(\ell)})_{i}>(u_{k}^{(\ell)})_{i+1}$ and for all $1\leq h\leq k-1$, 
    \begin{equation}\label{eq:bk}
    (\bb_h)_i=(\bb_h)_{i+1}=\begin{cases}(\bb_k)_i=\cir& \text{if}\ (u_{k}^{(\ell)})_i=1,\\
    (\bb_k)_{i+1}=\squ&\text{if}\ (u_{k}^{(\ell)})_{i+1}=\bar{1},\end{cases}
    \end{equation}
    either by (I.a) if $k=1$, or by propagating (I.b) from row $k-1$ down to row 1. 
    The configuration \eqref{eq:bk} on the first $k-1$ rows of $M$ implies that $p=\varphi_i(M)$. Moreover, it follows that $(u_1^{(\ell)})_i>(u_1^{(\ell)})_{i+1}$ by propagating the argument for (Case 3). Now, $\projC(M)=u_1^{(1)}+\cdots+u_L^{(L)}$, and for all $\ell$, $(u_1^{(\ell)})_i=(u_1^{(\ell+1)})_i$ whenever $(u_1^{(\ell+1)})_i\neq 0$. Thus it follows that scenario (I) at row 1 of $M$ is equivalent to the property (B) above for the entire multiline queue $M$.

    On the other hand, if $p<\varphi_i(M)$, by the converse of the above argument, there will be some row $1\leq j<k$ such that $v_j^{(\ell)}=u_j^{(\ell)}$ for all $\ell$, and this equality will propagate down to $v_1^{(\ell)}=u_1^{(\ell)}$.  Here, $\projC(\widetilde{f}_i^p M)=\projC(M)$, obtaining case (A).

    \end{proof}

\begin{example}
    Consider the multiline queue $M\in\MLQC((2,2,1),4)$ with $\projC(M)=(2,1,2,\bar 2)$ and $F(M)=\{f_1,f_2,f_3,f_3^2\}$. The four transitions in the crystal Markov chain from $M$ are shown below, with displaced particles highlighted: 
    \begin{center}
    \scalebox{0.9}
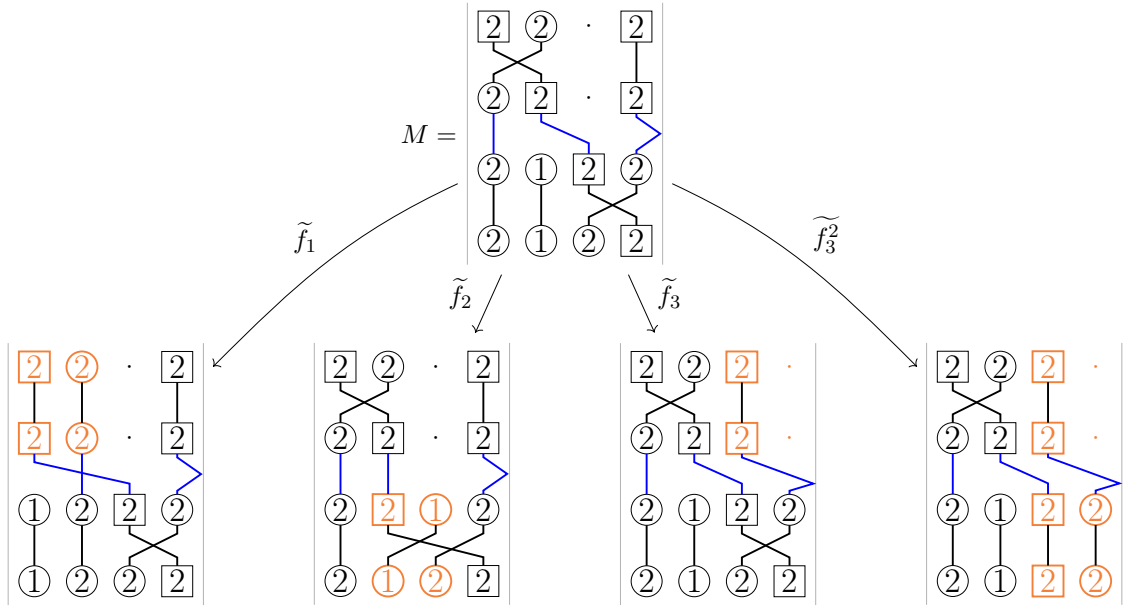
{
        \begin{tikzpicture}
            \def\sca{0.7}
    
            \node (M) at (0,0) {
                \begin{tikzpicture}[scale=\sca]
                    \def \w{1};
                    \def \h{1};
                    \def \r{0.325};
                    \draw[gray!60] (-0.05,0)--(-0.05,5.5);
                    \draw[gray!60] (4.05,0)--(4.05,5.5);
                    \foreach \xx\yy\ll in {
                    0/0/2,1/0/1,2/0/2,
                    0/1.5/2,1/1.5/1,3/1.5/2,
                    0/3/2,
                    1/4.5/2}
                    {
                    \draw (\w*.5+\w*\xx,\h*.5+\h*\yy) circle (\r cm);
                    \node at (\w*.5+\w*\xx,\h*.5+\h*\yy) {\large \ll};
                    }
                    \foreach \xx\yy\ll in { 
                    3/0/2, 2/1.5/2, 1/3/2, 3/3/2, 0/4.5/2, 3/4.5/2
                    }
                    {
                    \node at (\w*.5+\w*\xx,\h*.5+\h*\yy) {$\squa{2*\r*\sca}$};
                    \node at (\w*.5+\w*\xx,\h*.5+\h*\yy) {\large \ll};
                    }
                    \foreach \xx\yy in {
                    2/3,2/4.5
                    }
                    {
                    \node at (\w*.5+\w*\xx,\h*.5+\h*\yy) {$\cdot$};
                    }
                    \draw[black,thick] (\w*0.5,\h*2-\r)--(\w*0.5,\h*.5+\r);
                    \draw[black,thick] (\w*1.5,\h*2-\r)--(\w*1.5,\h*.5+\r);
                    \draw[black,thick] (\w*2.5,\h*2-\r)--(\w*2.5,\h*1.5)--(\w*3.5,\h*1)--(\w*3.5,\h*.5+\r);
                    \draw[black,thick] (\w*3.5,\h*2-\r)--(\w*3.5,\h*1.5)--(\w*2.5,\h*1)--(\w*2.5,\h*.5+\r);
                    \draw[black,thick] (\w*0.5,\h*5-\r)--(\w*0.5,\h*4.5)--(\w*1.5,\h*4)--(\w*1.5,\h*3.5+\r);
                    \draw[black,thick] (\w*1.5,\h*5-\r)--(\w*1.5,\h*4.5)--(\w*0.5,\h*4)--(\w*0.5,\h*3.5+\r);
                    \draw[black,thick] (\w*3.5,\h*5-\r)--(\w*3.5,\h*3.5+\r);
                    \draw[blue, thick] (\w*0.5,\h*3.5-\r)--(\w*0.5,\h*2+\r);
                    \draw[blue, thick] (\w*1.5,\h*3.5-\r)--(\w*1.5,\h*3)--(\w*2.5,\h*2.55)--(\w*2.5,\h*2+\r);
                    \draw[blue, thick] (\w*3.5,\h*3.5-\r)--(\w*3.5,\h*3.1)-- (4,2.75) --(\w*3.5,\h*2.4)--(\w*3.5,\h*2+\r);
                \end{tikzpicture}
            };

            \def\ver{5}
            \def\hor{2.25}
            \node (F1) at (-3*\hor,-\ver) {
                \begin{tikzpicture}[scale=\sca]
                    \def \w{1};
                    \def \h{1};
                    \def \r{0.325};
                    \draw[gray!60] (-0.05,0)--(-0.05,5.5);
                    \draw[gray!60] (4.05,0)--(4.05,5.5);
                    \foreach \xx\yy\ll in {
                    0/0/1,1/0/2,2/0/2,
                    0/1.5/1,1/1.5/2,3/1.5/2
                    }
                    {
                    \draw (\w*.5+\w*\xx,\h*.5+\h*\yy) circle (\r cm);
                    \node at (\w*.5+\w*\xx,\h*.5+\h*\yy) {\large \ll};
                    }
                    \foreach \xx\yy\ll in { 
                    3/0/2, 2/1.5/2, 3/3/2, 3/4.5/2
                    }
                    {
                    \node at (\w*.5+\w*\xx,\h*.5+\h*\yy) {$\squa{2*\r*\sca}$};
                    \node at (\w*.5+\w*\xx,\h*.5+\h*\yy) {\large \ll};
                    }
                    \foreach \xx\yy\ll in {
                    1/3/2,
                    1/4.5/2}
                    {
                    \draw[Orange,thick] (\w*.5+\w*\xx,\h*.5+\h*\yy) circle (\r cm);
                    \node[Orange,thick] at (\w*.5+\w*\xx,\h*.5+\h*\yy) {\large \ll};
                    }
                    \foreach \xx\yy\ll in { 
                    0/3/2, 0/4.5/2
                    }
                    {
                    \node[Orange,thick] at (\w*.5+\w*\xx,\h*.5+\h*\yy) {$\tsqua{2*\r*\sca}$};
                    \node[Orange,thick] at (\w*.5+\w*\xx,\h*.5+\h*\yy) {\large \ll};
                    }    
                    \foreach \xx\yy in {
                    2/3,2/4.5
                    }
                    {
                    \node at (\w*.5+\w*\xx,\h*.5+\h*\yy) {$\cdot$};
                    }
                    \draw[black,thick] (\w*0.5,\h*2-\r)--(\w*0.5,\h*.5+\r);
                    \draw[black,thick] (\w*1.5,\h*2-\r)--(\w*1.5,\h*.5+\r);
                    \draw[black,thick] (\w*2.5,\h*2-\r)--(\w*2.5,\h*1.5)--(\w*3.5,\h*1)--(\w*3.5,\h*.5+\r);
                    \draw[black,thick] (\w*3.5,\h*2-\r)--(\w*3.5,\h*1.5)--(\w*2.5,\h*1)--(\w*2.5,\h*.5+\r);
                    \draw[black,thick] (\w*0.5,\h*5-\r)--(\w*0.5,\h*3.5+\r);
                    \draw[black,thick] (\w*1.5,\h*5-\r)--(\w*1.5,\h*3.5+\r);
                    \draw[black,thick] (\w*3.5,\h*5-\r)--(\w*3.5,\h*3.5+\r);
                    \draw[blue, thick] (\w*1.5,\h*3.5-\r)--(\w*1.5,\h*2+\r);
                    \draw[blue, thick] (\w*0.5,\h*3.5-\r)--(\w*0.5,\h*3)--(\w*2.5,\h*2.55)--(\w*2.5,\h*2+\r);
                    \draw[blue, thick] (\w*3.5,\h*3.5-\r)--(\w*3.5,\h*3.1)-- (4,2.75) --(\w*3.5,\h*2.4)--(\w*3.5,\h*2+\r);
                \end{tikzpicture}
            };

            \node (F2) at (-\hor,-\ver) {
                \begin{tikzpicture}[scale=\sca]
                    \def \w{1};
                    \def \h{1};
                    \def \r{0.325};
                    \draw[gray!60] (-0.05,0)--(-0.05,5.5);
                    \draw[gray!60] (4.05,0)--(4.05,5.5);
                    \foreach \xx\yy\ll in {
                    0/0/2,
                    0/1.5/2,3/1.5/2,
                    0/3/2,
                    1/4.5/2}
                    {
                    \draw (\w*.5+\w*\xx,\h*.5+\h*\yy) circle (\r cm);
                    \node at (\w*.5+\w*\xx,\h*.5+\h*\yy) {\large \ll};
                    }
                    \foreach \xx\yy\ll in { 
                    3/0/2, 1/3/2, 3/3/2, 0/4.5/2, 3/4.5/2
                    }
                    {
                    \node at (\w*.5+\w*\xx,\h*.5+\h*\yy) {$\squa{2*\r*\sca}$};
                    \node at (\w*.5+\w*\xx,\h*.5+\h*\yy) {\large \ll};
                    }
                    \foreach \xx\yy\ll in {
                    1/0/1,2/0/2,
                    2/1.5/1}
                    {
                    \draw[Orange,thick] (\w*.5+\w*\xx,\h*.5+\h*\yy) circle (\r cm);
                    \node[Orange,thick] at (\w*.5+\w*\xx,\h*.5+\h*\yy) {\large \ll};
                    }
                    \foreach \xx\yy\ll in { 
                    1/1.5/2
                    }
                    {
                    \node[Orange,thick] at (\w*.5+\w*\xx,\h*.5+\h*\yy) {$\tsqua{2*\r*\sca}$};
                    \node[Orange,thick] at (\w*.5+\w*\xx,\h*.5+\h*\yy) {\large \ll};
                    }
                    \foreach \xx\yy in {
                    2/3,2/4.5
                    }
                    {
                    \node at (\w*.5+\w*\xx,\h*.5+\h*\yy) {$\cdot$};
                    }
                    \draw[black,thick] (\w*0.5,\h*2-\r)--(\w*0.5,\h*.5+\r);
                    \draw[black,thick] (\w*1.5,\h*2-\r)--(\w*1.5,\h*1.5)--(\w*3.5,\h*1)--(\w*3.5,\h*.5+\r);
                    \draw[black,thick] (\w*2.5,\h*2-\r)--(\w*2.5,\h*1.5)--(\w*1.5,\h*1)--(\w*1.5,\h*.5+\r);
                    \draw[black,thick] (\w*3.5,\h*2-\r)--(\w*3.5,\h*1.5)--(\w*2.5,\h*1)--(\w*2.5,\h*.5+\r);
                    \draw[black,thick] (\w*0.5,\h*5-\r)--(\w*0.5,\h*4.5)--(\w*1.5,\h*4)--(\w*1.5,\h*3.5+\r);
                    \draw[black,thick] (\w*1.5,\h*5-\r)--(\w*1.5,\h*4.5)--(\w*0.5,\h*4)--(\w*0.5,\h*3.5+\r);
                    \draw[black,thick] (\w*3.5,\h*5-\r)--(\w*3.5,\h*3.5+\r);
                    \draw[blue, thick] (\w*0.5,\h*3.5-\r)--(\w*0.5,\h*2+\r);
                    \draw[blue, thick] (\w*1.5,\h*3.5-\r)--(\w*1.5,\h*2+\r);
                    \draw[blue, thick] (\w*3.5,\h*3.5-\r)--(\w*3.5,\h*3.1)-- (4,2.75) --(\w*3.5,\h*2.4)--(\w*3.5,\h*2+\r);
                \end{tikzpicture}
            };

            \node (F3) at (\hor,-\ver) {
                \begin{tikzpicture}[scale=\sca]
                    \def \w{1};
                    \def \h{1};
                    \def \r{0.325};
                    \draw[gray!60] (-0.05,0)--(-0.05,5.5);
                    \draw[gray!60] (4.05,0)--(4.05,5.5);
                    \foreach \xx\yy\ll in {
                    0/0/2,1/0/1,2/0/2,
                    0/1.5/2,1/1.5/1,3/1.5/2,
                    0/3/2,
                    1/4.5/2}
                    {
                    \draw (\w*.5+\w*\xx,\h*.5+\h*\yy) circle (\r cm);
                    \node at (\w*.5+\w*\xx,\h*.5+\h*\yy) {\large \ll};
                    }
                    \foreach \xx\yy\ll in { 
                    3/0/2, 2/1.5/2, 1/3/2, 0/4.5/2
                    }
                    {
                    \node at (\w*.5+\w*\xx,\h*.5+\h*\yy) {$\squa{2*\r*\sca}$};
                    \node at (\w*.5+\w*\xx,\h*.5+\h*\yy) {\large \ll};
                    }
                    \foreach \xx\yy\ll in { 
                    2/3/2, 2/4.5/2
                    }
                    {
                    \node[Orange,thick] at (\w*.5+\w*\xx,\h*.5+\h*\yy) {$\tsqua{2*\r*\sca}$};
                    \node[Orange,thick] at (\w*.5+\w*\xx,\h*.5+\h*\yy) {\large \ll};
                    }
                    \foreach \xx\yy in {
                    3/3,3/4.5
                    }
                    {
                    \node[Orange,thick] at (\w*.5+\w*\xx,\h*.5+\h*\yy) {$\cdot$};
                    }
                    \draw[black,thick] (\w*0.5,\h*2-\r)--(\w*0.5,\h*.5+\r);
                    \draw[black,thick] (\w*1.5,\h*2-\r)--(\w*1.5,\h*.5+\r);
                    \draw[black,thick] (\w*2.5,\h*2-\r)--(\w*2.5,\h*1.5)--(\w*3.5,\h*1)--(\w*3.5,\h*.5+\r);
                    \draw[black,thick] (\w*3.5,\h*2-\r)--(\w*3.5,\h*1.5)--(\w*2.5,\h*1)--(\w*2.5,\h*.5+\r);
                    \draw[black,thick] (\w*0.5,\h*5-\r)--(\w*0.5,\h*4.5)--(\w*1.5,\h*4)--(\w*1.5,\h*3.5+\r);
                    \draw[black,thick] (\w*1.5,\h*5-\r)--(\w*1.5,\h*4.5)--(\w*0.5,\h*4)--(\w*0.5,\h*3.5+\r);
                    \draw[black,thick] (\w*2.5,\h*5-\r)--(\w*2.5,\h*3.5+\r);
                    \draw[blue, thick] (\w*0.5,\h*3.5-\r)--(\w*0.5,\h*2+\r);
                    \draw[blue, thick] (\w*1.5,\h*3.5-\r)--(\w*1.5,\h*3)--(\w*2.5,\h*2.55)--(\w*2.5,\h*2+\r);
                    \draw[blue, thick] (\w*2.5,\h*3.5-\r)--(\w*2.5,\h*3.1)-- (4,2.55) --(\w*3.5,\h*2.4)--(\w*3.5,\h*2+\r);
                \end{tikzpicture}
            };

            \node (F32) at (3*\hor,-\ver) {
                \begin{tikzpicture}[scale=\sca]
                    \def \w{1};
                    \def \h{1};
                    \def \r{0.325};
                    \draw[gray!60] (-0.05,0)--(-0.05,5.5);
                    \draw[gray!60] (4.05,0)--(4.05,5.5);
                    \foreach \xx\yy\ll in {
                    0/0/2,1/0/1,
                    0/1.5/2,1/1.5/1,
                    0/3/2,
                    1/4.5/2}
                    {
                    \draw (\w*.5+\w*\xx,\h*.5+\h*\yy) circle (\r cm);
                    \node at (\w*.5+\w*\xx,\h*.5+\h*\yy) {\large \ll};
                    }
                    \foreach \xx\yy\ll in { 
                    0/4.5/2, 1/3/2
                    }
                    {
                    \node at (\w*.5+\w*\xx,\h*.5+\h*\yy) {$\squa{2*\r*\sca}$};
                    \node at (\w*.5+\w*\xx,\h*.5+\h*\yy) {\large \ll};
                    }
                    \foreach \xx\yy\ll in {
                    3/0/2,
                    3/1.5/2}
                    {
                    \draw[Orange,thick] (\w*.5+\w*\xx,\h*.5+\h*\yy) circle (\r cm);
                    \node[Orange,thick] at (\w*.5+\w*\xx,\h*.5+\h*\yy) {\large \ll};
                    }
                    \foreach \xx\yy\ll in { 
                    2/0/2, 2/1.5/2, 2/3/2, 2/4.5/2
                    }
                    {
                    \node[Orange,thick] at (\w*.5+\w*\xx,\h*.5+\h*\yy) {$\tsqua{2*\r*\sca}$};
                    \node[Orange,thick] at (\w*.5+\w*\xx,\h*.5+\h*\yy) {\large \ll};
                    }
                    \foreach \xx\yy in {
                    3/3,3/4.5
                    }
                    {
                    \node[Orange,thick] at (\w*.5+\w*\xx,\h*.5+\h*\yy) {$\cdot$};
                    }
                    \draw[black,thick] (\w*0.5,\h*2-\r)--(\w*0.5,\h*.5+\r);
                    \draw[black,thick] (\w*1.5,\h*2-\r)--(\w*1.5,\h*.5+\r);
                    \draw[black,thick] (\w*2.5,\h*2-\r)--(\w*2.5,\h*.5+\r);
                    \draw[black,thick] (\w*3.5,\h*2-\r)--(\w*3.5,\h*.5+\r);
                    \draw[black,thick] (\w*0.5,\h*5-\r)--(\w*0.5,\h*4.5)--(\w*1.5,\h*4)--(\w*1.5,\h*3.5+\r);
                    \draw[black,thick] (\w*1.5,\h*5-\r)--(\w*1.5,\h*4.5)--(\w*0.5,\h*4)--(\w*0.5,\h*3.5+\r);
                    \draw[black,thick] (\w*2.5,\h*5-\r)--(\w*2.5,\h*3.5+\r);
                    \draw[blue, thick] (\w*0.5,\h*3.5-\r)--(\w*0.5,\h*2+\r);
                    \draw[blue, thick] (\w*1.5,\h*3.5-\r)--(\w*1.5,\h*3)--(\w*2.5,\h*2.55)--(\w*2.5,\h*2+\r);
                    \draw[blue, thick] (\w*2.5,\h*3.5-\r)--(\w*2.5,\h*3.1)-- (4,2.55) --(\w*3.5,\h*2.4)--(\w*3.5,\h*2+\r);
                \end{tikzpicture}
            };
            \draw[->,out=205,in=45] (M) to node[above left]{$\widetilde{f_1}$} (F1);
            \draw[->] (M) to node[pos=0.75,above left]{$\widetilde{f_2}$} (F2);
            \draw[->] (M) to node[pos=0.75,above right]{$\widetilde{f_3}$} (F3);
            \draw[->,out=335,in=135] (M) to node[above right]{$\widetilde{f_3^2}$} (F32);
            \node at ($(M)-(2,0)$) {$M=$};
        \end{tikzpicture}
        }
    \end{center}
    We give examples of cases (A) and (B) in the proof of \cref{thm:lumping}. Denote $\tau\coloneq \projC(M)$. 
    \begin{itemize}
        \item $\varphi_1(M)=1$ and $\varphi_3(M)=2$, and since $\tau_1>\tau_2$ and $\tau_3>\tau_4$, we have $\projC(\widetilde{f}_1 M)=s_1\cdot\tau$ and $\projC(\widetilde{f}_3^2 M)=s_3\cdot\tau$, illustrating one direction of case (A).
        \item $\varphi_2(M)=1$, but $\tau_2\not> \tau_3$, so $\projC(\widetilde{f}_2 M)=\tau$, illustrating the other direction of case (A).
        \item $\varphi_3(M)=2$, so $\projC(\widetilde{f}_3 M)=\tau$.
    \end{itemize}
\end{example}

\begin{corollary}
    The stationary distribution of $\tasepC(\lambda,n)$ is given by \begin{equation}
        \prob_{\tasepC(\lambda,n)}(a) = \dfrac{1}{|\MLQ_C(\lambda,n)|}\left|\{ M\in\MLQ_C(\lambda,n)\colon\projC(M) = a \}\right|.
    \end{equation}  
\end{corollary}

The Markov chain on multiline queues constructed in \cite{FM07} is defined using \emph{ringing paths}. A ringing path is a sequence of positions determined row by row using only local data (the position at a given row depends only on that of the previous row). A transition is obtained by applying a deterministic local move at each position in the ringing path. This construction is quite different from our crystal Markov chain, where the entire multiline queue needs to be considered to determine the set of possible transitions at a given site. This naturally suggests the following question.

\begin{question}
     Is there an analogue of the ringing path Markov chain on type $C$ multiline queues, in which transitions are governed by local rules in the same sense? More broadly, does there exist a smaller Markov chain on type $C$ multiline queues, with fewer transitions that do not change the projected TASEP state?
\end{question}

\subsection{Markov chain on type $C$ rows and the rank-one TASEP}\label{sec:AT}

To conclude, we compare the rank-one case of our construction $\RowC(k,n)$ (equivalently, $\MLQC(1^k,n)$) to existing combinatorial objects in the literature for this case. In rank-one, the TASEP we study corresponds to the boundary parameter specialization $\alpha=\beta=1$ of the standard two-species TASEP with a fixed number of second-class particles, after identifying the labels $1$ , $0$, and $\bar 1$ with first class particles, second class particles, and holes, respectively. In this case, the crystal Markov chain is particularly simple, and its transitions are given by the operators $\{f_i\}$ and $\{f_i^2\}$ from \cref{def:f_i}. See \cref{fig:111} for an example.

The combinatorics of the open two-species TASEP is well-studied, and stationary distributions can be described through combinatorial objects related to Catalan combinatorics \cite{Mandelshtam2016}, which can be treated as a special case of rhombic alternative tableaux. We focus on the $k=n$ case of the ordinary single species TASEP, which is an illustrative simplification of the general case; the general case can be recovered with minor modifications. In this specialization, the corresponding objects are the Catalan tableaux \cite{Viennot07Catalan} and the Duchi--Schaeffer complete configurations \cite{DuchiSchaeffer05}. Both objects have associated Markov chains on them that project to the TASEP \cite{CW07markov}. 

For the $k=n$ case, we give a bijection to the complete configurations of Duchi and Schaeffer in \cite{DuchiSchaeffer05}. In particular, this bijection allow us to enhance our type $C$ rows with weights in parameters $\alpha$ and $\beta$ to obtain the stationary distribution for the two-species TASEP for general $\alpha$ and $\beta$. 

    \def\sca{0.45}
    \def\hh{8}
    \def\vv{7}
    \begin{figure}[ht]
    \centering
    \begin{tikzpicture}[scale=\sca]
        \node (A) at (0*\hh,-1.4*\vv) {
        \begin{tikzpicture}[scale=\sca]
            \def \w{1};
            \def \h{1};
            \def \r{0.325};
            \draw[gray!60] (-0.05,0)--(-0.05,2.5);
            \draw[gray!60] (3.05,0)--(3.05,2.5);
    
            \foreach \xx\yy in {
            0/0,1/0,2/0,
            0/1.5,1/1.5,2/1.5}
            {
            \draw (\w*.5+\w*\xx,\h*.5+\h*\yy) circle (\r cm);
            \node at (\w*.5+\w*\xx,\h*.5+\h*\yy) {};
            }
    
            \draw[black,thick] (\w*0.5,\h*2-\r)--(\w*0.5,\h*.5+\r);
            \draw[black,thick] (\w*1.5,\h*2-\r)--(\w*1.5,\h*.5+\r);
            \draw[black,thick] (\w*2.5,\h*2-\r)--(\w*2.5,\h*.5+\r);
        \end{tikzpicture}
        };

        \node (B) at (1.5*\hh,-2*\vv) {
        \begin{tikzpicture}[scale=\sca]
            \def \w{1};
            \def \h{1};
            \def \r{0.325};
            \draw[gray!60] (-0.05,0)--(-0.05,2.5);
            \draw[gray!60] (3.05,0)--(3.05,2.5);
    
            \foreach \xx\yy in {
            0/0,1/0,
            0/1.5,1/1.5}
            {
            \draw (\w*.5+\w*\xx,\h*.5+\h*\yy) circle (\r cm);
            \node at (\w*.5+\w*\xx,\h*.5+\h*\yy) {};
            }
            \foreach \xx\yy in { 
            2/0,2/1.5
            }
            {
            \node at (\w*.5+\w*\xx,\h*.5+\h*\yy) {$\squa{2*\r*\sca}$};
            \node at (\w*.5+\w*\xx,\h*.5+\h*\yy) {};
            }
    
            \draw[black,thick] (\w*0.5,\h*2-\r)--(\w*0.5,\h*.5+\r);
            \draw[black,thick] (\w*1.5,\h*2-\r)--(\w*1.5,\h*.5+\r);
            \draw[black,thick] (\w*2.5,\h*2-\r)--(\w*2.5,\h*.5+\r);
        \end{tikzpicture}
        };

        \node (C) at (0.5*\hh,-0.6*\vv) {
        \begin{tikzpicture}[scale=\sca]
            \def \w{1};
            \def \h{1};
            \def \r{0.325};
            \draw[gray!60] (-0.05,0)--(-0.05,2.5);
            \draw[gray!60] (3.05,0)--(3.05,2.5);
    
            \foreach \xx\yy in {
            0/0,1/0,
            0/1.5,2/1.5}
            {
            \draw (\w*.5+\w*\xx,\h*.5+\h*\yy) circle (\r cm);
            \node at (\w*.5+\w*\xx,\h*.5+\h*\yy) {};
            }
            \foreach \xx\yy in {
            2/0,
            1/1.5
            }
            {
            \node at (\w*.5+\w*\xx,\h*.5+\h*\yy) {$\squa{2*\r*\sca}$};
            \node at (\w*.5+\w*\xx,\h*.5+\h*\yy) {};
            }
    
            \draw[black,thick] (\w*0.5,\h*2-\r)--(\w*0.5,\h*.5+\r);
            \draw[black,thick] (\w*2.5,\h*2-\r)--(\w*2.5,\h*1.5)--(\w*1.5,\h*1)--(\w*1.5,\h*.5+\r);
            \draw[black,thick] (\w*1.5,\h*2-\r)--(\w*1.5,\h*1.5)--(\w*2.5,\h*1)--(\w*2.5,\h*.5+\r);
        \end{tikzpicture}
        };

        \node (D) at (0.5*\hh,-2*\vv) {
        \begin{tikzpicture}[scale=\sca]
            \def \w{1};
            \def \h{1};
            \def \r{0.325};
            \draw[gray!60] (-0.05,0)--(-0.05,2.5);
            \draw[gray!60] (3.05,0)--(3.05,2.5);
    
            \foreach \xx\yy in {
            0/0,1/0,
            1/1.5,2/1.5}
            {
            \draw (\w*.5+\w*\xx,\h*.5+\h*\yy) circle (\r cm);
            }
            \foreach \xx\yy in {
            0/1.5,
            2/0
            }
            {
            \node at (\w*.5+\w*\xx,\h*.5+\h*\yy) {$\squa{1.9*\r*\sca}$};
            }
    
            \draw[black,thick] (\w*1.5,\h*2-\r)--(\w*1.5,\h*.5+\r);
            \draw[black,thick] (\w*2.5,\h*2-\r)--(\w*2.5,\h*1.5)--(\w*0.5,\h*1)--(\w*0.5,\h*.5+\r);
            \draw[black,thick] (\w*0.5,\h*2-\r)--(\w*0.5,\h*1.5)--(\w*2.5,\h*1)--(\w*2.5,\h*.5+\r);
        \end{tikzpicture}
        };

        \node (E) at (1.5*\hh,-0*\vv) {
        \begin{tikzpicture}[scale=\sca]
            \def \w{1};
            \def \h{1};
            \def \r{0.325};
            \draw[gray!60] (-0.05,0)--(-0.05,2.5);
            \draw[gray!60] (3.05,0)--(3.05,2.5);
    
            \foreach \xx\yy in {
            0/0,2/0,
            0/1.5,2/1.5}
            {
            \draw (\w*.5+\w*\xx,\h*.5+\h*\yy) circle (\r cm);
            }
            \foreach \xx\yy in {
            1/0,
            1/1.5
            }
            {
            \node at (\w*.5+\w*\xx,\h*.5+\h*\yy) {$\squa{1.9*\r*\sca}$};
            }
            \draw[black,thick] (\w*0.5,\h*2-\r)--(\w*0.5,\h*.5+\r);
            \draw[black,thick] (\w*1.5,\h*2-\r)--(\w*1.5,\h*.5+\r);
            \draw[black,thick] (\w*2.5,\h*2-\r)--(\w*2.5,\h*.5+\r);
        \end{tikzpicture}
        };

        \node (F) at (2.5*\hh,-1*\vv) {
        \begin{tikzpicture}[scale=\sca]
            \def \w{1};
            \def \h{1};
            \def \r{0.325};
            \draw[gray!60] (-0.05,0)--(-0.05,2.5);
            \draw[gray!60] (3.05,0)--(3.05,2.5);
    
            \foreach \xx\yy in {
            0/0,2/0,
            1/1.5,2/1.5}
            {
            \draw (\w*.5+\w*\xx,\h*.5+\h*\yy) circle (\r cm);
            }
            \foreach \xx\yy in {
            1/0,
            0/1.5
            }
            {
            \node at (\w*.5+\w*\xx,\h*.5+\h*\yy) {$\squa{1.9*\r*\sca}$};
            }
    
            \draw[black,thick] (\w*2.5,\h*2-\r)--(\w*2.5,\h*.5+\r);
            \draw[black,thick] (\w*0.5,\h*2-\r)--(\w*0.5,\h*1.5)--(\w*1.5,\h*1)--(\w*1.5,\h*.5+\r);
            \draw[black,thick] (\w*1.5,\h*2-\r)--(\w*1.5,\h*1.5)--(\w*0.5,\h*1)--(\w*0.5,\h*.5+\r);
        \end{tikzpicture}
        };

        \node (G) at (0*\hh,-0*\vv) {
        \begin{tikzpicture}[scale=\sca]
            \def \w{1};
            \def \h{1};
            \def \r{0.325};
            \draw[gray!60] (-0.05,0)--(-0.05,2.5);
            \draw[gray!60] (3.05,0)--(3.05,2.5);
    
            \foreach \xx\yy in {
            1/0,2/0,
            1/1.5,2/1.5}
            {
            \draw (\w*.5+\w*\xx,\h*.5+\h*\yy) circle (\r cm);
            }
            \foreach \xx\yy in {
            0/0,
            0/1.5
            }
            {
            \node at (\w*.5+\w*\xx,\h*.5+\h*\yy) {$\squa{1.9*\r*\sca}$};
            }
            \draw[black,thick] (\w*0.5,\h*2-\r)--(\w*0.5,\h*.5+\r);
            \draw[black,thick] (\w*1.5,\h*2-\r)--(\w*1.5,\h*.5+\r);
            \draw[black,thick] (\w*2.5,\h*2-\r)--(\w*2.5,\h*.5+\r);
        \end{tikzpicture}
        };

        \node (H) at (4*\hh,-1.4*\vv) {
        \begin{tikzpicture}[scale=\sca]
            \def \w{1};
            \def \h{1};
            \def \r{0.325};
            \draw[gray!60] (-0.05,0)--(-0.05,2.5);
            \draw[gray!60] (3.05,0)--(3.05,2.5);
    
            \foreach \xx\yy in {
            0/0,1/0,2/0,
            0/1.5,1/1.5,2/1.5
            }
            {
            \node at (\w*.5+\w*\xx,\h*.5+\h*\yy) {$\squa{1.9*\r*\sca}$};
            }
    
            \draw[black,thick] (\w*0.5,\h*2-\r)--(\w*0.5,\h*.5+\r);
            \draw[black,thick] (\w*1.5,\h*2-\r)--(\w*1.5,\h*.5+\r);
            \draw[black,thick] (\w*2.5,\h*2-\r)--(\w*2.5,\h*.5+\r);
        \end{tikzpicture}
        };

        \node (I) at (4*\hh,-0*\vv) {
        \begin{tikzpicture}[scale=\sca]
            \def \w{1};
            \def \h{1};
            \def \r{0.325};
            \draw[gray!60] (-0.05,0)--(-0.05,2.5);
            \draw[gray!60] (3.05,0)--(3.05,2.5);
    
            \foreach \xx\yy in {
            2/0,
            2/1.5}
            {
            \draw (\w*.5+\w*\xx,\h*.5+\h*\yy) circle (\r cm);
            }
            \foreach \xx\yy in {
            0/0,1/0,
            0/1.5,1/1.5
            }
            {
            \node at (\w*.5+\w*\xx,\h*.5+\h*\yy) {$\squa{1.9*\r*\sca}$};
            }
    
            \draw[black,thick] (\w*0.5,\h*2-\r)--(\w*0.5,\h*.5+\r);
            \draw[black,thick] (\w*1.5,\h*2-\r)--(\w*1.5,\h*.5+\r);
            \draw[black,thick] (\w*2.5,\h*2-\r)--(\w*2.5,\h*.5+\r);
        \end{tikzpicture}
        };

        \node (J) at (1.5*\hh,-1*\vv) {
        \begin{tikzpicture}[scale=\sca]
            \def \w{1};
            \def \h{1};
            \def \r{0.325};
            \draw[gray!60] (-0.05,0)--(-0.05,2.5);
            \draw[gray!60] (3.05,0)--(3.05,2.5);
    
            \foreach \xx\yy in {
            1/0,
            2/1.5}
            {
            \draw (\w*.5+\w*\xx,\h*.5+\h*\yy) circle (\r cm);
            }
            \foreach \xx\yy in {
            0/0,2/0,
            0/1.5,1/1.5
            }
            {
            \node at (\w*.5+\w*\xx,\h*.5+\h*\yy) {$\squa{1.9*\r*\sca}$};
            }
    
            \draw[black,thick] (\w*0.5,\h*2-\r)--(\w*0.5,\h*.5+\r);
            \draw[black,thick] (\w*2.5,\h*2-\r)--(\w*2.5,\h*1.5)--(\w*1.5,\h*1)--(\w*1.5,\h*.5+\r);
            \draw[black,thick] (\w*1.5,\h*2-\r)--(\w*1.5,\h*1.5)--(\w*2.5,\h*1)--(\w*2.5,\h*.5+\r);
        \end{tikzpicture}
        };

        \node (K) at (3.5*\hh,-2*\vv) {
        \begin{tikzpicture}[scale=\sca]
            \def \w{1};
            \def \h{1};
            \def \r{0.325};
            \draw[gray!60] (-0.05,0)--(-0.05,2.5);
            \draw[gray!60] (3.05,0)--(3.05,2.5);
    
            \foreach \xx\yy in {
            0/0,
            2/1.5}
            {
            \draw (\w*.5+\w*\xx,\h*.5+\h*\yy) circle (\r cm);
            }
            \foreach \xx\yy in {
            1/0,2/0,
            0/1.5,1/1.5
            }
            {
            \node at (\w*.5+\w*\xx,\h*.5+\h*\yy) {$\squa{1.9*\r*\sca}$};
            }
    
            \draw[black,thick] (\w*1.5,\h*2-\r)--(\w*1.5,\h*.5+\r);
            \draw[black,thick] (\w*0.5,\h*2-\r)--(\w*0.5,\h*1.5)--(\w*2.5,\h*1)--(\w*2.5,\h*.5+\r);
            \draw[black,thick] (\w*2.5,\h*2-\r)--(\w*2.5,\h*1.5)--(\w*0.5,\h*1)--(\w*0.5,\h*.5+\r);
        \end{tikzpicture}
        };

        \node (L) at (2.5*\hh,-0*\vv) {
        \begin{tikzpicture}[scale=\sca]
            \def \w{1};
            \def \h{1};
            \def \r{0.325};
            \draw[gray!60] (-0.05,0)--(-0.05,2.5);
            \draw[gray!60] (3.05,0)--(3.05,2.5);
    
            \foreach \xx\yy in {
            1/0,
            1/1.5}
            {
            \draw (\w*.5+\w*\xx,\h*.5+\h*\yy) circle (\r cm);
            }
            \foreach \xx\yy in {
            0/0,2/0,
            0/1.5,2/1.5
            }
            {
            \node at (\w*.5+\w*\xx,\h*.5+\h*\yy) {$\squa{1.9*\r*\sca}$};
            }
    
            \draw[black,thick] (\w*0.5,\h*2-\r)--(\w*0.5,\h*.5+\r);
            \draw[black,thick] (\w*1.5,\h*2-\r)--(\w*1.5,\h*.5+\r);
            \draw[black,thick] (\w*2.5,\h*2-\r)--(\w*2.5,\h*.5+\r);
        \end{tikzpicture}
        };

        \node (M) at (3.5*\hh,-0.6*\vv) {
        \begin{tikzpicture}[scale=\sca]
            \def \w{1};
            \def \h{1};
            \def \r{0.325};
            \draw[gray!60] (-0.05,0)--(-0.05,2.5);
            \draw[gray!60] (3.05,0)--(3.05,2.5);
    
            \foreach \xx\yy in {
            0/0,
            1/1.5}
            {
            \draw (\w*.5+\w*\xx,\h*.5+\h*\yy) circle (\r cm);
            }
            \foreach \xx\yy in {
            1/0,2/0,
            0/1.5,2/1.5
            }
            {
            \node at (\w*.5+\w*\xx,\h*.5+\h*\yy) {$\squa{1.9*\r*\sca}$};
            }
    
            \draw[black,thick] (\w*2.5,\h*2-\r)--(\w*2.5,\h*.5+\r);
            \draw[black,thick] (\w*0.5,\h*2-\r)--(\w*0.5,\h*1.5)--(\w*1.5,\h*1)--(\w*1.5,\h*.5+\r);
            \draw[black,thick] (\w*1.5,\h*2-\r)--(\w*1.5,\h*1.5)--(\w*0.5,\h*1)--(\w*0.5,\h*.5+\r);
        \end{tikzpicture}
        };

        \node (N) at (2.5*\hh,-2*\vv) {
        \begin{tikzpicture}[scale=\sca]
            \def \w{1};
            \def \h{1};
            \def \r{0.325};
            \draw[gray!60] (-0.05,0)--(-0.05,2.5);
            \draw[gray!60] (3.05,0)--(3.05,2.5);
    
            \foreach \xx\yy in {
            0/0,
            0/1.5}
            {
            \draw (\w*.5+\w*\xx,\h*.5+\h*\yy) circle (\r cm);
            }
            \foreach \xx\yy in {
            1/0,2/0,
            1/1.5,2/1.5
            }
            {
            \node at (\w*.5+\w*\xx,\h*.5+\h*\yy) {$\squa{1.9*\r*\sca}$};
            }
    
            \draw[black,thick] (\w*0.5,\h*2-\r)--(\w*0.5,\h*.5+\r);
            \draw[black,thick] (\w*1.5,\h*2-\r)--(\w*1.5,\h*.5+\r);
            \draw[black,thick] (\w*2.5,\h*2-\r)--(\w*2.5,\h*.5+\r);
        \end{tikzpicture}
        };

        \draw [ForestGreen,ultra thick,->] (A) to node[above]{\footnotesize $f_3$} (B);

        \draw [NavyBlue,dashed,->,thick,out=130,in=290] (B) to node[left]{\footnotesize $f_2$} (C);
        \draw [NavyBlue,ultra thick,->, out=120, in=240] (B) to node[left]{\footnotesize $f_2^2$} (E);

        \draw [Red,dashed,->,thick] (C) to node[left]{\footnotesize $f_1$} (D);
        \draw [NavyBlue,ultra thick,->,out=60,in=200] (C) to node[pos=0.25,above]{\footnotesize $f_2$} (E);

        \draw [NavyBlue,ultra thick,->] (D) to node[above]{\footnotesize $f_2$} (F);

        \draw [Red,dashed,->,thick,out=290,in=90] (E) to node[pos=0.8,above]{\footnotesize $f_1$} (F);
        \draw [Red,ultra thick,->] (E) to node[above]{\footnotesize $f_1^2$} (G);
        \draw [ForestGreen,ultra thick,->, out=325,in=135] (E) to node[pos=0.875,right]{\footnotesize $f_3$} (N);

        \draw [Red,ultra thick,->] (F) to node[pos=0.6,above]{\footnotesize $f_1$} (G);
        \draw [ForestGreen,ultra thick,->] (F) to node[pos=0.65,above]{\footnotesize $f_3$} (M);

        \draw [Plum,ultra thick,->] (G) to node[left]{\footnotesize $f_0$} (A);
        \draw [ForestGreen,ultra thick,->,out=25,in=155] (G) to node[above]{\footnotesize $f_3$} (L);

        \draw [Plum,ultra thick,->] (H) to node[above]{\footnotesize $f_0$} (N);

        \draw [Plum,ultra thick,->,out=155,in=25] (I) to node[above]{\footnotesize $f_0$} (E);
        \draw [ForestGreen,ultra thick,->] (I) to node[right]{\footnotesize $f_3$} (H);

        \draw [NavyBlue,ultra thick,->] (J) to node[pos=0.6,above]{\footnotesize $f_2$} (I);
        \draw [Plum,ultra thick,->] (J) to node[pos=0.65,above]{\footnotesize $f_0$} (C);

        \draw [Red,ultra thick,->] (K) to node[above]{\footnotesize $f_1$} (J);

        \draw [Plum,ultra thick,->,out=215,in=45] (L) to node[pos=0.875,left]{\footnotesize $f_0$} (B);
        \draw [NavyBlue,ultra thick,->] (L) to node[above]{\footnotesize $f_2^2$} (I);
        \draw [NavyBlue,dashed,->,thick] [->,out=250,in=90] (L) to node[pos=0.8,above]{\footnotesize $f_2$} (J);

        \draw [NavyBlue,dashed,->,thick] (M) to node[right]{\footnotesize $f_2$} (K);
        \draw [Red,ultra thick,->,out=120,in=340] (M) to node[pos=0.25,above]{\footnotesize $f_1$} (L);

        \draw [Red,ultra thick,->,out=60, in=300] (N) to node[right]{\footnotesize $f_1^2$} (L);
        \draw [Red,dashed,->,thick,out=50,in=250] (N) to node[right]{\footnotesize $f_1$} (M);

    \end{tikzpicture}
    \caption{We show the crystal Markov chain $\mlqC(1^3,3)$, which projects to $\tasepC(1^3,3)$. Transitions that are invisible to the projection are shown as dashed arrows.}\label{fig:111}
    \end{figure}
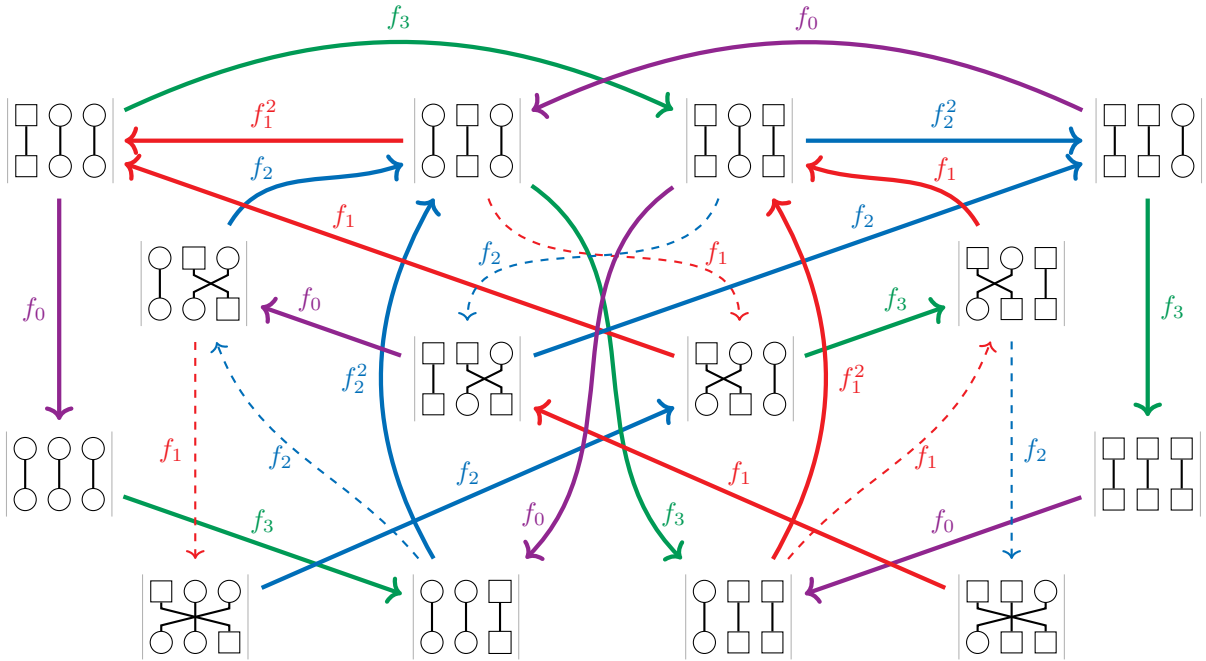

Complete configurations have a resemblance to our type $C$ rows, as they are double row strings in the letters $\{\circ,\bullet\}$, which correspond to our letters $\{\squ,\cir\}$. The convention in \cite{DuchiSchaeffer05} is to associate the TASEP state to the upper row of the configuration via the map $(\circ,\bullet)\mapsto(\bar 1,1)$; this is opposite to our convention of associating the TASEP type to the bottom row of our type $C$ rows. 

 \begin{definition}
     [\cite{DuchiSchaeffer05}] A complete configuration of size $n$ is a biword of length $n$, where each biletter is an element in $\{
     \circ,\bullet\}^2$, satisfying two
conditions:
\begin{itemize}\item \emph{The balance condition}: The biword contains a total of $n$ $\bullet$'s and $n$ $\circ$'s.
\item \emph{The positivity condition}: For $1\leq j\leq n$, the number of $\bullet$'s in the first $j$ biletters is at least $j$. 
\end{itemize}
Denote by $\CC(n)$ the set of complete configurations of size $n$.
 \end{definition}

 We define a map $\V\colon\CC(n)\to\RowC(n,n)$ by mapping the biletters in $(b_1\cdots b_n)\in \CC(n)$ to type $C$ biletters $(c_1,\ldots,c_n)$ where $c_i\in\{\,\squ\,,\,\cir\,\}^2$ as follows. Let $b_i\mapsto c_i$ according to:
 \begin{align}\label{eq:config map}
     \begin{pmatrix}
        \circ\\ \bullet 
     \end{pmatrix}\mapsto\sqsq,\qquad
     \begin{pmatrix}
        \bullet\\ \circ 
     \end{pmatrix}\mapsto\circir,\qquad
    \begin{pmatrix}
        \circ\\ \circ 
     \end{pmatrix}\mapsto \cirsq,\qquad
    \begin{pmatrix}
        \bullet\\ \bullet 
     \end{pmatrix}\mapsto \sqcir\,.
 \end{align}

 \begin{lemma}
     The map $\V\colon\CC(n)\to\RowC(n,n)$ is a bijection. Moreover, the bottom component of $\V(C)$ is equal to the top component of $C\in\CC(n)$ under the equivalence $(\,\squ\,,\,\cir\,)\leftrightarrow(\circ,\bullet)$.
 \end{lemma}

 \begin{proof}
     Let $C\in\CC(n)$ and let $Q=\V(C)$. The balance condition on $C$ ensures that $Q$ has an equal number of $\cir$'s (equivalently, $\squ$'s) in the bottom and top components (equivalently, $\lh_{[1,n]}(Q)=0$). Thus since there are no empty columns, according to \cref{lem:balanced equiv} it suffices to show that  
     \begin{equation}\label{eq:ineq}
     \lh_{[1,j]}(Q)\leq 0\qquad \text{for all}\quad 1\leq j<n.
     \end{equation}
     Observing that $\lh_{[1,j]}(Q)$ is equal to the number of $\circ$'s minus the number of $\bullet$'s in the first $j$ biletters of $C$, we obtain that the desired inequality is exactly the positivity condition on $C$.  

     Conversely, the map in \eqref{eq:config map} is a bijection on biletters, and applying the inverse of this map to $Q\in\RowC(n,n)$ gives a balanced biword, since $\lh_{[1,n]}(Q)=0$, and the inequalities \eqref{eq:ineq} give exactly the
positivity condition on the resulting biword. Hence $\V^{-1}(Q)\in\CC(n)$, so $\V$ is a bijection. 
     
     The final claim follows directly from the definition of $\V$.
 \end{proof}

A \emph{weight} in $\alpha,\beta$ was defined on complete configurations to compute the stationary distribution for the TASEP with boundary parameters $\alpha,\beta$ \cite[Section 3.2]{DuchiSchaeffer05}. We give an equivalent definition for the weight on type $C$ rows $\RowC(n,n)$ after passing through the map $\V$, and noting that our notion of minimal balanced block that is not a singleton coincides with the definition of a ``block'' in \cite{DuchiSchaeffer05}.
\begin{definition}\label{def:ab weight}
    Let $Q=(c_1,\cdots, c_n)\in\RowC(n,n)$. Mark each minimal block equal to $\circir$ with $z$, and mark every minimal block not equal to $\circir$ and with no $z$'s to its left with $y$. Then define
\[\wt(Q)=\alpha^n\beta^n\alpha^{-n_y(Q)}\beta^{-n_z(Q)},\]
where $n_y(Q)$ and $n_z(Q)$ are the numbers of $y$ and $z$ marks in $Q$.
\end{definition}

\begin{example}
    We compute the weight of $Q\in\RowC(n,n)$ for $n=17$ (compare to \cite[Figure 13]{DuchiSchaeffer05}):
    \[Q=\begin{pNiceMatrix}[first-row]
        {\scriptstyle y} &{\scriptstyle y} & {\scriptstyle y} & {\scriptstyle y} & & & & {\scriptstyle y} & & & & & & &{\scriptstyle z} & & \\\squ&\squ&\squ&\squ&\squ&\squ&\cir&\squ&\squ&\cir&\squ&\squ&\cir&\cir&\cir&\squ&\squ\\
        \squ&\squ&\squ&\squ&\cir&\squ&\squ&\squ&\cir&\cir&\squ&\cir&\squ&\squ&\cir&\squ&\squ
        \CodeAfter
        \OverBrace{1-5}{2-7}{\scriptstyle y}
        \OverBrace{1-9}{2-14}{\scriptstyle y}
    \end{pNiceMatrix},\qquad \wt(Q)=(\alpha\beta)^{17}\alpha^{-7}\beta^{-1}.\]
\end{example}

As $\V$ is a weight-preserving bijection, we obtain the following as a corollary of \cite[Theorem 3.3]{DuchiSchaeffer05}. 

\begin{corollary}\label{cor:ab weight}
    The stationary probability of a state $\tau\in\StatesC(1^n,n)$ in $\tasepC(1^n,n)$  with boundary parameters $\alpha$ and $\beta$ is
    \[\frac{1}{Z_{n,n}}\sum_{\substack{Q\in\RowC(n,n)\\Q_B=\tau}}\wt(Q),\qquad\qquad Z_{n,n}=\sum_{Q\in\RowC(n,n)}\wt(Q).\]
\end{corollary}

Duchi and Schaeffer proved the above result by defining a Markov chain (the \emph{DS chain}) on weighted complete configurations that lumps to the open TASEP with boundary parameters $\alpha, \beta$. This chain is different from the one we obtain by transporting the crystal Markov chain $\mlqC(1^n,n)$ through $\V$. For instance, compare \cite[Figure 12]{DuchiSchaeffer05} with \cref{fig:111}, which shows two different Markov chains that project to $\tasepC(1^3,3)$. One notable difference from the crystal Markov chain is that the DS chain has no auxiliary transitions: every nontrivial transition projects to a transition of the TASEP chain. The same is true for the chain on Catalan tableaux derived from \cite{CW07markov}.

\begin{remark}
Duchi--Schaeffer also define complete configurations for a 3-species extension of TASEP  \cite{DuchiSchaeffer05}, consisting of concatenations of complete configurations separated by biletters of the form ${\ast\choose \ast}$. For a fixed number $r$ of such special biletters, these configurations are in bijection with our type $C$ rows $\RowC(n-r,n)$, by extending $\V$ to send the ${\ast\choose \ast}$ biletters to empty columns. The weight in \cref{def:ab weight} can be similarly extended to $\RowC(k,n)$, but we do not pursue this comparison further, since the 3-species TASEP in \cite{DuchiSchaeffer05} has different boundary dynamics from the rank-one open TASEP studied here. Moreover, it is unclear how to extend such a construction to incorporate $\alpha$ and $\beta$ weights in higher-rank type $C$ multiline queues, and finding a formula for the stationary distribution of the rank $L$ open TASEP with general boundary parameters $\alpha,\beta$ remains an open problem.
\end{remark}

\printbibliography

\end{document}